\documentclass[final,10pt]{article}
\usepackage{lineno}
\usepackage{epsfig,amssymb,amsmath,version,amsthm}
\usepackage{amssymb,version,graphicx,fancybox,mathrsfs,multirow}
\usepackage[notcite,notref]{showkeys}
\usepackage{url}
\usepackage{subfigure,epstopdf,bm}
\usepackage{color}
\usepackage[notcite,notref]{showkeys}
\usepackage{stmaryrd}
\usepackage{subfigure,epstopdf,arydshln}
\usepackage{color}
\usepackage{cases}
\usepackage{accents}
\usepackage{appendix}
\usepackage{algorithm}
\usepackage{algorithmicx}
\usepackage{algpseudocode}
\usepackage{amsmath}
\usepackage{amsfonts}
\usepackage{amssymb}
\usepackage{graphicx}%
\modulolinenumbers[5]

\catcode`\@=11 \theoremstyle{plain}
\@addtoreset{equation}{section}

\@addtoreset{figure}{section}
\renewcommand\thefigure{\thesection.\@arabic\c@figure}
\@addtoreset{table}{section}
\renewcommand\thetable{\thesection.\@arabic\c@table}

\floatname{algorithm}{Algorithm}
\newtheorem{thm}{\bf Theorem}

\newtheorem{proposition}{Proposition}[section]
\newenvironment{theorem}{\begin{thm}} {\end{thm}}
\newtheorem{cor}{\bf Corollary}

\newtheorem{lmm}{\bf Lemma}

\newenvironment{lemma}{\begin{lmm}}{\end{lmm}}
\theoremstyle{remark}
\newtheorem{rem}{Remark}[section]

\def \ri {{\rm i}}
\newcommand{\bs}[1]{\boldsymbol{#1}}

\graphicspath{{./Figures/}} {}

\title{An Efficient Iterative Method for Solving Multiple Scattering in Locally Inhomogeneous Media}
\author{Ziqing Xie, Rui Zhang, Bo Wang, Li-lian Wang}
\begin{document}

\maketitle

%
%
%
%
%
%
%

\begin{abstract}
In this paper, an efficient iterative method is proposed for solving multiple scattering problem in locally inhomogeneous media. The key idea is to enclose the inhomogeneity of the media by well separated artificial boundaries and then apply purely outgoing wave decomposition for the scattering field outside the enclosed region. As a result, the original multiple scattering problem can be  decomposed into a finite number of single scattering problems, where each of them communicates with the other scattering problems only through its surrounding artificial boundary. Accordingly,  they can be solved in a parallel manner at each iteration. This framework enjoys a great flexibility in using different combinations of   iterative algorithms and single scattering problem solvers. The spectral element method seamlessly integrated  with the non-reflecting boundary condition and the GMRES iteration is advocated  and implemented in this work. The convergence of the proposed method is proved by using the compactness of involved integral operators. Ample numerical examples are presented to show its high accuracy and efficiency.
\end{abstract}

{\bf key word:} Multiple scattering, inhomogeneous media, iterative method, spectral element method,  non-reflecting boundary condition, GMRES iteration


\section{Introduction}
The acquaintance  of many physical phenomena and engineering processes can be significantly  enhanced
by accurately simulating the multiple scattering problems involving configurations of many obstacles.
Typically, for two-dimensional time-harmonic acoustic multiple scattering in inhomogeneous media,
we consider the Helmholtz equation of the form
\begin{equation}\label{helmholtzeq}
\Delta u(\bs x) + \kappa^2n^2(\bs x) u(\bs x) = 0,  \quad\mathrm{in} \ \ \mathbb{R}^2\setminus \Omega,
\end{equation}
where $u=u^{\rm sc}+u^{\rm in}$ is the total field, $\kappa$ is the wave number, $n^2(\bs x)$ is the index of refraction, $\Omega$ is a region occupied by $M$ impenetrable scatterers in $\mathbb R^2$, see Fig. \ref{scatterersconfigure}. The scattering field $u^{\rm sc}$ satisfies the Sommerfeld radiation condition
\begin{equation}
\frac{\partial u^{\rm sc}}{\partial r}-\ri\kappa u^{\rm sc}={\color{blue}o}\big( r^{-1/2}\big),\quad  \mathrm{as} \ \ r:=|\bm{x}|\rightarrow \infty.
\end{equation}
On the boundaries of the scatterers, the Dirichlet, Neumann or Robin boundary conditions can  be imposed according to different  materials of the scatterers.  Here, let $M_1$, $M_2$, $M_3$ be the number of scatterers with Dirichlet, Neumann and Robin boundary conditions, respectively,  and denote by  $\Omega_{1j}$, $\Omega_{2j}$, $\Omega_{3j}$ the $j$-th scatterer in each group. Accordingly, we  denote
the domain of all obstacles and its boundary by
\begin{equation}\label{OmegaObs}
\Omega=\Omega_1\cup\Omega_2\cup\Omega_3,\quad \partial\Omega=\partial\Omega_1\cup\partial\Omega_2\cup\partial\Omega_3\;\;\;  {\rm with}\;\;\;  \Omega_{i}=\bigcup_{j=1}^{M_i} \Omega_{ij},
\end{equation}
for $i=1,2,3$ corresponding to the Dirichlet, Neumann and Robin boundary conditions, respectively.
For notational convenience, we express these three types of    boundary conditions on the scatterers as
\begin{equation}\label{bconscatterers}
\mathscr{B}_iu=0, \quad \bs x\in\partial\Omega_i,\quad i=1,2,3,
\end{equation}
where
\begin{equation}
\mathscr{B}_1=\mathcal I,\quad \mathscr{B}_2=\frac{\partial}{\partial\bs n},\quad \mathscr{B}_3=\frac{\partial}{\partial\bs n}+h\mathcal I.
\end{equation}
Here, $\mathcal I$ is the identity operator, $\bs n$ is the unit outward normal on $\partial\Omega_2$ and $h$ is a given function defined on $\partial\Omega_3$.
\begin{figure}[htbp]
	\centering
	\subfigure[]{\includegraphics[scale=0.55]{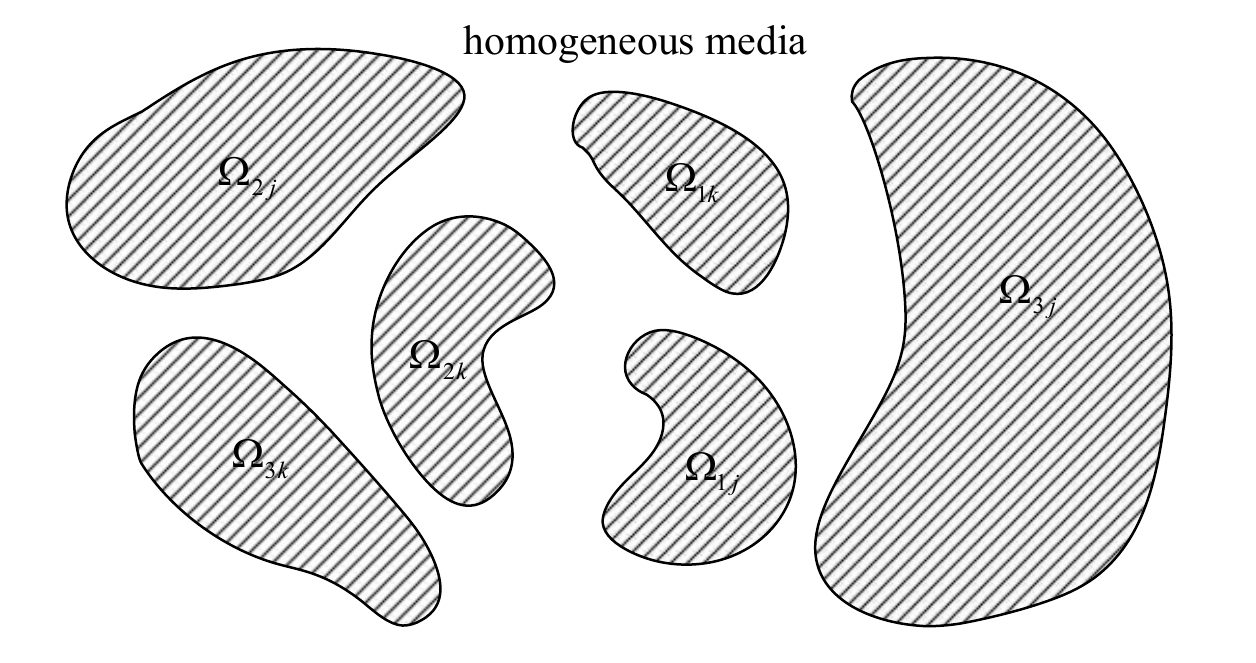}}\quad
	\subfigure[]{\includegraphics[scale=0.5]{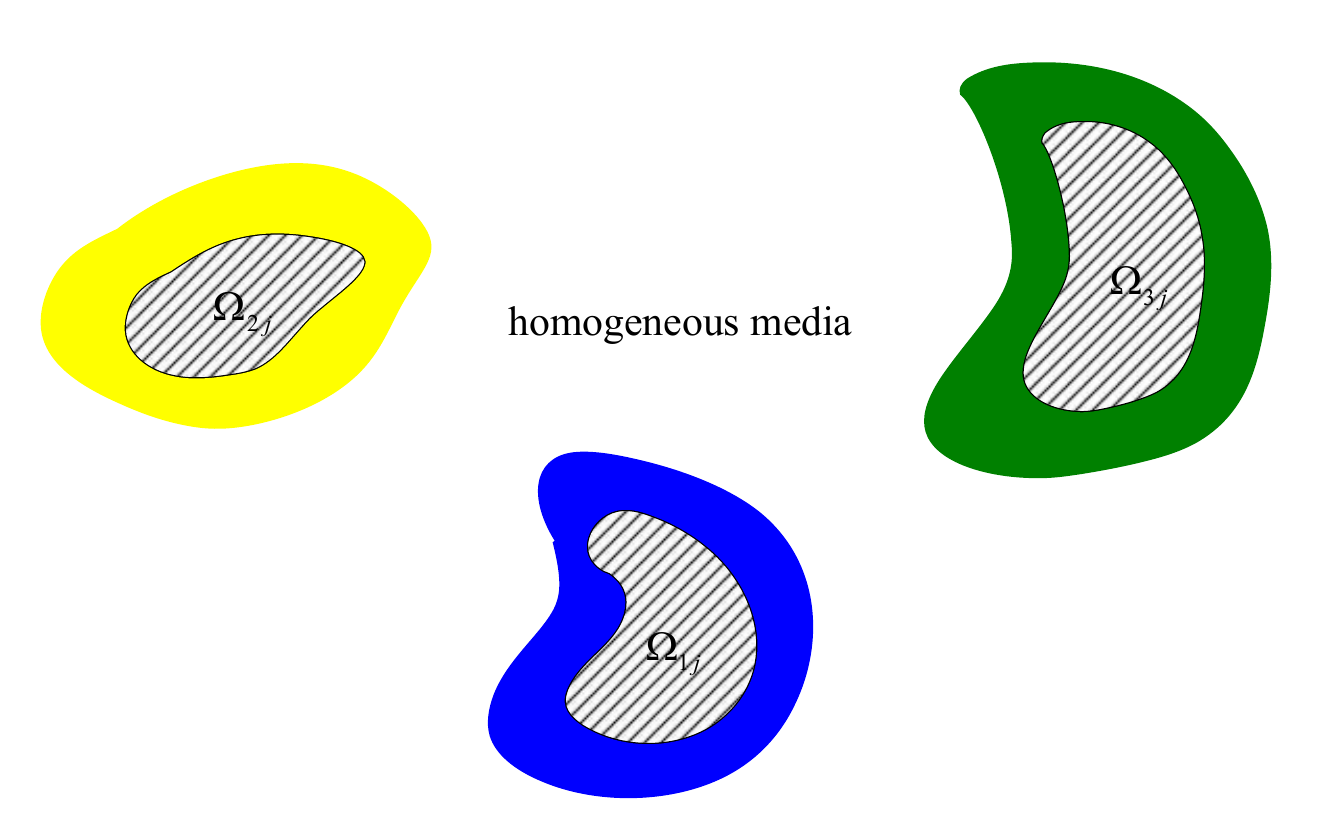}}
	\caption{Configurations of multiple scattering. (a): Scatterers embedded in homogeneous media;   (b): Well-separated scatterers with inhomogeneous media in colored area.}
	\label{scatterersconfigure}
\end{figure}

 It is known that  analytic solutions for wave scattering problems from multiple arbitrary shaped obstacles embedded in inhomogeneous media are not available.  Partially for this reason,
 many early works are mostly restricted to cylindrical and spherical obstacles embedded in homogeneous media, where the modal expansions of the scattered fields play an essential role (cf. \cite{young1975multiple,ragheb1985scattering,elsherbeni1994comparative,gabrielli2001acoustic,amirkulova2015acoustic}).  We highlight that the reader-friendly monograph by  Martin \cite{martin2006multiple} was largely concerned with time-harmonic waves with multiple obstacles  and with exact methods including separation of variables, integral equations and $T$-matrices, but only the last chapter is concerned with some numerics.

  Among limited works for multiple scattering problems  with general bounded  scatterers (compared with intensive studies of single scattering problems), the boundary integral method  is one of the methods of choice. 
By reformulating a scattering problem into an integral equation on the boundary of scatterers (cf. \cite{colton2013integral,kleefeld2012exterior,acosta2015surface}), numerical methods (e.g., boundary element methods) have been developed based on the Galerkin or collocation formulations (cf.  \cite{martin1985integral,ganesh2011efficient}). Very recently, the boundary integral method with fast multipole acceleration and hybrid numerical-asymptotic boundary element method have been investigated for relatively  low frequency (cf. \cite{lai2018fast}) and high frequency problems (cf. \cite{chandler2015high,gibbs2019high}), respectively. However, it is noteworthy that the boundary integral method relies on the Green's function to derive integral equation on the boundary,  which in general is not applicable to inhomogeneous media.

  In analogue with solving general single scattering problems,  one can reduce the unbounded domain by a proper  domain truncation technique, before applying  a finite-domain solver, e.g., the finite element method.
  Grote and Kirsch \cite{grote2004dirichlet} proposed  the  non-reflecting boundary conditions  (NRBC)  based on Dirichlet-to-Neumann operators for truncating  multiple time-harmonic acoustic scattering problems with well-separated scatterers in two dimensions, where each scatterer is surrounded by an NRBC.
  The framework therein  is well-suited for numerical discretization, but it requires to solve coupled systems.
%
Indeed, Acosta and Villamizar \cite{acosta2010coupling, acosta2012dtn} discussed the multiple acoustic scattering from scatterers of complex shape using coupling of Dirichlet-to-Neumann boundary condition and the finite difference method.  
In practice, the iterative method is more desirable.
Based on the decomposition of the scattering wave into purely outgoing waves, Neumann iterative method was proposed in \cite{balabane2004boundary}, where at each iteration, only single scattering problems need to be solved.
This iterative technique has been further developed for high frequency problems with a large number of scatterers, see \cite{ecevit2009analysis, anand2010analysis,ganesh2011efficient,Geuzaine2010An,ganesh2015efficient} and the references therein. Recently, a block Gauss-Seidel iterative method was employed to solve the linear systems resulted from the finite element discretization (cf. \cite{li2013two, jiang2012adaptive,wu2016adaptive}). Error estimates  between the iterative scheme at continuous level  and its finite element discretization was analyzed in \cite{jiang2012adaptive}.   Most of the aforementioned works are for homogeneous media.

In this paper, we propose an efficient iterative method for solving multiple scattering problem in locally inhomogeneous media, and show  the convergence of the method. The algorithm consists of three components. Firstly, the scatterers and inhomogeneity of the media are enclosed by well-separated artificial boundaries such that purely outgoing wave decomposition is applicable outside the enclosed domains.  Then the  scattering field is  decomposed into purely outgoing waves and boundary integral equations with respect to density functions on the artificial boundaries are formulated. Secondly, we change the unknowns in the resulted  boundary integral equations by using solution operators of the interior and exterior problems. New equations using the value of purely outgoing fields and total field on the artificial boundaries as unknowns can be derived. Thirdly, the iterative methodology (e.g., Gauss-Seidel, general minimum residue) is applied.

The proposed method enjoys  the advantage that only the interior problems (together with analytic  formulas for solutions of exterior problems) with respect to single scatterer need to be solved separately  at each iteration.  Various single scattering problem solvers and iterative methods can be applied. Thus, this approach possesses   excellent flexibility and high parallelizability. In this work, the high order spectral element method with non-reflecting boundary condition (NRBC) and GMRES iterative method is adopted. We remark that other numerical PDE solvers (e.g., finite element method or finite difference method) can also be used for solving single scattering problems. Here,  we basically solve a boundary integral equation on the artificial boundaries. The convergence of the method can be proved by using the compactness of involved integral operators. Moreover, the well-conditioning  feature of the boundary integral equation leads to a small number of iterations for convergence. This is the main difference between our method and the Neumann iterative method.  Numerical results show that the number of iterations is nearly independent of the mesh size and polynomial degree used in the discretisation. This paper will focus  on two dimensional scenarios, but the proposed method is extendable to three dimensional cases. 

Note that the final discrete systems resulted from the discretisation proposed by \cite{grote2004dirichlet,acosta2010coupling, acosta2012dtn} have block structure, so the block iterative methods can be directly applied (e.g.,  block Gauss-Seidel iterative method \cite{li2013two, jiang2012adaptive}).  Although  the pursuit of ``decoupling'' between scatterers is similar to these works, the derivation of our iterative algorithm is from the boundary integral equations that leads to the use of 
purely outgoing waves  rather than the whole scattering field on the boundaries of the scatterers (homogeneous media case) or the artificial boundary (inhomogeneous media case) for the communication between scatterers.  Moreover, this allows us to conduct the convergence analysis based on the tools in the boundary integral equations.  Indeed, the numerical comparisons  show that
such a  treatment of the interactions between scatterers is more effective than the existing approaches in particular for a large number of scatterers. This also can relax the assumption of the well separateness of the scatterers when the problems in homogeneous media are tackled.    

The rest of this paper is organized as follows. In section \ref{sect2}, we use the multiple scattering problem in homogeneous media to illustrate the main idea of the iterative method. By using the classic potential theory, the 	boundary integral equations with respect to purely outgoing fields are derived for three typical boundary conditions. Then the iterative algorithm using GMRES iteration is presented. In section \ref{sect3}, the iterative method in locally inhomogeneous media is proposed. We first introduce artificial boundaries to enclose the inhomogeneity of the media and then show that an outgoing wave decomposition can be used to derive equations with respect to outgoing fields and total field on the artificial boundaries. Iterative method together with spectral element discretization is proposed for the coupled equations. In section \ref{sect4}, we give theoretical proof for the convergence of the iterative method by using the compactness of the involved integral operators. Various numerical examples are presented in section \ref{sect5}. By compared with a direct spectral element discretization to the truncation using one sufficiently large artificial boundary to enclose all scatterers inside, we validate the effectiveness of our method. The more efficient communication strategy between scatterers is also validated by numerical comparisons with  the approach  in \cite{grote2004dirichlet}.

\section{Iterative method for multiple scattering in homogeneous media}\label{sect2}
In this section, we focus on the multiple scattering problem in homogeneous media, i.e., $n(\bs x)\equiv 1$ in \eqref{helmholtzeq}, and propose an iterative method based on the boundary integral equations on the scatterers. As we shall see in the next section,  this actually paves the way for the algorithm and analysis of the multiple scattering in locally inhomogeneous media, where the boundary integral  formulations on the artificial boundaries can be seamlessly integrated with the interior solver for each single scatterer.   In particular, due to the circular artificial boundaries, the integral operators related to the single scattering problems \eqref{localscatteringproblemoutside}  can be solved analytically, so the boundary integral operators are only used in the derivation and the convergence analysis.

\subsection{Integral equations on the boundaries of the scatterers}
Given a generic bounded domain $D\subset\mathbb{R}^2$ and a density function $\phi\in L^2(\partial D)$, the corresponding single-layer and double-layer potentials are defined as (cf. \cite{liu1995helmholtz, colton2013integral}):
\begin{equation}\label{layeredpotential}
\mathcal{S}\phi(\bm{x}):= \int_{\partial D} G_{\kappa}(\bm{x},\bm{y})\phi(\bm{y})\,\mathrm{d}\bm{y},\quad \mathcal{D}\phi(\bm{x}):=\int_{\partial D}\frac{\partial G_{\kappa}(\bm{x},\bm{y})}{\partial \bm{n}(\bm{y})}\phi(\bm{y})\,\mathrm{d}S_{\bm{y}},\quad \bm{x}\notin \partial D,
\end{equation}
where \begin{equation}\label{greenfunction}
G_{\kappa}(\bm{x},\bm{y})=
\displaystyle -\frac{\ri}{4}H_0^{(1)}(\kappa|\bm{x}-\bm{y}|),
\end{equation}
is the Green's function in free space. For any function $v(\bs x)$, we distinguish
 its  limit values  obtained by approaching the boundary $\partial D$ from inside $\mathbb{R}^2\setminus\bar{D}$
and $D$, respectively,  by
\begin{equation}
v^+(\bm{x})=\lim_{\substack{\bm{y}\rightarrow \bm{x}\\ \bm{y}\notin D}} v(\bm{y}),\quad v^-(\bm{x})=\lim_{\substack{\bm{y}\rightarrow \bm{x}\\ \bm{y}\in D}} v(\bm{y}),\quad \ \bm{x}\in \partial D.
\end{equation}
Similarly, we denote the limit values of the normal derivative at $\bs x\in\partial D$ from two sides  by
\begin{eqnarray}
{\partial}^{+}_{\bm{n}} v (\bm{x}):=\lim\limits_{\substack{\bm{y}\rightarrow \bm{x}\\ \bm{y}\notin D}} \nabla v(\bm{y})\cdot\bm{n_x}, \quad
\partial^-_{\bm{n}} v (\bm{x}):=\lim\limits_{\substack{\bm{y}\rightarrow \bm{x}\\ \bm{y}\in D}} \nabla u(\bm{y})\cdot\bm{n}_{\bm{x}}, \quad \ \bm{x}\in \partial D.
\end{eqnarray}
Although the Green's function is singular as $\bm x\rightarrow\bm y$, the limits of $\mathcal{S}\phi(\bm{x}), \mathcal{D}\phi(\bm{x})$ as $\bs x$ approach $\partial D$ remain finite and well-defined. Indeed,  in electrostatics, they represent a physical potential and an electric field generated by finite charge or dipole densities. According to the potential theory (cf. \cite{verchota1984layer, colton2013integral}),  they are given by the following Cauchy principle value (p.v.)
and Hadamard finite part (p.f.).
\begin{proposition}\label{theorem1}
	The single layer potential $\mathcal{S}\phi(\bm{x})$ is continuous across the boundary $\partial D$, and
	\begin{equation}
	\mathcal{S}\phi(\bm{x})=\mathrm{p.v.} \int_{\partial D} G_{\kappa}(\bm{x},\bm{y})\phi(\bm{y})\mathrm{d}S_{\bm{y}},\quad \  \bm{x}\in\partial D,
	\end{equation}
	while $\frac{\partial \mathcal{S}\phi(\bm{x})}{\partial \bs n}$ has a jump, namely
	\begin{eqnarray}
	{\partial}^{\pm}_{\bm{n}} \mathcal{S}\phi(\bm{x})=\mathrm{p.v.}\int_{\partial D}\frac{\partial G_{\kappa}(\bm{x},\bm{y})}{\partial \bm{n}(\bm{x})}\phi(\bm{y})\mathrm{d}S_{\bm{y}}\mp \frac{\phi(\bs x)}{2} ,\quad \  \bm{x}\in\partial D,
	\end{eqnarray}
	which implies
	\begin{equation}
	\big\llbracket {\partial}_{\bm{n}} \mathcal{S}\phi(\bm{x})\big\rrbracket:=\partial_{\bm{n}}^+\mathcal{S}\phi(\bm{x})-\partial_{\bm{n}}^-\mathcal{S}\phi(\bm{x})=-\phi(\bs x), \quad \  \bm{x}\in\partial D.
	\end{equation}
\end{proposition}

\begin{proposition}\label{theorem2}
	The double layer potential $\mathcal{D}\phi(\bm{x})$ is discontinuous across $\partial D$, and there holds
	\begin{eqnarray}
	(\mathcal{D}\phi(\bm{x}))^{\pm}=\mathrm{p.v.}\int_{\partial D}\frac{\partial G_{\kappa}(\bm{x},\bm{y})}{\partial \bm{n}(\bm{y})}\phi(\bm{y})\mathrm{d}S_{\bm{y}}\pm \frac{\phi(\bs x)}{2},\quad \  \bm{x}\in\partial D,
	\end{eqnarray}
	and
	\begin{equation}
	\big\llbracket \mathcal{D}\phi(\bm{x})\big\rrbracket:=(\mathcal{D}\phi(\bm{x}))^+-(\mathcal{D}\phi(\bm{x}))^-=\phi(\bs x),\quad \  \bm{x}\in\partial D.
	\end{equation}
	Meanwhile, the normal derivative of the double layer potential is continuous across $\partial D$ and
	\begin{equation}
	\partial_{\bs n} \mathcal{D}\phi(\bm{x})=\mathrm{p.f.}\int_{\partial D}\frac{\partial^2 G_{\kappa}(\bm{x},\bm{y})}{\partial\bm{n}(\bs x)\partial \bm{n}(\bm{y})}\phi(\bm{y})\mathrm{d}S_{\bm{y}},\quad \  \bm{x}\in\partial D.
	\end{equation}
\end{proposition}

The potential theory introduced above can be  directly used to derive boundary integral equations for multiple scattering problem in homogeneous media. However, it will lead to a very large linear system when a large number of scatterers are involved. To overcome this, we formulate the boundary integral equations based on the decomposed form derived from the superposition principle. It is known  that the scattering field $u^{\rm sc}$ of the multiple scattering problem \eqref{helmholtzeq}-\eqref{bconscatterers} in homogeneous media has the following unique decomposition (cf. \cite{Antoine2008On}):
\begin{equation}\label{superposition}
u^{\rm sc}(\bm{x})=\sum_{j=1}^{M_1} w_{1j}(\bm{x})+\sum_{j=1}^{M_2} w_{2j}(\bm{x})+\sum_{j=1}^{M_3} w_{3j}(\bm{x}),
\end{equation}
where $w_{ij}$ are the solutions of the single scattering problems
\begin{subequations}\label{localscatteringproblem}
	\begin{numcases}
	\displaystyle \Delta w_{ij} +\kappa^2 w_{ij} =0,\quad\mathrm{in}\quad \Omega_{ij}^{\infty}:=\mathbb R^2\backslash\bar{\Omega}_{ij},\label{localscatteringproblem1}\\
	\displaystyle \mathscr B_{ij} w_{ij}=g_{ij},\quad \mathrm{on}\quad \partial \Omega_{ij},\label{localscatteringproblem2}\\
	\displaystyle \frac{\partial w_{ij}}{\partial \bm{n}}-\ri kw_{ij}=o\big( r^{-\frac{1}{2}}\big),\quad \mathrm{as}\quad r:=|\bm{x}|\rightarrow \infty,\label{localscatteringproblem3}
	\end{numcases}
\end{subequations}
for  $j=1, \cdots, M_i, i=1, 2, 3$. The input data is given by
\begin{equation}\label{localboundaryvalue}
g_{ij}=-\mathscr B_{ij} u^{\rm in}-\sum_{k=1,k\neq j}^{M_i} \mathscr B_{ij} w_{ik}-\sum_{\ell=1,\ell\neq i}^{3}\sum_{k=1}^{M_{\ell}} \mathscr B_{ij} w_{\ell k}. 
\end{equation}
The last two terms in \eqref{localboundaryvalue} involve the scattering fields from all other scatterers. It is seen that due to the interaction between the scatterers, the incident wave $g_{ij}$ for the $j$-th scatterer in the $i$-th group is the combination of $u^{\rm in}$ and the scattering fields generated by all other scatterers. This shows how  the multiple scattering system is coupled.

According to the potential theory, the single and double-layer potentials
defined by \eqref{layeredpotential} satisfy the Helmholtz equation in $\mathbb{R}^2\backslash\partial D$ and the Sommerfeld radiation condition at infinity. For the uniqueness of the resulted boundary integral equations, we define the mixed potentials (cf. \cite{colton2013integral,kleefeld2012exterior}):
\begin{equation}\label{mixedpotential}
\mathscr K_{1j}\phi_{1j}:=\mathcal D_{1j}\phi_{1j}+\ri\eta \mathcal S_{1j}\phi_{1j},\quad \mathscr K_{ij}\phi_{ij}:=-\ri\eta\mathcal D_{ij}\phi_{ij} -\mathcal S_{ij}\phi_{ij}, \quad i=2, 3,
\end{equation}
for the densities $\{\phi_{1j}, \phi_{2j}, \phi_{3j}\},$ where
$$\mathcal S_{ij}\phi_{ij}:=\int_{\partial \Omega_{ij}} G_{\kappa}(\bm{x},\bm{y})\phi_{ij}(\bm{y})\mathrm{d}\bm{y},\quad \mathcal{D}_{ij}\phi_{ij}(\bm{x}):=\int_{\partial \Omega_{ij}}\frac{\partial G_{\kappa}(\bm{x},\bm{y})}{\partial \bm{n}(\bm{y})}\phi_{ij}(\bm{y})\mathrm{d}S_{\bm{y}},\quad\forall \bs x\notin\partial\Omega_{ij},$$
are single and double layer potentials on $\partial\Omega_{ij}$, and $\eta$ is a given constant satisfying $\eta\mathfrak{Re}\kappa\geq 0$.
Then the solutions of single scattering problems \eqref{localscatteringproblem1}-\eqref{localscatteringproblem3} have the form
\begin{equation}\label{integralrepsolu}
w_{ij}(\bs x)=\mathscr K_{ij}\phi_{ij}(\bs x),\quad \bs x\notin \partial\Omega_{ij},
\end{equation}
while the density functions  $\{{\phi}_{ij}\}$ satisfy boundary integral equations:
\begin{equation}\label{BIElocal}
\frac{1}{2}\phi_{ij}+\widehat{\mathscr K}_{ij}\phi_{ij}=g_{ij},\quad\bs x\in \partial\Omega_{ij}, \;\; j=1, 2, \cdots M_i, \;\;i=1, 2, 3.
\end{equation}
The boundary integral operators $\widehat{\mathscr K}_{ij}$ are defined as
\begin{eqnarray}\label{16}
\widehat{\mathscr{K}}_{1j}:=\widehat{\mathcal D}_{1j}+\ri\eta \widehat{\mathcal S}_{1j},\quad
\widehat{\mathscr{K}}_{2j}:=-\ri\eta\widehat{\mathcal D}_{2j}-\widehat{\mathcal S}_{2j}, \quad
\widehat{\mathscr{K}}_{3j}:=-\ri\eta\widehat{\mathcal D}_{3j}- \widehat{\mathcal S}_{3j}-\frac{\ri}{2}\eta h,
\end{eqnarray}
where
\begin{equation*}
\begin{split}
&\widehat{\mathcal S}_{1j}\phi_{1j}(\bs x):=\mathrm{p.v.} \int_{\partial \Omega_{1j}} G_{\kappa}(\bm{x},\bm{y})\phi_{1j}(\bm{y})\mathrm{d}S_{\bm{y}},\;  \widehat{\mathcal D}_{1j}\phi_{1j}(\bs x):=\mathrm{p.v.}\int_{\partial \Omega_{1j}}\frac{\partial G_{\kappa}(\bm{x},\bm{y})}{\partial \bm{n}(\bm{y})}\phi_{1j}(\bm{y})\mathrm{d}S_{\bm{y}},\\
&\widehat{\mathcal S}_{2j}\phi_{2j}(\bs x):=\mathrm{p.v.} \int_{\partial \Omega_{2j}} \frac{\partial G_{\kappa}(\bm{x},\bm{y})}{\partial\bs n(x)}\phi_{2j}(\bm{y})\mathrm{d}S_{\bm{y}},\; \widehat{\mathcal D}_{2j}\phi_{2j}(\bs x):=\mathrm{p.f.}\int_{\partial \Omega_{2j}}\frac{\partial^2 G_{\kappa}(\bm{x},\bm{y})}{\partial \bm{n}(\bm{x})\partial \bm{n}(\bm{y})}\phi_{2j}(\bm{y})\mathrm{d}S_{\bm{y}},\\
&\widehat{\mathcal S}_{3j}\phi_{3j}(\bs x):=h\Big(\mathrm{p.v.} \int_{\partial \Omega_{3j}} G_{\kappa}(\bm{x},\bm{y})\phi_{3j}(\bm{y})\mathrm{d}S_{\bm{y}}\Big)+\mathrm{p.v.} \int_{\partial \Omega_{3j}} \frac{\partial G_{\kappa}(\bm{x},\bm{y})}{\partial\bs n(x)}\phi_{3j}(\bm{y})\mathrm{d}S_{\bm{y}},\\
&\widehat{\mathcal D}_{3j}\phi_{3j}(\bs x):=h\Big(\mathrm{p.v.} \int_{\partial \Omega_{3j}} \frac{\partial G_{\kappa}(\bm{x},\bm{y})}{\partial \bs n(y)}\phi_{3j}(\bm{y})\mathrm{d}S_{\bm{y}}\Big)+\mathrm{p.f.}\int_{\partial \Omega_{3j}}\frac{\partial^2 G_{\kappa}(\bm{x},\bm{y})}{\partial \bm{n}(\bm{x})\partial \bm{n}(\bm{y})}\phi_{3j}(\bm{y})\mathrm{d}S_{\bm{y}}.
\end{split}
\end{equation*}
Applying the integral representations \eqref{integralrepsolu} to \eqref{localboundaryvalue} gives
\begin{equation}\label{gigs}
g_{ij}(\bs x)=-\mathscr B_{ij}u^{\rm in}(\bs x)-\sum_{k=1,k\neq j}^{M_i} \mathscr B_{ij} \mathscr K_{ik}\phi_{ik}(\bs x)-\sum_{\ell=1,\ell\neq i}^{3}\sum_{k=1}^{M_{\ell}} \mathscr B_{ij} \mathscr K_{\ell k}\phi_{\ell k}(\bs x),\;\;\forall\bs x\in\partial\Omega_{ij}.
\end{equation}
Then,  substituting \eqref{gigs}  into \eqref{BIElocal} leads to the following system of boundary integral equations, which  is uniquely solvable (cf. \cite{colton2013integral}). 
\begin{proposition}\label{integralformula} 
The boundary integral equations of the problem \eqref{localscatteringproblem}-\eqref{localboundaryvalue} take the form:
\begin{equation}\label{BIEglobal}
\frac{1}{2}\phi_{ij}+\widehat{\mathscr K}_{ij}\phi_{ij}+\sum_{k=1,k\neq j}^{M_i} \mathscr B_{ij} \mathscr K_{ik}\phi_{ik}+\sum_{\ell=1,\ell\neq i}^{3}\sum_{k=1}^{M_{\ell}} \mathscr B_{ij} \mathscr K_{\ell k}\phi_{\ell k}=-\mathscr B_{ij}u^{\rm in},
\end{equation}
for all $\bs x\in \partial\Omega_{ij},  j=1, 2, \cdots M_i, i=1, 2, 3$. The scattering field $u^{\rm sc}$ can be obtained by using \eqref{superposition} and \eqref{integralrepsolu}. 
\end{proposition}

\subsection{Iterative method}
Recall that $w_{ij}$ and $\phi_{ij}$ are solution and density of the scattering problem \eqref{localscatteringproblem}, so we have
\begin{equation}\label{localboundaryintegraleq}
\Big(\frac{1}{2}\mathcal I+\widehat{\mathscr{K}}_{ij}\Big)\phi_{ij}=W_{ij}:=\mathscr B_{ij}w_{ij},\quad {\rm on}\;\;\partial\Omega_{ij}.
\end{equation}
According to the boundary integral equation theory (cf. \cite{colton2013integral}),  the operator $\widetilde{\mathscr K}_{ij}:=\frac{1}{2}\mathcal I+\widehat{\mathscr{K}}_{ij}$ is invertible and its inverse is  bounded.
Applying $\phi_{ij}=\widetilde{\mathscr K}_{ij}^{-1}W_{ij}$ to \eqref{BIEglobal}, we obtain
\begin{equation}\label{BIEglobalnew}
W_{ij}+\sum_{k=1,k\neq j}^{M_i} \mathscr B_{ij} \mathscr K_{ik}\widetilde{\mathscr K}_{ik}^{-1}W_{ik}+\sum_{\ell=1,\ell\neq i}^{3}\sum_{k=1}^{M_{\ell}} \mathscr B_{ij} \mathscr K_{\ell k}\widetilde{\mathscr K}_{\ell k}^{-1}W_{\ell k}=-\mathscr B_{ij}u^{\rm in}, \quad{\rm on}\;\; \partial\Omega_{ij}.
\end{equation}
Equivalently, we have 
\begin{equation}\label{conciseeq}
(\mathcal I+\mathbb K)\bs W=\bs b,
\end{equation}
where $\bs W=(W_{11}, \cdots, W_{1M_1}, W_{21}, \cdots, W_{2M_2}, W_{31}, \cdots, W_{3M_3})^{\rm T}$, and
\begin{equation}\label{matrixform}
\mathbb K=
\begin{pmatrix}
\mathcal O & \mathscr B_{11} & \cdots & \mathscr B_{11}\\
\mathscr B_{12} & \mathcal O & \cdots & \mathscr B_{12}\\
\vdots  & \vdots & \ddots & \vdots\\
\mathscr B_{3M_3} & \mathscr B_{3M_3} &\cdots & \mathcal O \\
\end{pmatrix}\begin{pmatrix}
\mathscr K_{11}\widetilde{\mathscr K}_{11}^{-1}  & \mathcal O &\cdots & \mathcal O\\
\mathcal O  & \mathscr K_{12}\widetilde{\mathscr K}_{12}^{-1} &\cdots & \mathcal O\\
\vdots  & \vdots & \ddots  & \vdots\\
\mathcal O & \mathcal O&\cdots  & \mathscr K_{3M_3}\widetilde{\mathscr K}_{3M_3}^{-1}
\end{pmatrix}.
\end{equation}
Let $\mathscr{S}_{ij}: C(\partial \Omega_{ij})\mapsto C^2(\mathbb{R}^2\setminus \overline{\Omega}_{ij})$ be the solution operator of the single scattering problem
\begin{equation}\label{localgeneral}
\begin{cases}
\displaystyle \Delta v+\kappa^2 v=0,\quad  & \bm{x}\in \mathbb{R}^2\setminus \overline{\Omega}_{ij},\\
\displaystyle \mathscr B_{ij}v=\psi,\quad & \bm{x} \in \partial \Omega_{ij},\\
\displaystyle \frac{\partial v}{\partial r}-{\ri}\kappa v=o\big(r^{-\frac{1}{2}}\big),\quad & \mathrm{as}\;\; r:=|\bm{x}|\rightarrow \infty.
\end{cases}
\end{equation}
Then we have 
\begin{equation}\label{equiveqn}
\mathscr K_{ij}\widetilde{\mathscr K}_{ij}^{-1}=\mathscr S_{ij},
\end{equation}
and $\mathscr K_{ij}\widetilde{\mathscr K}_{ij}^{-1}W_{ij}$ for any given data $W_{ij}$ can be obtained by solving single scattering problem \eqref{localgeneral} with $\psi=W_{ij}$.

Different from the classic boundary integral method which usually solves \eqref{BIEglobal} for density, we solve the integral equation \eqref{BIEglobalnew} by the iterative method. Many iterative approaches (e.g., Gauss-Seidel or  generalized minimal residue) can be employed. For the sake of convergence analysis, we choose  the generalized minimal residual (GMRES cf. \cite{saad1986gmres}) iterative method (see {\bf Algorithm 1} below). 

As we have seen in the proposed iterative algorithm, the key part in GMRES iteration is the computation of the terms $\mathscr K_{ij}\widetilde{\mathscr K}_{ij}^{-1}W_{ij}^{(k)}$ in $(\mathcal I+\mathbb K)\bs W^{(k)},$ where $W_{ij}^{(k)}$ is  given boundary data. This can be done by first solving  the boundary integral equations \eqref{localboundaryintegraleq},  and then calculating the integrals for  $\mathscr K_{ij}$ (cf. \eqref{mixedpotential}). Here, we will apply the spectral element solver with NRBC truncation for the single scattering problems in \eqref{localgeneral}.
We first  truncate the unbounded computational domain $\mathbb R^2\setminus\Omega_{ij}$ by using a circular artificial  boundary namely $\Gamma^{ij}$ centered at $\bs c_{ij}=(x^c_{ij}, y^c_{ij})$ with radius $R_{ij}$ (see Fig. \ref{dtntruncatiion}). Denote by $B_{ij}$  the domain enclosed by $\Gamma^{ij}$.
Then the scattering problem \eqref{localgeneral} can be reduced to the following  boundary value 
 \begin{algorithm}\label{algorithm1}
 	\caption{Iterative algorithm for multiple scattering in homogeneous media}
 	\begin{algorithmic}
 		\State  \underline{\em Initialisation}
 		\smallskip
 		\State  (i) Given $\bs W^{(0)}$ on the boundary of $\{\Omega_{ij}\}$, stopping threshold:  $\varepsilon$ and the maximum number of iterations:  $n_{max}$;
 		\smallskip
 		\State (ii)  Solve \eqref{localgeneral} with $\psi=W_{ij}^{(0)}$ for $w_{ij}^{(0)}$,  $i=1,2, 3, j=1, 2, \cdots,  M_i$;
 		\smallskip
 		\State (iii) From  $\mathscr K_{ij}\widetilde{\mathscr K}_{ij}^{-1}W_{ij}^{(0)}=w_{ij}^{(0)}$, we compute $\bs r^{(0)}=\bs b-(\mathcal I+\mathbb K){\bm W}^{(0)}$, and $\bs v^{(1)}=\bs r^{(0)}/\|\bs r^{(0)}\|$.
 		\smallskip
 		\State  \underline{\em Iterative steps}
 		\smallskip
 		\For{$n=1,2,\cdots, n_{max}$}
 		\For{$k=2 \to n$}
 		\State Solve \eqref{localgeneral} with $\psi=v_{ij}^{(k-1)}$ for $\tilde v_{ij}$, $i=1,2, 3, j=1, 2, \cdots,  M_i$;
 		\State From  $\mathscr K_{ij}\widetilde{\mathscr K}_{ij}^{-1} v^{(k-1)}_{ij}=\tilde v_{ij}$, we compute $\bs v^{(k)}=(\mathcal I+\mathbb K)\bs v^{(k-1)}$;
 		\For{$m=1 \to k-1$}
 		\begin{equation}
 		h_{m,k-1}=(\bs v^{(m)}, \bs v^{(k)});\quad \bs v^{(k)}=\bs v^{(k)}-h_{m,k-1}\bs v^{(m)};
 		\end{equation}
 		\EndFor
 		\State $h_{k,k-1}=\|\bs v^{(k)}\|$, $\bs v^{(k)}=\bs v^{(k)}/h_{k,k-1}$;
 		\EndFor
 		\State Compute $W^{(n)}=W^0+\sum_{i=1}^{n-1}y_i\bs v_i,$ where $\bs y=(y_1, y_2, \cdots, y_{n-1})^{\rm T}$ minimizes
 		\begin{equation}
 		J(\bs y):=\|(\|\bs r_0\| \bs e_1)-\bar{H}_{n-1}\bs y\|.
 		\end{equation}
 		\If{$J(\bs y)\leq \varepsilon$}
 		\State  Stop iteration.
 		\EndIf
 		\EndFor
 		\smallskip
 		\State  \underline{\em Final step}
 		\smallskip
 		\State Solve \eqref{localgeneral} with $\psi=W_{ij}^{(n)}$ for $w_{ij}^{(n)}$, $j=1, 2, \cdots,  M_i$, $i=1,2, 3$.
 	\end{algorithmic}
 \end{algorithm}
 \begin{figure}[ht!]
 	\centering
 	\includegraphics[scale=0.55]{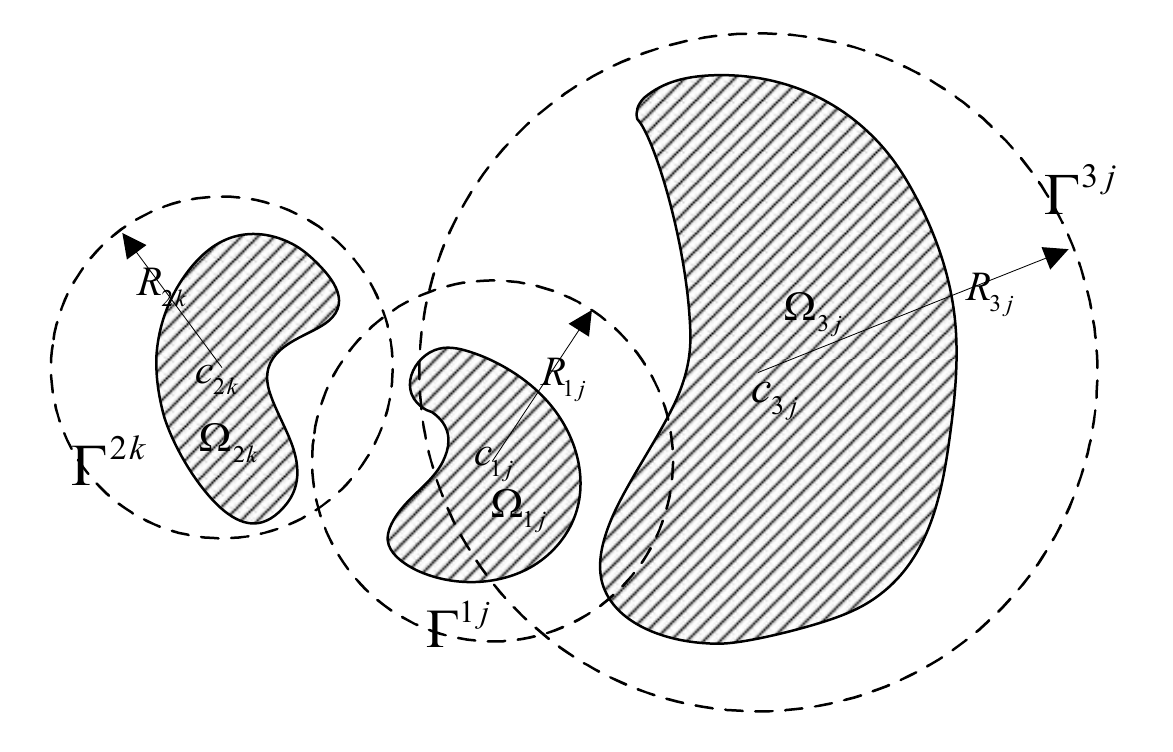}\quad
 	\includegraphics[scale=0.45]{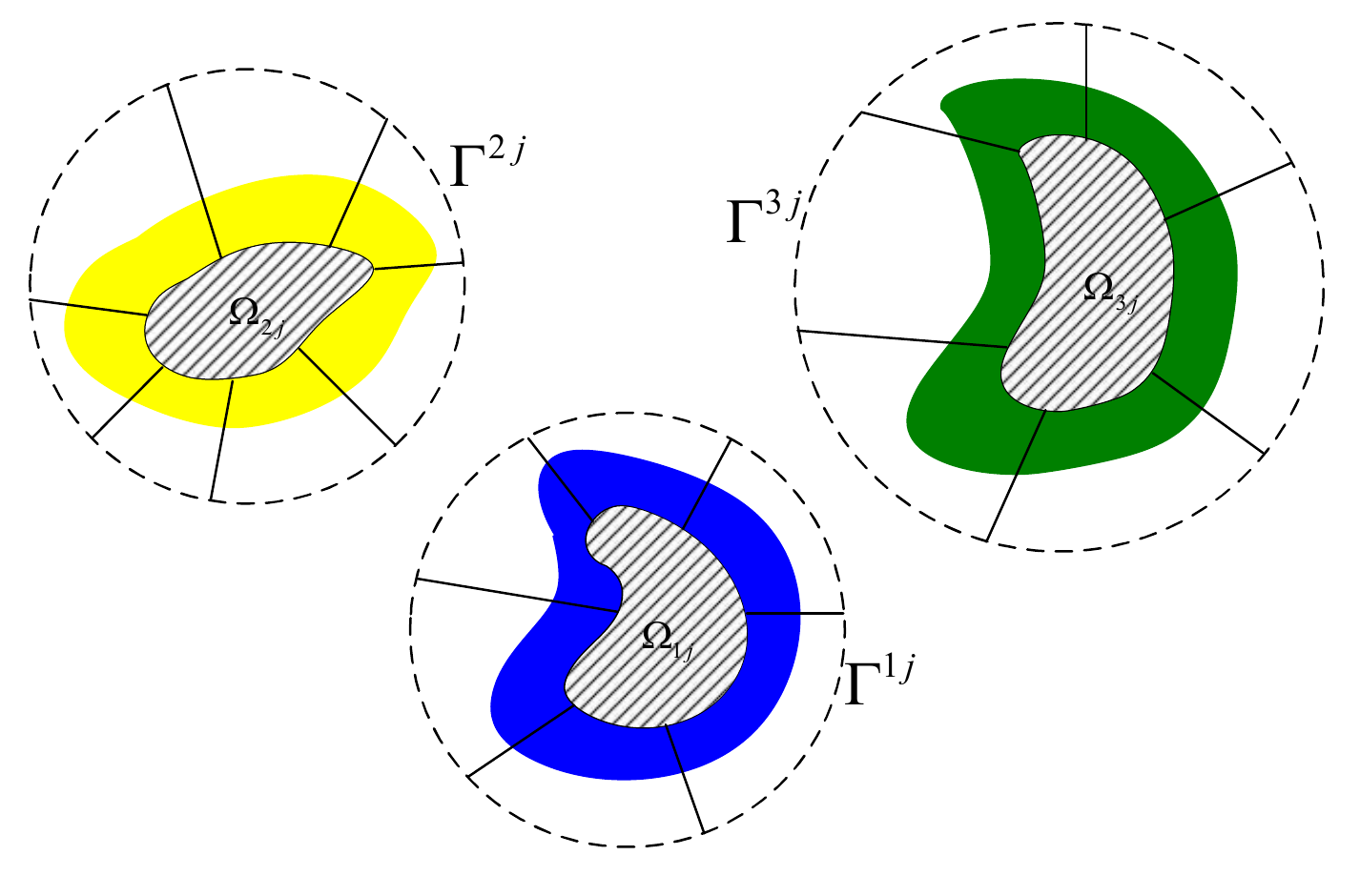}
 	\caption{Left: Intersecting artificial boundaries. Right: Well-separated artificial boundaries and spectral element mesh  for inhomogeneous media problem.}
 	\label{dtntruncatiion}
 \end{figure}
problems (BVPs):
\begin{equation}\label{boundedproblem}
\begin{cases}
\displaystyle \Delta v+\kappa^2 v = 0,\quad &\bm{x}\in B_{ij}\backslash\Omega_{ij},\\
\displaystyle \mathscr B_{ij}v=\psi\quad &\bm{x}\in \partial\Omega_{ij},\\
\displaystyle \partial_{\bs n} v=\mathscr T_{ij}v,\quad &\bm{x}\in \Gamma^{ij},
\end{cases}
\end{equation}
where the DtN operators in the NRBC are given by
\begin{equation}\label{DtNmapping}
\mathscr T_{ij}v:=\sum_{n=-\infty}^{\infty}\kappa\frac{H_n^{(1)'}(\kappa R_{ij})}{H^{(1)}_n(\kappa R_{ij})}\widehat{v}_ne^{\ri n\theta_{ij}}.
\end{equation}
Here, $(r_{ij}, \theta_{ij})$ are the polar coordinates of $\bs x-\bs c_{ij}$, $\{H_n^{(1)}(z)\}$ are the Hankel functions of the first kind,  and
\begin{equation}
\widehat{v}_n=\frac{1}{2\pi}\int_0^{2\pi}
v(x^c_{ij}+R_{ij}\cos\theta_{ij}, y^c_{ij}+R_{ij}\sin\theta_{ij})e^{\ri n\theta_{ij}}\mathrm{d}\theta_{ij},
\end{equation}
are the Fourier coefficients of $v$ on the artificial boundary $\Gamma^{ij}$.
The scattering fields outside the truncated domains will be calculated from the data along the artificial boundaries by using  the separation of variable method. More details on the discretization will be provided  in the next section. 
\begin{rem}\label{remark1}
	Artificial boundaries  in other forms (e.g., ellipse) can also be used in \eqref{boundedproblem} for better adaptation to the shapes of the scatterers. Different from the well separated artificial boundaries required by the DtN boundary condition proposed in \cite{grote2004dirichlet}, the artificial boundaries $\Gamma^{ij}$ used here are independent of each other. They are used individually in the truncation of each single scattering problem \eqref{localgeneral}. Therefore overlapped artificial boundaries can be used as shown in Fig. \ref{dtntruncatiion} (Left). This can relax the assumption of the well separateness of the scatterers.  
\end{rem}

Although the essential unknowns are the boundary data $W_{ij}$ on $\partial\Omega_{ij}$, the purely outgoing components $w_{ij}$ in the exterior domains $\mathbb R^2\setminus\overline{\Omega}_{ij}$ will be calculated in the iterations.  
According to the algorithm, only $M$ single scattering problems need to be solved individually at each iteration. Since the system \eqref{BIEglobalnew} are an equivalent form of the boundary integral system \eqref{BIEglobal}, it enjoys the nice property of relatively small condition number. Consequently, the proposed iterative method converges within a small number of iterations. It will be validated by the numerical  examples in section \ref{sect5} that the number of iterations is nearly independent of the mesh size and polynomial degree used in the discretisation.


\section{Iterative method for the multiple scattering in locally  inhomogeneous media}\label{sect3}
In this section, we present the iterative method for multiple scattering problem in locally inhomogeneous media. In general, purely outgoing wave decomposition is not available when inhomogeneous medium is involved. Nevertheless, it is usually reasonable to assume that the inhomogeneity of the medium is confined  in a finite domain \cite{grote2004dirichlet}.

\subsection{Integral equations on the artificial boundaries}
Assuming  that all scatterers are well separated,  we can surround them by $M$ non-intersecting circles $\{\Gamma^{ij}\}_{j=1}^{M_i}, i=1, 2, 3$ centered at $\{\bs c_{ij}=(x_{ij}^c, y_{ij}^c)\}$ with radii $R_{ij}$ (see Fig. \ref{dtntruncatiion}). Denote by $B_{ij}$ the domain enclosed by $\Gamma^{ij}$, $B_i=\cup_{j=1}^{M_i}B_{ij}$, $B=B_1\cup B_{2}\cup B_{3}$. We further assume that $1-n(\bs x)$ vanishes outside the finite region $B$, i.e., the inhomogeneity is confined  inside $B$ (see Fig. \ref{dtntruncatiion} (right)). A medium that satisfies this assumption is called a locally inhomogeneous medium. Therefore, we only have homogeneous medium outside the region $B,$ and the scattering field $u^{\rm sc}=u-u^{\rm in}$ outside $B$ has a unique decomposition (cf. \cite{Antoine2008On}):
\begin{equation}\label{superpositionoutside}
u^{\rm sc}(\bm{x})=\sum_{j=1}^{M_1} w_{1j}(\bm{x})+\sum_{j=1}^{M_2} w_{2j}(\bm{x})+\sum_{j=1}^{M_3} w_{3j}(\bm{x}), \quad \bs x\in\mathbb R^2\setminus \bar B,
\end{equation}
where $\{w_{ij}\}_{j=1}^{M_i}$ are the solutions of scattering problems:
\begin{equation}\label{localscatteringproblemoutside}
	\begin{cases}
	\displaystyle \Delta w_{ij} +\kappa^2 w_{ij} =0,\quad & \mathrm{in}\quad\mathbb R^2\backslash \bar B_{ij},\\[3pt]
	\displaystyle w_{ij}=g_{ij},\quad & \mathrm{on}\quad \Gamma^{ij},\\[4pt]
	\displaystyle \frac{\partial w_{ij}}{\partial r}-\ri \kappa w_{ij}=o\big(r^{-\frac{1}{2}}\big),\quad  &\mathrm{as}\quad r:=|\bm{x}|\rightarrow \infty, 
	\end{cases}
\end{equation}
for  $j=1, \cdots, M_i,$ $i=1, 2, 3$. The boundary data $g_{ij}$ is
\begin{equation}
g_{ij}=u- u^{\rm in}-\sum_{k=1,k\neq j}^{M_i} w_{ik}-\sum_{\ell=1,\ell\neq i}^{3}\sum_{k=1}^{M_{\ell}}  w_{\ell k},\quad \mathrm{on}\quad \Gamma^{ij}.\label{localboundaryvalueoutside}
\end{equation}
It is worthy to  point out that we have used the total field $u$ to determine $g_{ij}$.

\begin{rem}
	In the model problem, the inhomogeneity of the medium is assumed to be in the neighbourhood of each scatterer and well-separated. If the inhomogeneity around some scatterers is not well-separated, these scatterers should be treated as a group  surrounded by a relatively larger artificial boundary. 
\end{rem}

Define the  mixed potentials ${\mathscr K}'_{ij}\phi_{ij}=\mathcal D'_{ij}\phi_{ij}+\ri\eta \mathcal S'_{ij}\phi_{ij}$
with the  densities $\phi_{ij}$, where $\mathcal S'_{ij}\phi_{ij}$, $\mathcal D'_{ij}\phi_{ij}$ are single and double layer potentials on $\Gamma^{ij}$, and $\eta$ is a given constant such that $\eta\mathfrak{Re}\kappa\geq 0$.
According to the boundary integral  theory, the solutions of the local scattering problems \eqref{localscatteringproblemoutside} have the following integral representations:
\begin{equation}\label{integralrepsoluoutside}
w_{ij}(\bs x)=\mathscr K'_{ij}\phi_{ij}(\bs x), \quad \forall\bs x\in\mathbb R^2\setminus\bar {B}_{ij},
\end{equation}
where the density functions  $\{{\phi}_{ij}\}$ satisfy the  boundary integral equations
\begin{equation}\label{BIElocaloutside}
\widehat{\mathscr K}'_{ij}\phi_{ij}+\frac{1}{2}\phi_{ij}=g_{ij},\quad\bs x\in \Gamma^{ij}, \;\;\;  j=1, 2, \cdots,  M_i, \;\;i=1, 2, 3. 
\end{equation}
Here, the boundary integral operators $\widehat{\mathscr K}'_{ij}$ are defined as
\begin{equation}\label{integralopnohom}
\widehat{\mathscr K}'_{ij}\phi_{ij}=\ri\eta\, \mathrm{p.v.} \int_{\Gamma^{ij}} G_{\kappa}(\bm{x},\bm{y})\phi_{ij}(\bm{y})\,\mathrm{d}S_{\bm{y}}+\mathrm{p.v.}\int_{\Gamma^{ij}}\frac{\partial G_{\kappa}(\bm{x},\bm{y})}{\partial \bm{n}(\bm{y})}\phi_{ij}(\bm{y})\mathrm{d}S_{\bm{y}}.
\end{equation}
Boundary integral equations \eqref{BIElocaloutside} are derived by applying the Dirichlet boundary conditions in  \eqref{localscatteringproblemoutside} 
and limiting properties given in Theorem \ref{theorem1} and Theorem \ref{theorem2}. Thus, inserting  \eqref{integralrepsoluoutside} into \eqref{localboundaryvalueoutside}, we obtain
\begin{equation}
g_{ij}=u-u^{\rm in}-\sum_{k=1,k\neq j}^{M_i} \mathscr K'_{ik}\phi_{ik}-\sum_{\ell=1,\ell\neq i}^{3}\sum_{k=1}^{M_{\ell}} \mathscr K'_{\ell k}\phi_{\ell k},\quad{\rm on}\;\;\Gamma^{ij},
\end{equation}
A substitution of the above equations in \eqref{BIElocaloutside} gives the following system of integral equations
\begin{equation}\label{BIEglobaloutside}
\frac{1}{2}\phi_{ij}+\widehat{\mathscr K}'_{ij}\phi_{ij}+\sum_{k=1,k\neq j}^{M_i} \mathscr K'_{ik}\phi_{ik}+\sum_{\ell=1,\ell\neq i}^{3}\sum_{k=1}^{M_{\ell}} \mathscr K'_{\ell k}\phi_{\ell k}-u=-u^{\rm in}, \quad{\rm on}\;\;\Gamma^{ij},
\end{equation}
for $ j=1, 2, \cdots M_i,  i=1, 2, 3.$

\subsection{Iterative method}
Note that the equations in \eqref{BIEglobaloutside} involve the values of the total field $u$ confined on the artificial boundaries $\Gamma^{ij}$. Nevertheless, they can be determined by the densities $\{\phi_{ij}\}_{j=1}^{M_i}, i=1, 2, 3$ via solving the boundary value problems in $B_{ij}\setminus \overline{\Omega}_{ij}$, respectively. For notational convenience, let 
\begin{equation}\label{solvingopinterior}
\mathscr{S}'_{ij}: H^{-\frac{1}{2}}(\Gamma_{ij})\mapsto H^1(B_{ij}\setminus \overline{\Omega}_{ij}),
\end{equation}
be the solution operator of the inhomogeneous interior problem
\begin{equation}\label{localgeneralinterior}
\begin{cases}
\displaystyle \Delta v(\bs x)+\kappa^2n(\bs x) v(\bs x)=0,\quad & \bm{x}\in B_{ij}\setminus \overline{\Omega}_{ij},\\[2pt]
\displaystyle \mathscr B_{ij}v=0,\quad & \bm{x} \in \partial \Omega_{ij},\\[2pt]
\displaystyle \mathscr{T}_{ij}^{\prime}v=\mathscr{T}_{ij}^{\prime}u^{\rm in}+\psi,\quad & \bs x\in \Gamma^{ij},
\end{cases}
\end{equation}
where $\mathscr{T}_{ij}^{\prime}:=\frac{\partial}{\partial \bs n}-\mathscr{T}_{ij}$, and $\mathscr{T}_{ij}$ is the DtN operator defined in \eqref{DtNmapping}.
 Recall the decomposition \eqref{superpositionoutside}, the total field on $\Gamma^{ij}$ has the decomposition: 
$$u(\bs x)=u^{\rm in}(\bs x)+\sum\limits_{\ell=1}^3\sum\limits_{k=1}^{M_{\ell}}w_{\ell k}(\bs x),\quad \bs x\in\Gamma^{ij}.$$
Moreover, the purely outgoing wave $w_{ij}$ satisfies the boundary condition $\mathscr{T}_{ij}^{\prime}w_{ij}(x)=0$ for all $\bs x\in \Gamma^{ij}$. Then, we have
\begin{equation}
\mathscr{T}_{ij}^{\prime}u=\mathscr{T}_{ij}^{\prime}\Big(u^{\rm in}+\sum_{k=1,k\neq j}^{M_i} w_{ik}+\sum_{\ell=1,\ell\neq i}^{3}\sum_{k=1}^{M_{\ell}} w_{\ell k}\Big),
\end{equation}
which implies that $u$ in the domain $\bar B_{ij}\setminus \Omega_{ij}$ is the solution of \eqref{localgeneralinterior} with boundary data
$$\psi=\mathscr{T}_{ij}^{\prime}\Big(\sum_{k=1,k\neq j}^{M_i} w_{ik}+\sum_{\ell=1,\ell\neq i}^{3}\sum_{k=1}^{M_{\ell}} w_{\ell k}\Big).$$
The following classic conclusion states the well-posedness of the BVP \eqref{localgeneralinterior} (cf. \cite{melenk2010convergence}).
\begin{theorem}\label{interiorwellposed}
Let $\Omega_{ij}$ be a Lipchitz domain, $n(\bs x)\in L^{\infty}(B_{ij}\setminus \overline{\Omega}_{ij})$, $\psi\in H^{-\frac{1}{2}}(\Gamma^{ij})$. Then \eqref{localgeneralinterior} has a unique weak solution in $H^1(B_{ij}\setminus \overline{\Omega}_{ij})$ such that
	\begin{equation}
	\|v\|_{H^1(B_{ij}\setminus \overline{\Omega}_{ij})}\leq C\|\psi\|_{H^{-\frac{1}{2}}(\Gamma^{ij})}.
	\end{equation}
\end{theorem}

By using the  solution operator $\mathscr{S}'_{ij}$ and the  representation \eqref{integralrepsoluoutside}, the total field $u$ in $\bar B_{ij}\setminus \Omega_{ij}$ can be represented as
\begin{equation}
u(\bs x)=\mathscr S'_{ij}\mathscr{T}_{ij}^{\prime} \Big(\sum_{k=1,k\neq j}^{M_i} \mathscr K'_{ik}\phi_{ik}+\sum_{\ell=1,\ell\neq i}^{3}\sum_{k=1}^{M_{\ell}} \mathscr K'_{\ell k}\phi_{\ell k}\Big),\quad \bs x\in \bar B_{ij}\setminus \Omega_{ij}.
\end{equation}
Substituting it into \eqref{BIEglobaloutside}, we obtain
\begin{equation}\label{BIEglobaloutside1}
\frac{1}{2}\phi_{ij}+\widehat{\mathscr K}'_{ij}\phi_{ij}+(\mathcal I-\mathscr S'_{ij}\mathscr{T}_{ij}^{\prime})\Big(\sum_{k=1,k\neq j}^{M_i} \mathscr K'_{ik}\phi_{ik}+\sum_{\ell=1,\ell\neq i}^{3}\sum_{k=1}^{M_{\ell}} \mathscr K'_{\ell k}\phi_{\ell k}\Big)=-u^{\rm in},\quad{\rm on}\;\;\Gamma^{ij}.
\end{equation}
Recall that $w_{ij}$ and $\phi_{ij}$ are the scattering fields and the corresponding densities of the single scattering problems \eqref{localscatteringproblemoutside}. Then
\begin{equation}
\Big(\frac{1}{2}\mathcal I+\widehat{\mathscr{K}}^{\prime}_{ij}\Big)\phi_{ij}=W_{ij}:=w_{ij}\big|_{\Gamma^{ij}}.
\end{equation}
Again the boundary integral equation theory (cf. \cite{colton2013integral})  implies that the operator $\widetilde{\mathscr K}_{ij}^{\prime}:=\frac{1}{2}\mathcal I+\widehat{\mathscr{K}}_{ij}^{\prime}$ is invertible and its inverse is a bounded linear operator.
Applying $\phi_{ij}=\widetilde{\mathscr K}_{ij}^{\prime -1}W_{ij}$ to \eqref{BIEglobaloutside1}, we obtain
\begin{equation}\label{newintegralformoutside}
W_{ij}+(\mathcal I-\mathscr S^{\prime}_{ij}\mathscr{T}_{ij}^{\prime})\Big(\sum_{k=1,k\neq j}^{M_i} \mathscr K'_{i k}\widetilde{\mathscr K}_{ik}^{\prime -1}W_{ik}+\sum_{\ell=1,\ell\neq i}^{3}\sum_{k=1}^{M_{\ell}}\mathscr K'_{\ell k} \widetilde{\mathscr K}_{\ell k}^{\prime -1}W_{\ell k}\Big)=-u^{\rm in},\;{\rm on}\; \Gamma^{ij}
\end{equation}
for $ j=1, 2, \cdots M_i$ and $i=1, 2, 3$. More concisely, \eqref{newintegralformoutside} can be written as
\begin{equation}\label{conciseeqheter}
(\mathcal I+\mathbb K^{\prime}-\mathbb S^{\prime}\mathbb K^{\prime})\bs W=\bs b,
\end{equation}
where $\bs W=(W_{11}, \cdots, W_{1M_1}, W_{21}, \cdots, W_{2M_2}, W_{31}, \cdots, W_{3M_3})^{\rm T}$, and
\begin{equation*}
\begin{split}
\mathbb{S}^{\prime}&=
\begin{pmatrix}
\mathscr S^{\prime}_{11}\mathscr{T}_{11}^{\prime}   &\cdots & \mathcal O\\
\vdots  & \ddots  & \vdots\\
\mathcal O &\cdots  & \mathscr S^{\prime}_{3M_3}\mathscr{T}_{3M_3}^{\prime}
\end{pmatrix},
\mathbb K^{\prime}=\begin{pmatrix}
\mathcal O & \mathscr K'_{12}\widetilde{\mathscr K}_{12}^{\prime -1} & \cdots & \mathscr K'_{3 M_3}\widetilde{\mathscr K}_{3M_3}^{\prime -1}\\
\mathscr K'_{11}\widetilde{\mathscr K}_{11}^{\prime -1} & \mathcal O & \cdots & \mathscr K'_{3M_3}\widetilde{\mathscr K}_{3M_3}^{\prime -1}\\
\vdots  & \vdots & \ddots & \vdots\\
\mathscr K'_{11}\widetilde{\mathscr K}_{11}^{\prime -1} & \mathscr K'_{1 2}\widetilde{\mathscr K}_{12}^{\prime -1} &\cdots & \mathcal O \\
\end{pmatrix}.
\end{split}
\end{equation*}
As in the case of homogeneous media, we apply the GMRES iterative method to solve the system \eqref{conciseeqheter}. We refer to  {\bf Algorithm 2} for a summary of the algorithm.
\begin{algorithm}\label{algorithm2}
	\caption{Iterative algorithm for multiple scattering in locally inhomogeneous media}
	\begin{algorithmic}
		\State  \underline{\em Initialisation}
		\State (i) Given $\bs W^{(0)}$ on $\{\Gamma^{ij}\}$, stop residue $\varepsilon$ and maximum iteration steps $n_{max}$.
		\State (ii) Solve \eqref{localscatteringproblemoutside} with $g_{ij}=W_{ij}^{(0)}$ for $w_{ij}^{(0)}$ outside $B_{ij}$,  $j=1, 2, \cdots,  M_i, i=1,2, 3$.
		\State (iii) Solve \eqref{localgeneralinterior} with $$\psi=\mathscr{T}_{ij}^{\prime}\bigg(\sum\limits_{m=1, m\neq j}^{M_i}w_{im}^{(0)}+\sum\limits_{\ell=1,\ell\neq i}^{3}\sum\limits_{m=1}^{M_{\ell}}w_{\ell m}^{(0)}\bigg),$$
		for the total field $u^{(0)}(\bs x)$ in $B_{ij}\setminus\overline{\Omega}_{ij}$, $j=1, 2, \cdots,  M_i, i=1,2, 3$.
		\State (iv) Use $w_{ij}^{(0)}$ and $u^{(0)}$ to compute $\bs r^{(0)}=\bs b-(\mathcal I+\mathbb K^{\prime}-\mathbb S^{\prime}\mathbb K^{\prime})\bs W^{(0)}$, and $\bs v^{(1)}=\bs r^{(0)}/\|\bs r^{(0)}\|$.
\vskip 2pt
		\State  \underline{\em Iterative steps} 
		\For{$n=1,2,\cdots, n_{max}$}
		\For{$k=2 \to n$}
		\State Solve \eqref{localscatteringproblemoutside} with $g_{ij}=v_{ij}^{(k-1)}$ for $w_{ij}^{(k-1)}$ outside $B_{ij}$, $i=1,2, 3$, $j=1, 2, \cdots,  M_i$.
		\State Solve \eqref{localgeneralinterior} with $$\psi=\mathscr{T}_{ij}^{\prime}\bigg(\sum\limits_{m=1, m\neq j}^{M_i}w_{i m}^{(k-1)}+\sum\limits_{\ell=1,\ell\neq i}^{3}\sum\limits_{m=1}^{M_{\ell}}w_{\ell m}^{(k-1)}\bigg),$$  
		\State for $\tilde v_{ij}$,  $j=1, 2, \cdots,  M_i$, $i=1,2, 3.$
		\State From $\mathscr K'_{ij} \widetilde{\mathscr K}_{ij}^{\prime -1}v_{ij}^{(k-1)}=w_{ij}^{(k-1)}$ and $$\mathscr S^{\prime}_{ij}\mathscr{T}_{ij}^{\prime}\Big(\sum_{m=1,m\neq j}^{M_i} \mathscr K'_{i m}\widetilde{\mathscr K}_{im}^{\prime -1}v_{im}^{(k-1)}+\sum_{\ell=1,\ell\neq i}^{3}\sum_{m=1}^{M_{\ell}}\mathscr K'_{\ell m} \widetilde{\mathscr K}_{\ell m}^{\prime -1}v_{\ell m}^{(k-1)}\Big)=\tilde v_{ij},$$ 
		\State we compute $\bs v^{(k)}=(\mathcal I+\mathbb K^{\prime}-\mathbb S^{\prime}\mathbb K^{\prime})\bs v^{(k-1)}$;
		\For{$m=1 \to k-1$}
		\begin{equation*}
		h_{m,k-1}=(\bs v^{(m)}, \bs v^{(k)});\quad \bs v^{(k)}=\bs v^{(k)}-h_{m,k-1}\bs v^{(m)}.
		\end{equation*}
		\EndFor
		\State $h_{k,k-1}=\|\bs v^{(k)}\|$, $\bs v^{(k)}=\bs v^{(k)}/h_{k,k-1}$.
		\EndFor
		\State Compute $W^{(n)}=W^0+\sum_{i=1}^{n-1}y_i\bs v_i,$ where $\bs y=(y_1, y_2, \cdots, y_{n-1})^{\rm T}$ minimizes
		\begin{equation*}
		J(\bs y):=\big\|\|\bs r_0\| \bs e_1-\bar{H}_{n-1}\bs y\big\|.
		\end{equation*}
		\If{$J(\bs y)\leq \varepsilon$}
		\State  Stop iteration.
		\EndIf
		\EndFor
		\State  \underline{\em Final step}
		\State Solve \eqref{localscatteringproblemoutside} with $g_{ij}=W_{ij}^{(n)}$ for $w_{ij}^{(n)}$ outside $B_{ij}$, $j=1, 2, \cdots,  M_i$, $i=1,2, 3$. 
		\State Solve \eqref{localgeneralinterior} with  $$\psi=\mathscr{T}_{ij}^{\prime}\bigg(\sum\limits_{m=1, m\neq j}^{M_i}w_{i m}^{(n)}+\sum\limits_{\ell=1,\ell\neq i}^{3}\sum\limits_{m=1}^{M_{\ell}}w_{\ell m}^{(n)}\bigg),$$ for total field $u_{ij}^{(n)}$ in $B_{ij}\setminus\overline{\Omega}_{ij}$,  $j=1, 2, \cdots,  M_i$, $i=1,2, 3$. 
	\end{algorithmic}
\end{algorithm}

Different from  the homogeneous media case, we have two types  of solution operators $\mathscr S_{ij}^{\prime}$ and $\widetilde{\mathscr K}_{\ell k}^{\prime -1}$ involved in the equations in \eqref{conciseeqheter}. For  $\mathscr S_{ij}^{\prime},$ we need to solve boundary value problems \eqref{localgeneralinterior}, which involve general inhomogeneous refraction index $n(\bs x)$. 
On the other hand, ${\mathscr K}_{\ell k}^{\prime}\widetilde{\mathscr K}_{\ell k}^{\prime -1}$ are actually the solution operators of the  problems \eqref{localscatteringproblemoutside} (exterior to a single scatterer). It also can be seen as the extension of the purely outgoing components outside the artificial boundaries $\Gamma^{ij}$ similar to the homogeneous case. High order discretization for the BVP \eqref{localgeneralinterior} (inclduding \eqref{boundedproblem}  as a special case) and the solution of the exterior problem \eqref{localscatteringproblemoutside} will be discussed in the next two subsections.

\begin{rem}
	Although the boundary data $W_{ij}$ of purely outgoing components $w_{ij}$ on $\Gamma^{ij}$ are the unknows in \eqref{conciseeqheter}, the total field $u(\bs x)$ in truncated domains $B_{ij}\setminus\Omega_{ij}$ will be calculated in all the iterations. 
\end{rem}

\subsection{High order spectral element discretization for BVPs} Without loss of generality, we take the BVP \eqref{localgeneralinterior} w.r.t a sound soft scatterer $\Omega_{1j}$ as an example to show the details of the high order spectral element discretization. Similar spectral element discretization  can be extended  to other situations straightforwardly. Let $\mathcal{T}=\{K^e\}_{e=1}^E$ be a non-overlapping quadrilateral partition of the domain $B_{1j}\backslash \overline \Omega_{1j}$ (see Fig. \ref{largeandsmall} (Right)). Assume that each element $K^e$ in the partition $\mathcal{T}$ can be obtained by a transformation
$\bs F^e$ from the reference square
\begin{equation}
\widehat{K}=\{\widehat{\bm{x}}=(\xi,\eta): -1\leqslant \xi,\eta\leqslant 1\}=[-1,1]^2.
\end{equation}
Let
\begin{equation}
\widehat{Q}_p=\mathrm{span}\big\{\xi^{p_1}\eta^{p_2}: -1\leqslant \xi,\eta\leqslant 1, 0\leqslant p_1,p_2\leqslant p\big\},
\end{equation}
 be the space of polynomials of degree less than $p$ along each coordinate direction. For any subdomain $K^e\in\mathcal T,$ we define the  finite dimensional space
\begin{equation}
\mathcal{W}_p(K^e)=\big\{\varphi : \varphi=\widehat{\varphi}\circ (\bs F^e)^{-1},\;\widehat{\varphi}\in\widehat{Q}_p \big\}.
\end{equation}
Then the spectral element approximation space is given by
\begin{equation}
\bm{V}_p=\Big\{v_p\in H^1(B_{1j}\backslash\Omega_{1j}):v_p\big|_{K^e}\in \mathcal{W}_p(K^e), \;v_p\big|_{\partial \Omega_{1j}}=0\Big\}.
\end{equation}
The spectral element discretization of \eqref{localgeneralinterior} is to  find $v_p\in \bm{V}_p$ such that 
\begin{equation}\label{semloalproblem}
\mathcal A(v_p, w)=-\langle \mathscr T_{1j}^{\prime}u^{\rm in}+\psi, w\rangle_{\Gamma^{1j}},\quad \forall w\in {\bm V}_p,
\end{equation}
where
\begin{equation}\label{discretebilinearform}
\mathcal A(v_p, w)=-(\nabla v_p, \nabla w)_{B_{1j}\setminus\Omega_{1j}}+\kappa^2(n(\bs x)v_p,w)_{B_{1j}\setminus\Omega_{1j}}+\langle \mathscr T_{1j}v_p, w\rangle_{\Gamma^{1j}}.
\end{equation}
\begin{rem}\label{truncTmk}  
	In real computation, the DtN boundary condition \eqref{DtNmapping}  needs to be approximated by the truncation:
	${\mathscr  T}^N_{ij}[v]:= \sum_{|n|=0}^{N}  \kappa\frac{H_n^{(1)'}(\kappa R_{ij})}{H^{(1)}_n(\kappa R_{ij})} \hat v_n  e^{\ri n\theta_{ij}}$  with  a suitable cut-off  number $N.$ Harari and Hughes \cite{harari1992analysis} showed that the choice of $N\ge \kappa R_{ij}$  could guarantee the solvability of the approximate problem with a certified  accuracy.   We also refer to \cite{hsiao2011error} for the error analysis and  numerical  studies on  the selection of an optimal cut-off number $N.$ In  practice, the   choice  $N\ge \kappa R_{ij}$ is always safe although it is conservative at times.  Grote and Keller \cite{grote1995nonreflecting} suggested a different modification  of the DtN boundary condition to remove the constraint on $\kappa R_{ij}$ for any fixed $N.$
\end{rem}

In the spectral element discretization, Lagrange nodal basis based on the Legendre Gauss-Lobatto (LGL) points is used and the continuous inner product $(\cdot, \cdot)_{B_{1j}\backslash\Omega_{1j}}$ can be evaluated by element-wise discrete inner product based on tensorial Legendre-Gauss-Lobatto(LGL) quadrature (see e.g., \cite{deville2002high}). However, much care is needed to deal with the term $\langle \mathscr T_{1j}^Nv_p, w\rangle_{\Gamma^{1j}}$, as the DtN operator is global, but the spectral-element approximation  is piecewise. One can evaluate by using the fast Fourier transform (FFT), but this requires an
intermediate interpolation to interplay between spectral-element grids and Fourier points. Since $v_p|_{\Gamma^{1j}}\in C^0$ a naive interpolation only results in a first-order convergence. Here the semi-analytic means introduced in \cite{yang2016seamless} is adopted to compute $\langle \mathscr T_{1j}^Nv_p, w\rangle_{\Gamma^{1j}}$.

Let us recap on the semi-analytic formula for the computation of $\langle \mathscr T_{1j}^Nv_p, w\rangle_{\Gamma^{ij}}$. Denote by $\{\xi_k=\eta_k\}_{k=0}^p$ (in ascending order) the LGL  points in $[-1,1],$ and $\{l_k\}_{k=0}^p$ the associated Lagrange interpolating basis polynomials. Correspondingly,  the spectral-element grids and basis on $K^e$ are given by
\begin{equation}\label{newbasis}
\bs x_{k\ell}=\bs F^e(\xi_k,\eta_\ell),\quad  \psi_{k\ell}(\bs x)=l_k(\xi)l_{\ell}(\eta),\quad 0\le k,\ell\le p,
\end{equation}
where $\bs F^e$ is the Gordon-Hall transform \cite{gordon1973transfinite}. Formally, we can write
\begin{equation}\label{uebasis0}
v_p(x,y)\big|_{K^e}= \sum_{k,\ell} \tilde v_{k\ell}^e\,  l_k(\xi)l_{\ell}(\eta),
\end{equation}
where the unknowns $\{\tilde v_{k\ell}^e\}$ are determined by the scheme  \eqref{semloalproblem}.

\begin{figure}[htbp]
	\begin{center}
		\subfigure[Curvilinear elements]{ \includegraphics[scale=.4]{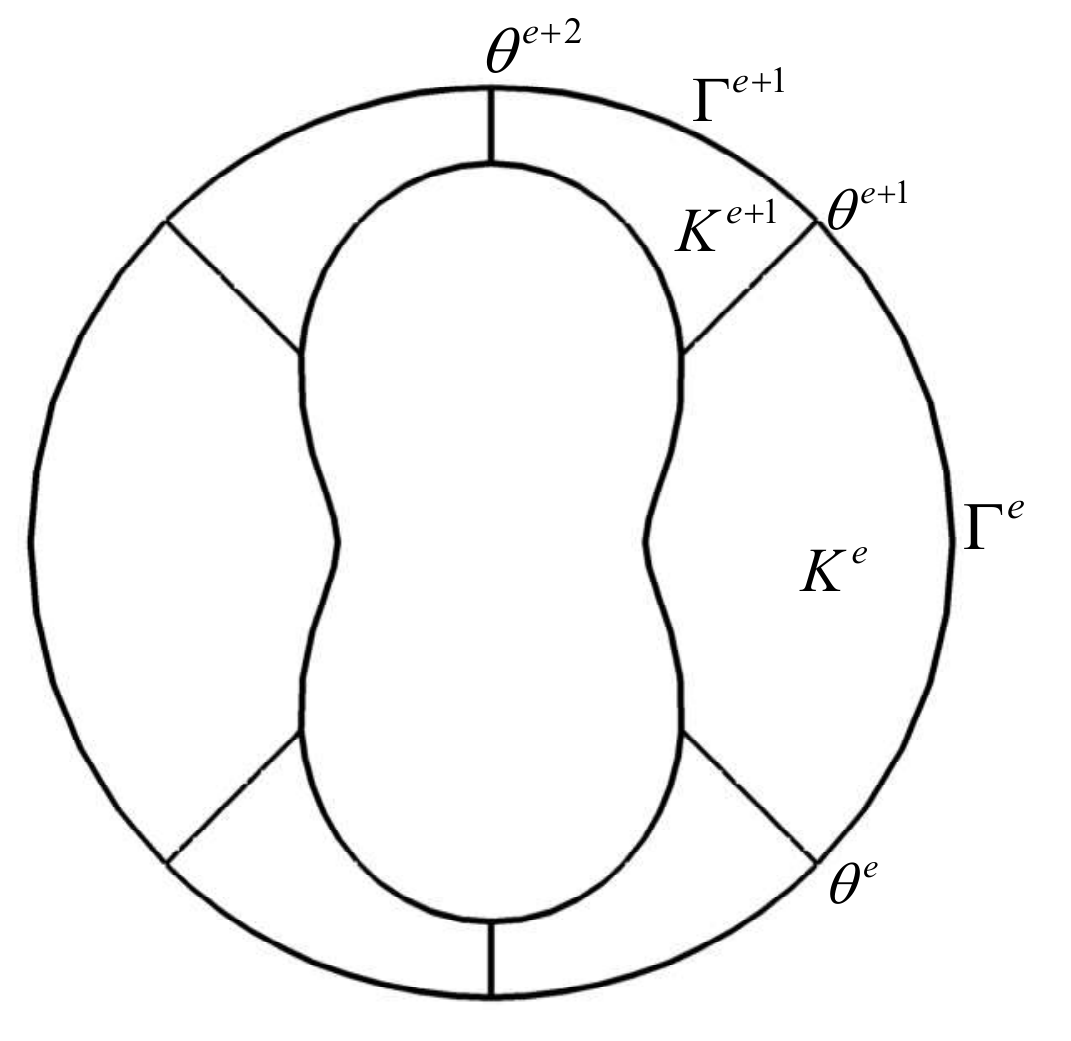}}\hspace*{6pt}
		\subfigure[LGL points on  $\widehat K$]{ \includegraphics[scale=.28]{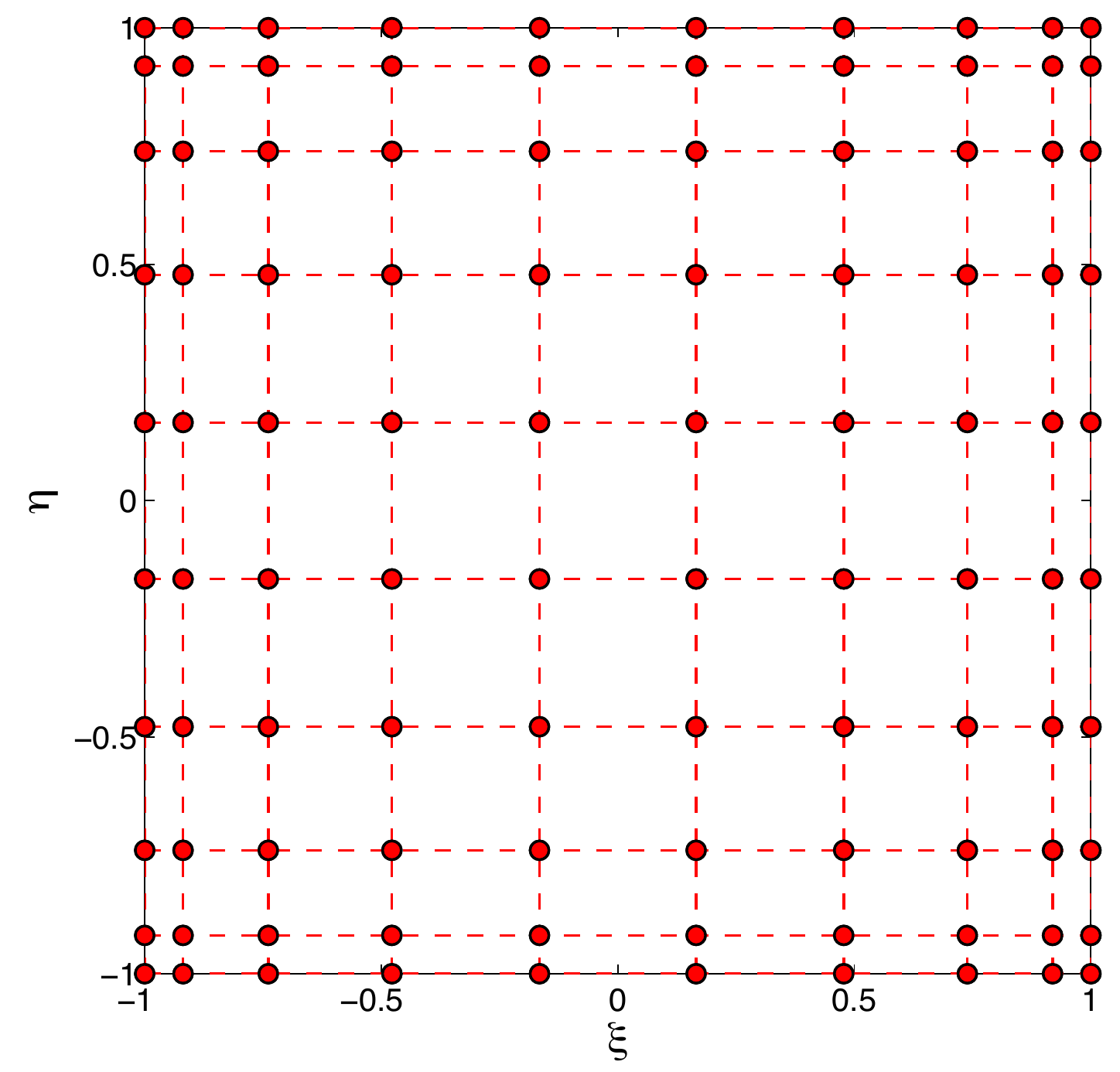}}\hspace*{8pt}
		\subfigure[Mapped LGL points on $K^e$]{ \includegraphics[scale=.31]{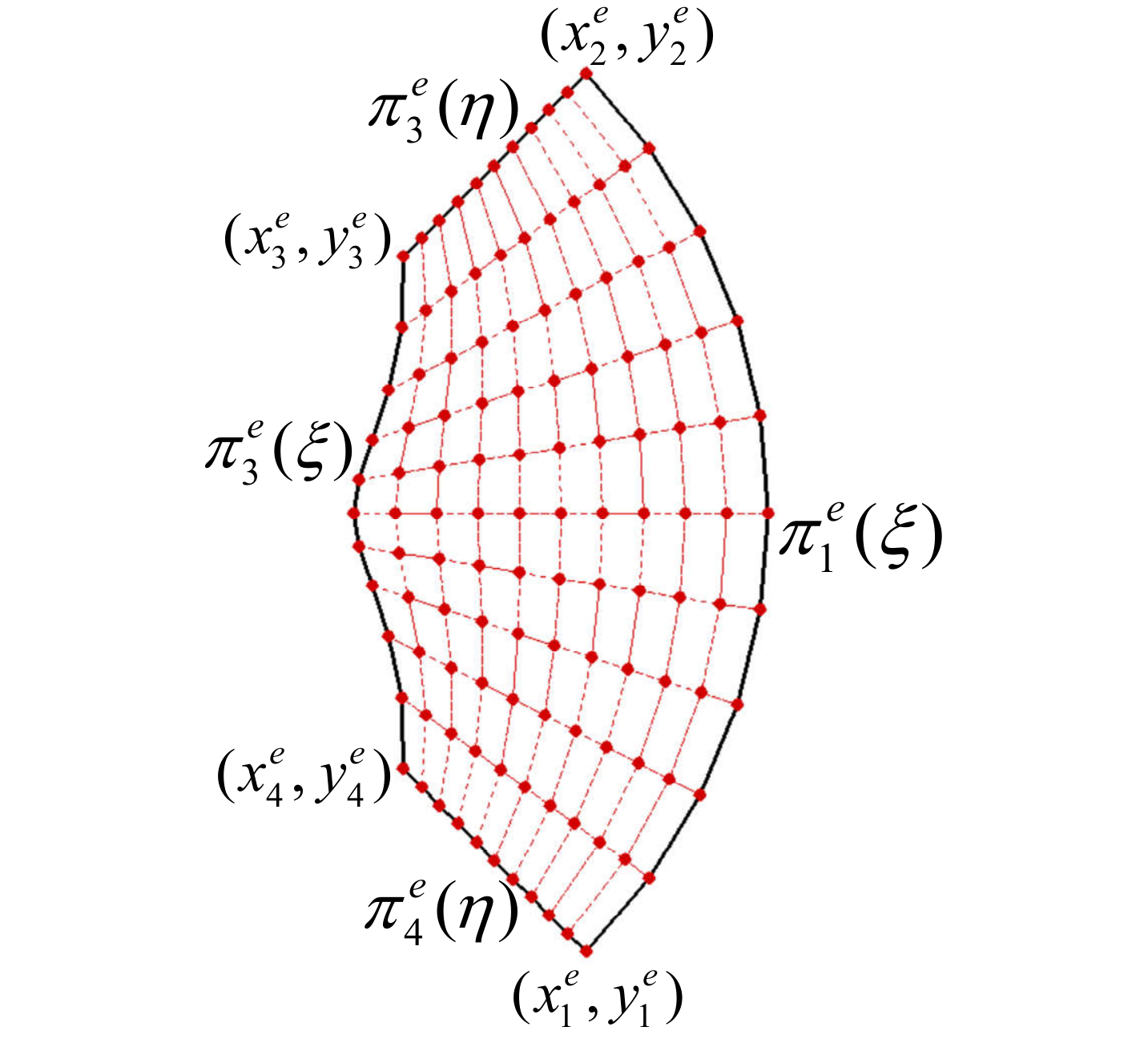}}
		\caption{\small Curvilinear elements and tensorial  LGL points on the reference square and a curvilinear element via  the new elemental mapping based on Gordon-Hall transformation.}
		\label{domaingraph}
	\end{center}
\end{figure}


We choose to use the Gordon-Hall transform for the mapping between reference square to our curvilinear element $K^e$. In particular, we consider a curvilinear element  $K^e$ with vertices $\{(x_k^e,y_k^e)\}_{k=1}^4$ along $\Gamma^{1j}$.  Let  $\{\bs \pi_k^e(t)=(\pi_{k1}^e(t), \pi_{k2}^e(t)), t\in [-1,1]\}_{k=1}^4$ be, respectively, the parametric form of four sides  such that
\begin{equation}\label{ordersetting}
\bs \pi_1^e(-1)= \bs \pi_4^e(1),\;\;  \bs \pi_1^e(1)= \bs \pi_2^e(1),\;\;  \bs \pi_2^e(-1)= \bs \pi_3^e(1),\;\; \bs \pi_3^e(-1)= \bs \pi_4^e(-1),
\end{equation}
see  Figure \ref{domaingraph} (b). In this case,  the Gordon-Hall transform takes the form
\begin{equation} \label{GordonHall}
\begin{split}
{\bs x}&={\bs F}^e(\xi,\eta)= {\bs \pi}_1^e(\xi) \frac{1+\eta}{2}+{\bs \pi}_3^e(\xi)\frac{1-\eta}{2} +\frac{1+\xi}{2} {\bs \pi}_2^e(\eta) +  \frac{1-\xi}{2} {\bs \pi}_4^e(\eta)\\
&- \bigg( {\bs \pi}_1^e(-1) \frac{1-\xi}{2}       +{\bs \pi}_1^e(1)\frac{1+\xi}{2}\bigg) \frac{1+\eta}{2} - \bigg( {\bs \pi}_3^e(-1) \frac{1-\xi}{2}       +{\bs \pi}_3^e(1)\frac{1+\xi}{2}\bigg) \frac{1-\eta}{2}\,,
\end{split}
\end{equation}
where the edge $\eta=1$ of $\widehat K$ is mapped to the arc $\Gamma^e=\{r=R_{ij},\; \theta\in (\theta_e,\theta_{e+1})\}$ of $K^e,$ i.e.,  
\begin{equation}\label{paraformA}
\Gamma^e\,:\,\; x=\pi_{11}^e(\xi),\;\; y= \pi_{12}^e(\xi),\quad \forall\,\xi\in [-1,1].
\end{equation}
Accordingly,  the  spectral-element grids in shifted polar coordinates on $\Gamma^e$  (see Fig.  \ref{domaingraph}) satisfy
\begin{equation}\label{thetaej}
\cos \theta_k^e=R_{1j}^{-1}(\pi_{11}^e(\xi_k)-x_{1j}^c)\;\; {\rm or}\;\;  \sin \theta_k^e=R_{1j}^{-1}(\pi_{12}^e(\xi_k)-y_{1j}^c),  \quad 0\le k\le p.
\end{equation}

Thanks to \eqref{DtNmapping} and \eqref{uebasis0}, the calculation of  $\langle \mathscr T_{ij}^Nv_p, w\rangle_{\Gamma^{ij}}$ needs to evaluate
\begin{equation}\label{mainint}
\begin{split}
\int_0^{2\pi} v_p(x,y)\big|_{\Gamma^{1j}} e^{-\ri n\theta}\,{\rm d} \theta= \sum_{e=1}^{E_{\Gamma}}\sum_k  \tilde v_{kp}^e \int_{-1}^1 l_k(\xi)  e^{-\ri n\theta(\xi)}\frac{{\rm d} \theta}{{\rm d} \xi}\, {{\rm d} \xi},
\end{split}
\end{equation}
where $E_{\Gamma}$ is the number of elements which have one edge coincide with $\Gamma^{1j}$. As the nodal basis $\{l_k\}$ can be represented in terms of  Legendre polynomials,  it suffices to compute
\begin{equation}\label{mainint3}
{\mathbb I}_{nm}^e:=\int_{-1}^1 P_m(\xi)\, e^{-\ri n\theta(\xi)}\, \frac{{\rm d}\theta} {{\rm d}\xi}\, {\rm d}\xi, \;\;\;\; {\rm for}\;\;  m\ge 0,
\end{equation}
where $P_m$ is the Legendre polynomial of degree $m$, and   by \eqref{paraformA},
\begin{equation}\label{dthetax}
\frac{{\rm d}\theta} {{\rm d}\xi}=\frac 1 R_{1j}  \frac{{\rm d}\gamma} {{\rm d}\xi}=R_{1j}^{-1}\sqrt{\big[\partial_\xi \pi_{11}^e(\xi)\big]^2+\big[\partial_\xi \pi_{12}^e(\xi)\big]^2}\,.
\end{equation}
It is seen that  the integrand is highly oscillatory for large $|m|,$ and the efficiency and accuracy in computing  ${\mathbb I}_{nm}^e$  essentially relies on the choice of the  parametric form for $\bs\pi_1^e(\xi).$ It has been shown in \cite{yang2016seamless} that the parametric
\begin{equation}\label{ArcMap2B}
\bs \pi^e_1(\xi)=\big(\pi_{11}^e(\xi), \pi_{12}^e(\xi)\big)=\big(R_{1j}\cos(\hat \theta_e\xi+\beta_e)+x_{1j}^c, R_{1j}\sin(\hat \theta_e\xi+\beta_e)-y_{1j}^c\big),
\end{equation}
with
\begin{equation}\label{hattheta}
\hat \theta_e=\frac{\theta^{e+1}-\theta^e}{2},\quad \beta_e=\frac{\theta^e+\theta^{e+1}}{2},
\end{equation}
has a very important property that $\theta$ is linearly depends on parameter $\xi$. So the Gordon-Hall transformation \eqref{GordonHall} with parametric \eqref{ArcMap2B}, we have
\begin{equation}\label{newrela}
\theta(\xi)=\hat \theta_e\xi+\beta_e,\quad \frac{{\rm d}\theta} {{\rm d}\xi}=\hat \theta_e,
\end{equation}
in \eqref{mainint3}.
This leads to the following analytic formula (cf. \cite{yang2016seamless}) for the integral \eqref{mainint3}:
\begin{equation}\label{analyticF1}
\begin{split}
& {\mathbb I}_{n0}^e=2\hat\theta_e\delta_{n0}; \quad {\mathbb I}_{nm}^e=\frac{2\hat \theta_eR_{1j}}{{\rm i}^n}\sqrt{\frac{\pi}{2m\hat \theta_e}}J_{n+1/2}(m\hat \theta_e)\,
e^{-{\rm i}m\beta_e},
\end{split}
\end{equation}
and ${\mathbb I}_{n,-m}^e=({\mathbb I}_{nm}^e)^*$ for $n\ge 0$, $m\ge 1, $
where $J_{n+1/2}$ is the Bessel function of the first kind.

\subsection{Computation of the scattering field outside the artificial boundary}
Since the purely outgoing wave $w_{\ell k}$ with respect to the scatterer $\Omega_{\ell k}$ will be an incident wave of all other scatterers,  we  need to compute $w_{\ell k}$ on $\partial\Omega_{ij}$ and $\Gamma_{ij}$ for multiple scattering problems in  homogeneous or locally inhomogeneous media, respectively, in the implementation of  the iterative algorithm.  We denote  $\Omega_{ij}$  another scatterer away from $\Omega_{\ell k},$ and denote  $\Gamma^{ij}$  another artificial boundary away from $\Gamma^{\ell k}$.

For the homogeneous media case, we can set the artificial boundary $\Gamma^{\ell k}$ used for truncation \eqref{boundedproblem} large enough (cf. \cite{Geuzaine2010An}) to enclose all other scatterers $\{\Omega_{ij}\}$ inside, see Fig. \ref{largeandsmall} (Left) for an example of two scatterers case. In this case, all boundary information on $\partial\Omega_{3j}$ (colored in red) can be obtained via the numerical solution of the BVP \eqref{boundedproblem} with respect to $\Omega_{1j}$. However, enclosing all other scatterers leads to a large computational domain and hence high computation cost.
\begin{figure}[htbp]
	\centering
	\includegraphics[scale=0.5]{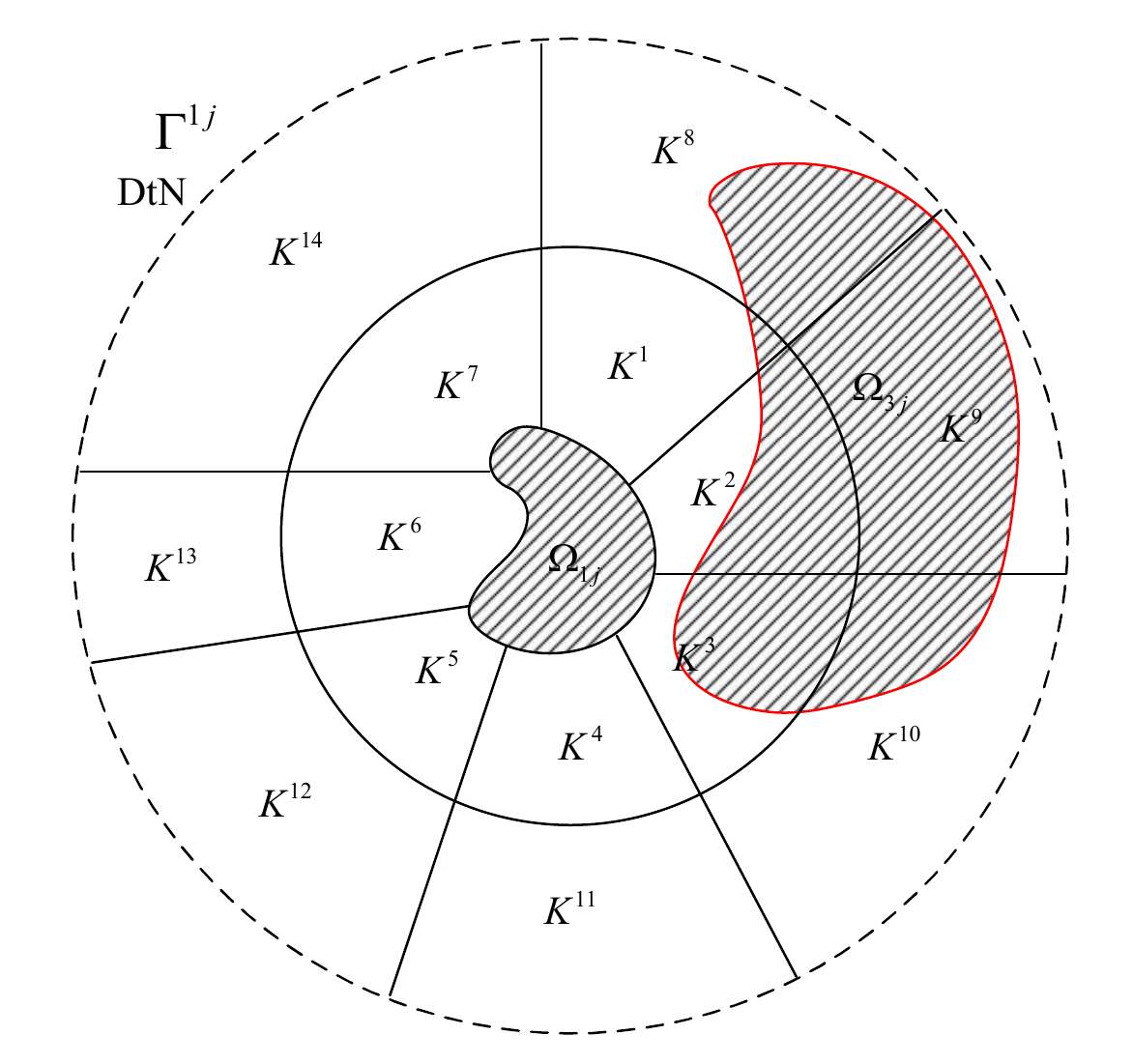}\quad
	\includegraphics[scale=0.6]{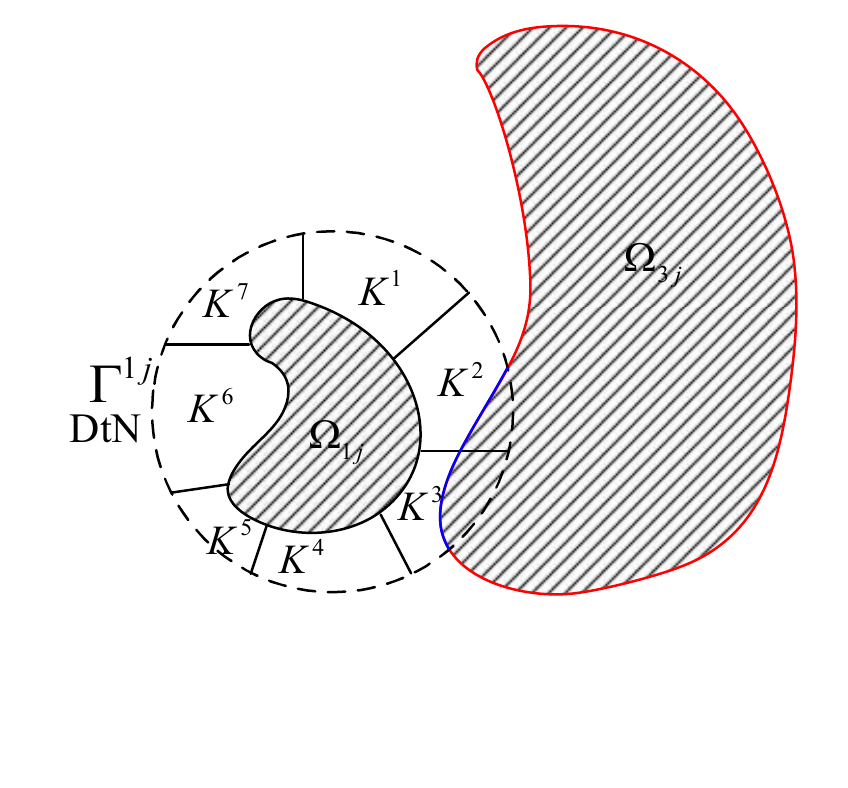}
	\caption{Left: Large artificial boundary to enclose other scatterers. Right: Small artificial boundary intersecting with the boundaries of other scatterers.}
	\label{largeandsmall}
\end{figure}
It is more efficient  to set artificial boundaries close to the scatterers (see  Fig. \ref{largeandsmall} (Right)). In fact,  the  part of the boundary of the scatterer $\Omega_{3j}$ (colored in blue) can be inside the domain $B_{1j}$ and the rest part of $\partial\Omega_{3j}$ (colored in red) can be outside $B_{1j}$. The boundary information $w_{1j}$ on the blue part can be obtained via spectral element solution of the BVP \eqref{boundedproblem} with respect to scatterer $\Omega_{1j}$. However, the boundary information on the red part requires an extension of the numerical solution outside $B_{1j}$. In general,  the extension of the spectral element approximation in $B_{\ell k}\setminus\Omega_{\ell k}$ can be obtained by using the values on the artificial boundary $\Gamma^{\ell k}$. Since $\Gamma^{\ell k}$ here has a circular shape, the extension of a given spectral element solution $v_p$ outside $B_{\ell k}$ is the separation variable solution given by
\begin{equation}
\displaystyle v_p^{\rm ext}(\bm{x})=\sum_{n=-\infty}^{\infty} \frac{\widehat{v}^n_p}{H_h^{(1)}(\kappa R_{\ell k})}H_n^{(1)}(\kappa r_{\ell k})e^{{\ri}n\theta_{\ell k}}, \ \ \bm{x}\notin B_{\ell k},
\end{equation}
where $(r_{\ell k}, \theta_{\ell k})$ is the polar coordinate of $\bs x-\bs c_{\ell k}$,
\begin{equation}
\widehat{v}^n_p = \frac{1}{2\pi} \int_{0}^{2\pi} v_p(x_{\ell k}^c+R_{ij}\cos\theta_{\ell k}, y_{\ell k}^c+R_{\ell k}\sin\theta_{\ell k}) e^{-{\ri}n\theta_{\ell k}}\mathrm{d}\theta_{\ell k},
\end{equation}
is the Fourier coefficients of $v_p|_{\Gamma^{\ell k}}$. By using the analytic formula \eqref{analyticF1}, the Fourier coefficients $\{\widehat{v}^n_p\}$ can be calculated accurately and efficiently for arbitrary high modes.

The scattering problems given by \eqref{localscatteringproblemoutside} will also be solved by using the separation variable method in the same manner due to the circular shape of the artificial boundary $\Gamma^{\ell k}$.

%

\subsection{ A comparison with the Grote-Kirsch's approach in \cite{grote2004dirichlet}} 
In  Grote and  Kirsch \cite{grote2004dirichlet},  the reduction of  a multiple scattering problem 
with well separated scatterers using  the circular/spherical DtN technique was proposed  as an extension of the DtN for a single scattering problem.  
 Consider for example  \eqref{helmholtzeq}-\eqref{bconscatterers} with the sound soft scatterers,  i.e., $\mathscr B_i=\mathcal I, i=1, 2, 3.$ 
%
 The reduced problem therein   is to 
find $u$ and $\{w_j\}_{j=1}^M$ satisfying 
\begin{equation}\label{trun2dmscatprob}
	\begin{cases}
	\Delta u+\kappa^2u=0 \quad & \hbox{in}\quad B\setminus\bar\Omega,\\
	u=g, \quad & \hbox{on}\quad\partial \Omega,\\
	\partial_{\bs n}u=\sum\limits_{j=1}^M\mathscr{T}_i[w_j],\;\;u=\sum\limits_{j=1}^M\mathscr{P}_i[w_j], 
	\quad & \hbox{on}\;\;\partial B_i,\;\; i=1, 2, \cdots, M, 
	\end{cases}
\end{equation}
where the transport   and propagation operators  are defined by
\begin{equation*}
\begin{split}
&\mathscr{T}_i[w_i](\theta_i):=\sum\limits_{|n|=0}^{\infty}\widehat w_n^i\frac{\kappa H_n^{(1)'}(\kappa R_i)}{H_n^{(1)}(\kappa R_i)}e^{\ri n\theta_i},\\
&\mathscr{T}_i[w_j](\theta_i):=\sum\limits_{|n|=0}^{\infty}\widehat w_n^j\Big(\frac{\kappa H_n^{(1)'}(\kappa r_j({\bs x}))}{H_n^{(1)}(\kappa R_j)}\frac{\partial r_j({\bs x})}{\partial r_i}+\frac{\ri nH_n^{(1)}(\kappa r_j({\bs x}))}{H_n^{(1)}(\kappa R_j)}\frac{\partial \theta_j({\bs x})}{\partial r_i}\Big)e^{\ri n\theta_j(\hat{\bs x})},\;j\neq i,\\
&\mathscr{P}_i[w_i](\theta_i):=w_i(\theta_i)\quad\mathscr{P}_i[w_j](\theta_i):=\sum\limits_{|n|=0}^{\infty}\widehat w_n^j
\frac{H_n^{(1)}(\kappa r_j({\bs x}))}{H_n^{(1)}(\kappa R_j)}e^{\ri n\theta_j({\bs x})},\;j\neq i,
\end{split}
\end{equation*}
for $\bs x=(R_i\cos\theta_i+x_i^c, R_i\sin\theta_i+y_i^c)\in \partial B_i$.  
 As mentioned in \cite{grote2004dirichlet},  one  can apply any finite-domain solver, e.g., the finite element or spectral element discretization, which typically leads to the linear system: 
\begin{equation}\label{blocksystem}
\left(\begin{array}{cc;{2pt/2pt}r}
 & & \mathbb O\\[2pt]
 \multicolumn{2}{c;{2pt/2pt}}{\raisebox{2ex}[0pt]{$\mathbb{\Huge K}$}} & -\mathbb T \\[2pt]\hdashline[2pt/2pt]
\mathbb O& \mathbb M & -\mathbb P 
\end{array}\right)
\begin{pmatrix}
\bs u^h_{\Omega}\\[2pt]
\bs u^h_{\partial B}\\[2pt]\hdashline[2pt/2pt]
\bs w^h
\end{pmatrix}=\begin{pmatrix}
\bs g\\[2pt]\hdashline[2pt/2pt]
\bs 0
\end{pmatrix},
\end{equation}
where $\{\bs u^h_{\Omega}, \bs u^h_{\partial B}\}$ and $\bs w^h$ are the unknowns for  the approximation  of $u$ and $w$, respectively. 
Denote by $\{\Phi_i\}$ and $\mathscr N=\mathscr N_{\Omega}\cup \mathscr N_{\partial \Omega}\cup\mathscr N_{\partial B}$ the nodal basis and nodes used in the discretization. Then the entries of $\mathbb K$, $\mathbb T$, $\mathbb M$, $\mathbb P$ and $\bs g$ are given by
\begin{equation*}
\begin{split}
&K_{ij}=(\nabla\Phi_j, \nabla\Phi_i)_{\Omega}-\kappa^2(\Phi_j, \Phi_i)_{\Omega},\quad i, j: \bs x_i, \bs x_j\in \mathscr N_{\Omega}\cup\mathscr N_{\partial B},\\
& T_{ij}=\langle\mathscr T\Phi_j,\Phi_i\rangle_{\partial B},\quad M_{ij}=\langle\Phi_j, \Phi_i\rangle_{\partial B} \quad M_{ij}=\langle\mathscr P\Phi_j, \Phi_i\rangle_{\partial B},\quad i,j:\bs x_i, \bs x_j\in\mathscr N_{\partial B},\\
&g_i=-\sum\limits_{j:\bs x_j\in\mathscr N_{\partial\Omega}}g(\bs x_j)K_{ij},\quad i:\bs x_i\in\mathscr N_{\partial\Omega},
\end{split}
\end{equation*}
where the operators $\mathscr T$ and $\mathscr P$ are consist of $\{\mathscr T_i\}$ and $\{\mathscr P_i\}$. In fact, the matrix $\mathbb K$ in \eqref{blocksystem} is block diagonal and the coupling of the scatterers is along the artificial boundary of each scatterer.  Indeed, the block iterative method (e.g., block Gauss-Seidel iterative method \cite{li2013two, jiang2012adaptive}) can be applied.
%


Although our approach follows the same spirit of ``decoupling'' the scatterers based on the superposition of waves and suitable iterative solvers,  it is different from the existing  methods in several aspects. 
Most importantly, we  reduce  the multiple scattering problem from a different perspective,	which allows us to conduct the convergence analysis and also leads to more efficient algorithm. 
Indeed, we derive the single scattering problems from the boundary integral theory, and  the communications between the scatterers are made simpler through the purely outgoing waves.  
The use of the  GMRES iteration can  effectively decouple the interior solver for the single scatterer  and interactions from other scatterers (see e.g., {\bf Algorithm 2}).  Through the intrinsic connections between the boundary integral formulation (for the incident waves from other scatterers) and the DtN operator (for the interior solver)  on the artificial boundary (cf. \eqref{equiveqn} and \eqref{newintegralformoutside}),  we are able to handle the interactions between the scatterers more efficiently and also show the convergence of the iterative approach from the boundary integral theory.  
In fact, this provides a more flexible numerical framework for multiple scattering problems, and can  relax the well separateness assumption of scatterers in  \cite{grote2004dirichlet}. 
The advantages of our approach are also verified by  numerical comparisons with the Grote-Kirsch's approach  in section \ref{sect5}.

\section{Convergence analysis for GMRES iteration}\label{sect4}
In this section, we prove  the convergence of the GMRES iteration for  \eqref{conciseeq} and \eqref{conciseeqheter}. The key step is to prove the compactness of operators $\mathbb K$ and $\mathbb K^{\prime}-\mathbb S^{\prime}\mathbb K^{\prime}$. Let us first review some properties of the integral operator defined by
\begin{equation}\label{generalintop}
(\mathscr A\phi)(\bs x):=\int_{G}K(\bs x, \bs y)\phi(\bs y)d\bs y,
\end{equation}
where $G$ is a measurable compact set in $\mathbb R^2$. The following conclusion can be found in many text books on linear integral equations (cf. \cite{mikhlin1960linear,kress1989linear}).
\begin{theorem}\label{theoremcompact}
	If the kernel $K(\bs x, \bs y)$ is continuous or weakly singular, then the integral operator $\mathscr A$ is a compact operator on $L^2(G)$.
\end{theorem}
\begin{theorem}\label{boundedinverse}
	Let $\bs X$ be a normed linear space, $\mathscr A: \bs X\rightarrow \bs X$ a compact linear operator, and let $\mathcal I-\mathscr A$ be injective. Then the inverse operator $(\mathcal I-\mathscr A)^{-1}$ exists and is bounded.
\end{theorem}

Let us first consider the operator $\mathbb K$ involved in homogeneous media case. It consists of the composition of operators $\mathscr B_{ij}$, $\mathscr K_{\ell k}$ and $\widetilde{\mathscr K}_{\ell k}^{-1}$.
\begin{theorem}\label{compactness}
	Suppose $\Omega_{ij}$ and $\Omega_{\ell k}$ are two different scatterers with $C^2$ boundary, $\mathscr B_{ij}$ are differential operators induced by boundary conditions on $\partial\Omega_{ij}$, $\mathscr K_{\ell k}$ are boundary integral operators defined in \eqref{mixedpotential}. Then the composition $\mathscr B_{ij}\mathscr K_{\ell k}$ are compact operators from $L^2(\partial\Omega_{\ell k})$ to $L^2(\partial\Omega_{ij})$.
\end{theorem}
\begin{proof}
	From the definition of $\mathscr B_{ij}$ and $\mathscr K_{\ell k}$, we have
	\begin{equation}
	(\mathscr B_{ij}\mathscr K_{\ell k})\phi_{\ell k}(\bs x)=\begin{cases}
	\mathscr B_{ij}\mathcal{D}_{1 k}\phi_{1 k}+\ri\eta \mathscr B_{ij}\mathcal S_{1k}\phi_{1 k} & \ell=1,\\[4pt]
	-\ri\eta\mathscr B_{ij}\mathcal{D}_{\ell k}\phi_{\ell k}- \mathscr B_{ij}\mathcal S_{\ell k}\phi_{\ell k},& \ell=2\;{\rm or}\; 3.
	\end{cases}
	\end{equation}
    Since $\partial\Omega_{\ell k}$ is a closed curve of class $C^2$, we have
	\begin{equation}
	\begin{split}
	(\mathscr B_{ij}\mathcal{S}_{1 k})\phi_{1 k}(\bs x)&=\int_{\partial\Omega_{1 k}}\mathscr B_{ij} G_{\kappa}(\bs x, \bs y)\phi_{1k}(\bs y)dS_{\bs y},\quad \bs x\in\partial\Omega_{ij},\\
	(\mathscr B_{ij}\mathcal{D}_{1 k})\phi_{1 k}(\bs x)&=\int_{\partial\Omega_{1 k}}\mathscr B_{ij}\Big[\frac{\partial G_{\kappa}(\bs x, \bs y)}{\partial\bs n_y}\Big]\phi_{1k}(\bs y)dS_{\bs y},\quad \bs x\in\partial\Omega_{ij}.
	\end{split}
	\end{equation}
	Moreover,
	\begin{equation}
	\begin{cases}
	\displaystyle\mathscr B_{1j} G_{\kappa}(\bs x, \bs y)=G_{\kappa}(\bs x, \bs y),\;\;\mathscr B_{1j}\Big[\frac{\partial G_{\kappa}(\bs x, \bs y)}{\partial\bs n_y}\Big]=\frac{\partial G_{\kappa}(\bs x, \bs y)}{\partial\bs n_y},\\
	\displaystyle\mathscr B_{2j} G_{\kappa}(\bs x, \bs y)=\frac{\partial G_{\kappa}(\bs x, \bs y)}{\partial\bs n(\bs x)}\;\;\mathscr B_{2j}\Big[\frac{\partial G_{\kappa}(\bs x, \bs y)}{\partial\bs n_y}\Big]=\frac{\partial^2 G_{\kappa}(\bs x, \bs y)}{\partial\bs n_x\partial\bs n_y},\\
	\displaystyle\mathscr B_{3j} G_{\kappa}(\bs x, \bs y)=\frac{\partial G_{\kappa}(\bs x, \bs y)}{\partial\bs n(\bs x)}+hG_{\kappa}(\bs x, \bs y),\\
	\displaystyle\mathscr B_{3j}\Big[\frac{\partial G_{\kappa}(\bs x, \bs y)}{\partial\bs n_y}\Big]=\frac{\partial^2 G_{\kappa}(\bs x, \bs y)}{\partial\bs n_x\partial\bs n_y}+h\frac{\partial G_{\kappa}(\bs x, \bs y)}{\partial\bs n_y},
	\end{cases}
	\end{equation}
	are all continuous for $\bs x\neq \bs y$.  By using Theorem \ref{theoremcompact}, we conclude that $\mathscr B_{ij}\mathcal{D}_{1 k}$ and $\mathscr B_{ij}\mathcal{S}_{1 k}$ are compact operators from $L^2(\partial\Omega_{1k})$ to $L^2(\partial\Omega_{ij})$. Therefore, $\mathscr B_{ij}\mathscr K_{1k}$ are compact operators from $L^2(\partial\Omega_{1k})$ to $L^2(\partial\Omega_{ij})$. The compactness of operators $\mathscr B_{ij}\mathscr K_{\ell k}: L^2(\partial\Omega_{\ell k})\rightarrow L^2(\partial\Omega_{ij}), \ell=2, 3$ can be proved in the same way.
\end{proof}

Note that $\widetilde{\mathscr K}_{ij}^{-1}$ are the solution operators of the linear integral equations \eqref{localboundaryintegraleq}. The well-posedness of the boundary integral equations \eqref{localboundaryintegraleq} implies the boundedness of $\widetilde{\mathscr K}_{ij}^{-1}$.
\begin{theorem}\label{boundednessinverse}
	Assume that all boundaries $\partial\Omega_{\ell k} $ is of class $C^2$, $\mathfrak{Im} (\bar{\kappa}h) \geqslant 0$ and wavenumber $\kappa$ satisfying $\mathfrak{Im} \kappa \geqslant 0$. Then $\widetilde{\mathscr K}_{\ell k}^{-1}$ are bounded linear operators on $L^2(\partial\Omega_{\ell k})$.
\end{theorem}
\begin{proof}
	The boundedness of operator $\widetilde{\mathscr K}_{1 k}^{-1}: L^2(\partial\Omega_{1 k})\rightarrow L^2(\partial\Omega_{1 k})$ is a direct consequence  of Theorem \ref{theoremcompact} and \ref{boundedinverse},  since
	$$\widetilde{\mathscr K}_{1k}=\frac{1}{2}\mathcal I+\widehat{\mathscr K}_{1k}=\frac{1}{2}\mathcal I+\widehat{\mathcal D}_{1k}+\ri\eta\widehat{\mathcal S}_{1k},$$
	is obviously injective in $L^2(\partial\Omega_{1k})$ and $\widehat{\mathcal D}_{1k}$ and $\widehat{\mathcal S}_{1k}$ are integral operators with weakly singular kernels.
	
	For operators $\widetilde{\mathscr K}_{2 k}^{-1}, \widetilde{\mathscr K}_{3 k}^{-1}$, we need to introduce integral operators
	\begin{equation}
	\widehat{\mathscr D}_{\ell k}^0\phi_{\ell k}(\bm{x}):=\mathrm{p.f.}\int_{\partial \Omega_{\ell k}}\frac{\partial^2 G_{\kappa_0}(\bm{x},\bm{y})}{\partial \bm{n}(\bm{x})\partial \bm{n}(\bm{y})}\phi_{\ell k}(\bm{y})\mathrm{d}S_{\bm{y}},\quad\bs x\in\Omega_{\ell k},\;\; \ell=2, 3,
	\end{equation}
	where $\kappa_0$ is a picked wave number which is not an interior eigenvalue to the corresponding Dirichlet and Neumann problems. An important result is that their inverse $(\widehat{\mathscr D}_{\ell k}^0)^{-1}$ exist and are compact on $L^2(\partial\Omega_{\ell k})$ (cf. \cite{mikhlin1960linear,kress1989linear}). Then the boundary integral equations \eqref{localboundaryintegraleq} for $i=2, 3$ can be transformed into the equivalent forms
	\begin{equation}
	\begin{split}
	(\widehat{\mathcal D}_{2k}^0)^{-1}\Big(\frac{1}{2}\mathcal I-\ri\eta(\widehat{\mathcal D}_{2k}-\widehat{\mathcal D}_{2k}^0)-\widehat S_{2k}\Big)\phi_{2k}-\ri\eta\phi_{2k}=(\widehat{\mathcal D}_{2k}^0)^{-1}W_{2k},\\
	(\widehat{\mathcal D}_{3k}^0)^{-1}\Big(\frac{1-\ri\eta h}{2}\mathcal I-\ri\eta(\widehat{\mathcal D}_{3k}-\widehat{\mathcal D}_{3k}^0)-\widehat S_{3k}\Big)\phi_{3k}-\ri\eta\phi_{3k}=(\widehat{\mathcal D}_{3k}^0)^{-1}W_{3k}.
	\end{split}
	\end{equation}
	One can verify that $\widehat{\mathcal D}_{\ell k}-\widehat{\mathcal D}_{\ell k}^0, \ell=2, 3$ are integral operators with weakly singular kernels, so they are compact on $L^2(\partial\Omega_{\ell k}), \ell=2, 3$. Together with the compactness of $(\widehat{\mathscr D}_{\ell k}^0)^{-1}$ and $\widehat{\mathcal S}_{\ell k}$, we conclude that
	\begin{equation}
	\begin{split}
	(\widehat{\mathcal D}_{2k}^0)^{-1}\Big(\frac{1}{2}\mathcal I-\ri\eta(\widehat{\mathcal D}_{2k}-\widehat{\mathcal D}_{2k}^0)-\widehat S_{2k}\Big),\quad (\widehat{\mathcal D}_{3k}^0)^{-1}\Big(\frac{1-\ri\eta h}{2}\mathcal I-\ri\eta(\widehat{\mathcal D}_{3k}-\widehat{\mathcal D}_{3k}^0)-\widehat S_{3k}\Big),
	\end{split}
	\end{equation}
	are compact operators on $L^2(\partial\Omega_{\ell k}), \ell=2, 3$. Then, the boundedness of $\widetilde{\mathscr K}_{2k}^{-1}$ and $\widetilde{\mathscr K}_{3k}^{-1}$ can be obtained from Theorem \ref{boundedinverse} and the following representations
	\begin{equation}
	\begin{split}
	\widetilde{\mathscr K}_{2k}^{-1}=\Big[\mathcal I-(\ri\eta\widehat{\mathcal D}_{2k}^0)^{-1}\Big(\frac{1}{2}\mathcal I-\ri\eta(\widehat{\mathcal D}_{2k}-\widehat{\mathcal D}_{2k}^0)-\widehat S_{2k}\Big)\Big]^{-1}\big(\ri\eta\widehat{\mathcal D}_{2k}^0\big)^{-1},\\
	\widetilde{\mathscr K}_{3k}^{-1}=\Big[\mathcal I-(\ri\eta\widehat{\mathcal D}_{3k}^0)^{-1}\Big(\frac{1-\ri\eta h}{2}\mathcal I-\ri\eta(\widehat{\mathcal D}_{3k}-\widehat{\mathcal D}_{3k}^0)-\widehat S_{3k}\Big)\Big]^{-1}\big(\ri\eta\widehat{\mathcal D}_{3k}^0\big)^{-1}.
	\end{split}
	\end{equation}
	This ends the proof.
\end{proof}

From Theorem \ref{compactness} and Theorem \ref{boundednessinverse}, we conclude that $\mathscr B_{ij}{\mathscr K}_{\ell k}\widetilde{\mathscr K}_{\ell k}^{-1}$ are compact operators from $L^2(\partial\Omega_{\ell k})$ to $L^2(\partial\Omega_{ij})$. Now we consider $\mathbb K$ which is an operator on the product space
$$
\bm{L}^2(\partial\Omega):= L^2(\partial\Omega_{11})\times\cdots\times L^2(\partial\Omega_{1M_1})\times\cdots\times L^2(\partial\Omega_{31})\times\cdots\times L^2(\partial\Omega_{3M_3}),
$$
with the inner product
\begin{equation}\label{innerprod}
(\bm{u},\bm{v})_{\bm{L}^2(\partial\Omega)} := \sum\limits_{i=1}^3\sum_{j=1}^{M_i}(u_{ij}, v_{ij})_{L^2(\partial\Omega_{ij})} ,
\end{equation}
and the norm $||\bm{u}||_{\bm{L}^2(\partial\Omega)}^2 = (\bm{u},\bm{u})_{\bm{L}^2(\partial\Omega)}$. It is evident that $\bm{L}^2(\partial\Omega)$ is a Hilbert space.

\begin{theorem}\label{compactnessinvecspace}
	The operator $\mathbb{K}$ defined in \eqref{matrixform} is compact on the Hilbert space $\bm{L}^2(\partial\Omega)$.
\end{theorem}
\begin{proof}
	For any bounded sequence $\bs v^{(n)}=(v_{11}^{(n)}, \cdots, v_{1M_1}^{(n)},\cdots,v_{31}^{(n)}, \cdots, v_{3M_3}^{(n)} )^{\rm T}, n=1, 2, \cdots, $ in $\bm{L}^2(\partial\Omega)$, denote by $\tilde{\bs v}^{(n)}=\mathbb{K}\bs v^{(n)}$. Then
	\begin{equation}
	\tilde{v}_{ij}^{(n)} = \sum_{k=1,k\neq j}^{M_i} \mathscr{B}_{ij}\mathscr{K}_{ik}\widetilde{\mathscr K}_{ik}^{-1}v_{ik}^{(n)}+\sum\limits_{\ell=1,\ell\neq k}^3\sum_{k=1}^{M_{\ell}} \mathscr{B}_{ij}\mathscr{K}_{\ell k}\widetilde{\mathscr K}_{\ell k}^{-1}v_{\ell k}^{(n)}.
	\end{equation}
	For each scatterer $\Omega_{\ell k}$,  $\{v_{\ell k}^{(n)}\}_{n=1}^{\infty}$ is a bounded sequence in $L^2(\partial\Omega_{\ell k})$ and $\mathscr{B}_{ij}\mathscr{K}_{\ell k}\widetilde{\mathscr K}_{\ell k}^{-1}$ are compact operators from $L^2(\partial\Omega_{\ell k})$ to $L^2(\partial\Omega_{ij})$. Thus, each sequence $\{\mathscr{B}_{ij}\mathscr{K}_{\ell k}\widetilde{\mathscr K}_{\ell k}^{-1}v_{\ell k}^{(n)}\}_{n=1}^{\infty}$ in $L^2(\partial\Omega_{ij})$ has a convergent subsequence if $\Omega_{ij}$ and $\Omega_{\ell k}$ are different scatterers. Denote the convergent subsequence by $\{v_{ij,\ell k}^{(n)}\}_{n=1}^{\infty}$.
	Then the sequence
	$$\hat v_{ij}^{(n)}=\sum_{k=1,k\neq j}^{M_i} v_{ij,ik}^{(n)}+\sum\limits_{\ell=1,\ell\neq k}^3\sum_{k=1}^{M_{\ell}} v_{ij,\ell k}^{(n)},\quad n=1, 2, \cdots$$
	defined as the finite sum of convergent sequences is a convergent subsequence of $\{\tilde{\bs v}^{(n)}\}_{n=1}^{\infty}$ in $\bs L^2(\partial\Omega)$. This complete the proof of the compactness of operator $\mathbb{K}$ on $\bm{L}^2(\partial\Omega)$.
\end{proof}

According to the spectral theorem for compact operator (cf. \cite{Yosida1978Functional}), $\mathcal{I}+\mathbb{K}$ has a countable sequence of eigenvalues with $1$ being the only possible accumulation point. This means the set of eigenvalues $\lambda_j$ for which $|\lambda_j-1|>\rho$ for any $\rho<1$ is finite. Then the convergence of GMRES iteration method for equation \eqref{conciseeq} can be concluded from the following result (cf. \cite{campbell1996gmres} Proposition 6.1): 
\begin{theorem}\label{gmresconvergence}
	Given a system of linear equations $(I+\mathscr {A})\bs w=\bm{b}$ where $\mathscr A$ is a compact linear operator. Let the eigenvalues of $\mathcal{I}+\mathscr{A}$ be numbered so that $|\lambda_j-1|\geqslant |\lambda_{j+1}-1|$, for $j\geqslant 1$. Given $\rho>0$, determine $0\leqslant M <\infty$ so that $\{\lambda_j\}_{j=1}^M\subset \{z: |z-1|>\rho\}$, are the outliers and  $\{\lambda_j\}_{j\geqslant M+1}\subset \{z: |z-1|<\rho\}$
	is cluster. Define the distance of the outliers from the cluster as
	$$
	\delta := \max_{|z-1|=\rho}\max_{1\leqslant j \leqslant M}\frac{|\lambda_j - z|}{|\lambda_j|}.
	$$
	Then for any $\bm{b}$, and $\bs w_0$
	$$
	\|\bs r_{d+k}\| \leqslant C_{\delta}\rho^k\|\bs r_0\|,
	$$
	where $\bs r_k$ is the residual generate by GMRES iteration at $k$th step, and the constant $C_\delta$ is independent of $k$.
\end{theorem}

Next, we consider the convergence of the proposed iterative method for the multiple scattering problem in locally inhomogeneous media. Following the same proof for homogeneous media case, we can verify that $\mathbb K^{\prime}$ is compact on Hilbert space
\begin{equation}
\bs L^2(\Gamma):=L^2(\Gamma^{11})\times\cdots\times L^2(\Gamma^{1M_1})\times\cdots\times L^2(\Gamma^{31})\times\cdots\times L^2(\Gamma^{3M_3}).
\end{equation}
Therefore, we now focus on the operator $\mathbb S^{\prime}\mathbb K^{\prime},$ which consists of $\mathscr{S}_{ij}^{\prime}\mathscr{T}_{ij}^{\prime}\mathscr{K}_{\ell k}^{\prime}\widetilde{\mathscr{K}}_{\ell k}^{\prime -1}$.
\begin{lemma}\label{compactboundintegral}
	Suppose $\Gamma^{ij}$ and $\Gamma^{\ell k}$ are different artificial boundaries. Then $(\mathscr K_{\ell k}^{\prime}\phi_{\ell k})(\bs x): L^2(\Gamma^{\ell k})\rightarrow H^{\frac{1}{2}}(\Gamma^{ij})$, $\bs x\in\Gamma^{ij}$ is a compact operator.
\end{lemma}
\begin{proof}
	Let $U$ be a bounded set in $L^2(\Gamma^{\ell k})$, i.e., $\|\phi_{\ell k}\|_{L^2(\Gamma^{\ell k})}\leq C$ for all $\phi_{\ell k}\in U$ and some $C>0$. Then
	\begin{equation}
	\begin{split}
	|(\mathscr K_{\ell k}^{\prime}\phi_{\ell k})(\bs x)|\leq &\int_{\Gamma^{\ell k}}\Big|\frac{\partial G_{\kappa}(\bm{x},\bm{y})}{\partial \bm{n}(\bm{y})}\Big| |\phi_{\ell k}(\bm{y})|\mathrm{d}S_{\bm{y}}+|\eta|\int_{\Gamma^{\ell k}}|G_{\kappa}(\bm{x},\bm{y})| |\phi_{\ell k}(\bm{y})|\mathrm{d}S_{\bm{y}}\\
	\leq & C|\Gamma^{\ell k}|^{\frac{1}{2}}\max\limits_{\bm{y}\in \Gamma^{\ell k}}\Big(\Big|\frac{\partial G_{\kappa}(\bm{x},\bm{y})}{\partial \bm{n}(\bm{y})}\Big|+|\eta||G_{\kappa}(\bm{x},\bm{y})|\Big),
	\end{split}
	\end{equation}
	for all $\bs x\in\Gamma^{ij}$ and all $\phi_{\ell k}\in U$, i.e., $\mathscr K_{\ell k}^{\prime}U$ is bounded in maximum norm. Since $\frac{\partial G_{\kappa}(\bm{x},\bm{y})}{\partial \bm{n}(\bm{y})}$ and $G_{\kappa}(\bm{x},\bm{y})$ are uniformly continuous on the compact set $\Gamma^{ij}\times\Gamma^{\ell k}$, for every $\varepsilon>0$, there exists $\delta >0$ such that
	\begin{equation}
	\Big|\frac{\partial G_{\kappa}(\bm{x},\bm{z})}{\partial \bm{n}(\bm{z})}-\frac{\partial G_{\kappa}(\bm{y},\bm{z})}{\partial \bm{n}(\bm{z})}\Big|\leq\frac{\varepsilon}{2C|\Gamma^{\ell k}|^{\frac{1}{2}}},\quad |G_{\kappa}(\bm{x},\bm{z})-G_{\kappa}(\bm{y},\bm{z})|\leq\frac{\varepsilon}{2C|\eta||\Gamma^{\ell k}|^{\frac{1}{2}}},
	\end{equation}
	for all $\bs x, \bs y\in \Gamma^{ij}$, $\bs z\in \Gamma^{\ell k}$ with $|\bs x-\bs y|<\delta$. Then
	\begin{equation}
	|(\mathscr K_{\ell k}^{\prime}\phi_{\ell k})(\bs x)-(\mathscr K_{\ell k}^{\prime}\phi_{\ell k})(\bs y)|\leq\varepsilon
	\end{equation}
	for all $\bs x, \bs y\in \Gamma^{ij}$ with $|\bs x-\bs y|<\delta$ and all $\phi_{\ell k}\in U$, i.e., $\mathscr K_{\ell k}^{\prime}U$ is equicontinuous. By the smoothness of the  Green's function $G_{\kappa}(\bs x, \bs y)$ for $\bs x\neq \bs y$, we can further prove that $\{\nabla(\mathscr K_{\ell k}^{\prime}\phi_{\ell k}): \phi_{\ell k}\in U\}$ is bounded and  equicontinuous in the same way. Therefore $(\mathscr K_{\ell k}^{\prime}\phi_{\ell k})(\bs x): L^2({\Gamma^{\ell k}})\rightarrow C^1(\Gamma^{ij})$ are compact. Then the statement of this  lemma follows from the facts that $C^1(\Gamma^{ij})$ is dense in $H^{\frac{1}{2}}(\Gamma^{ij})$ and $C^1$-norm is stronger than $H^{\frac{1}{2}}$-norm.
\end{proof}

Together with the well-posedness of the reduced boundary value problem \eqref{localgeneralinterior}, we can draw the conclusion on the compactness of $\mathscr{S}_{ij}^{\prime}\mathscr{T}_{ij}^{\prime}\mathscr{K}_{\ell k}^{\prime}\widetilde{\mathscr{K}}_{\ell k}^{\prime -1}$.
\begin{theorem}\label{compactall}
	Suppose $\Gamma^{ij}$ and $\Gamma^{\ell k}$ are different artificial boundaries, then $$(\mathscr{S}_{ij}^{\prime}\mathscr{T}_{ij}^{\prime}\mathscr{K}_{\ell k}^{\prime}\widetilde{\mathscr{K}}_{\ell k}^{\prime -1}W_{\ell k})(\bs x): L^2(\Gamma^{\ell k})\rightarrow H^{1}(B_{ij}\setminus\Omega_{ij})$$
	 is a compact operator.
\end{theorem}
\begin{proof}
	According to the trace theorem in Sobolev space, the DtN operator $\mathscr{T}_{ij}^{\prime}$ is a bounded linear operator from $H^{\frac{1}{2}}(\Gamma^{ij})$ to $H^{-\frac{1}{2}}(\Gamma^{ij})$ (cf. \cite{nedelec2001acoustic,melenk2010convergence}). Moreover, the well-posedness of the boundary value problem \eqref{localgeneralinterior} implies that $\mathscr S^{\prime}_{ij}$ is a bounded operator from $H^{-\frac{1}{2}}(\Gamma^{ij})$ to $H^1(B_{ij}\setminus\Omega_{ij})$. Hence $\mathscr{S}_{ij}^{\prime}\mathscr{T}_{ij}^{\prime}$ is bounded from $H^{\frac{1}{2}}(\Gamma^{ij})$ to $H^1(B_{ij}\setminus\Omega_{ij})$.
	
	Note that $\widehat{\mathscr K}_{\ell k}^{\prime}$ is an integral operator with weakly singular kernel, i.e., compact on $L^2(\Gamma^{\ell k})$. By Theorem \ref{boundedinverse}, we conclude that $\widetilde{\mathscr{K}}_{\ell k}^{\prime -1}=\big(\frac{1}{2}\mathcal I+\widehat{\mathscr K}_{\ell k}^{\prime}\big)^{-1}$ is a bounded operator on $L^2(\Gamma^{\ell k})$. Together with the boundedness of $\mathscr{S}_{ij}^{\prime}\mathscr{T}_{ij}^{\prime}$ in Theorem \ref{interiorwellposed} and Lemma \ref{compactboundintegral}, we complete the proof.
\end{proof}

With the compactness of operators $\mathscr{S}_{ij}^{\prime}\mathscr{T}_{ij}^{\prime}\mathscr{K}_{\ell k}^{\prime}\widetilde{\mathscr{K}}_{\ell k}^{\prime -1}$, it is not difficult to verify the compactness of $\mathbb S^{\prime}\mathbb K^{\prime}$ on $\bs L^2(\Gamma)$ by following the same proof in Theorem \ref{compactnessinvecspace}. Then the convergence of the iterative method for equation \eqref{conciseeqheter} is ensured by Theorem \ref{gmresconvergence}.

\section{Numerical Examples}\label{sect5}
In this section, numerical examples are presented to show the performance of the proposed iterative algorithms.  The shape of the scatterers is  determined by the parametric form of the boundary curve: 
\begin{equation}\label{scattererpara}
r_i=a\sin k(\theta_i-\theta_0)+b,\quad \theta\in [0, 2\pi],
\end{equation}
where $(r_i,\theta_i)$ is the polar coordinate of $\bs x$ with respect to a given center $\bs c_i$.  In all experiments, we take the  plane wave $e^{{\ri}\kappa y}$ as the incident wave.

\subsection{Homogeneous media}
{\bf Example 1:} We first test the accuracy of {\bf Algorithm 1}. Consider two scatterers determined by \eqref{scattererpara} with $k=2, a=0.3, b=0.7, \theta_0 = \pi/4$, $\bs c_1(0, 0)$ and $\bs c_2(2.6, 0)$. The GMRES iteration is set to stop at the tolerance  $1.0\mathrm{e}$-11. Since the exact solution is not available, we use the numerical solution computed by spectral element discretization with polynomial of degree $p= 40$ as reference solution $u^{\rm ref}$. In the computation of the reference solution, we use a large artificial boundary to enclose the two scatterers inside (see Fig. \ref{multiscatter} (Right)) and impose non-reflecting boundary condition on it. Instead, small artificial boundaries are used for the iterative method, see Fig. \ref{multiscatter} (Left).

\begin{figure}[ht!]
	\centering
	\includegraphics[width=0.35\textwidth]{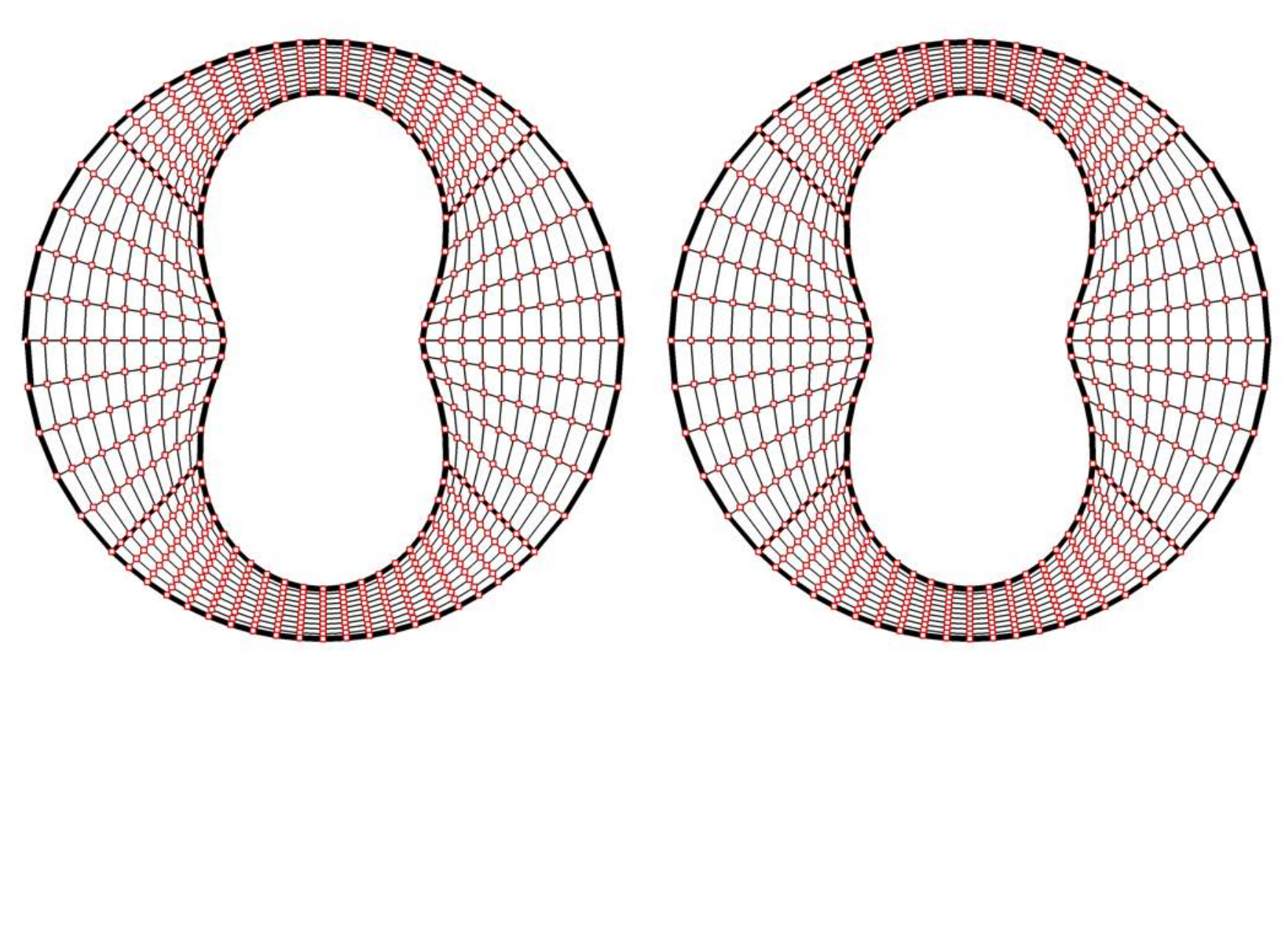}\qquad\qquad
	\includegraphics[width=0.35\textwidth]{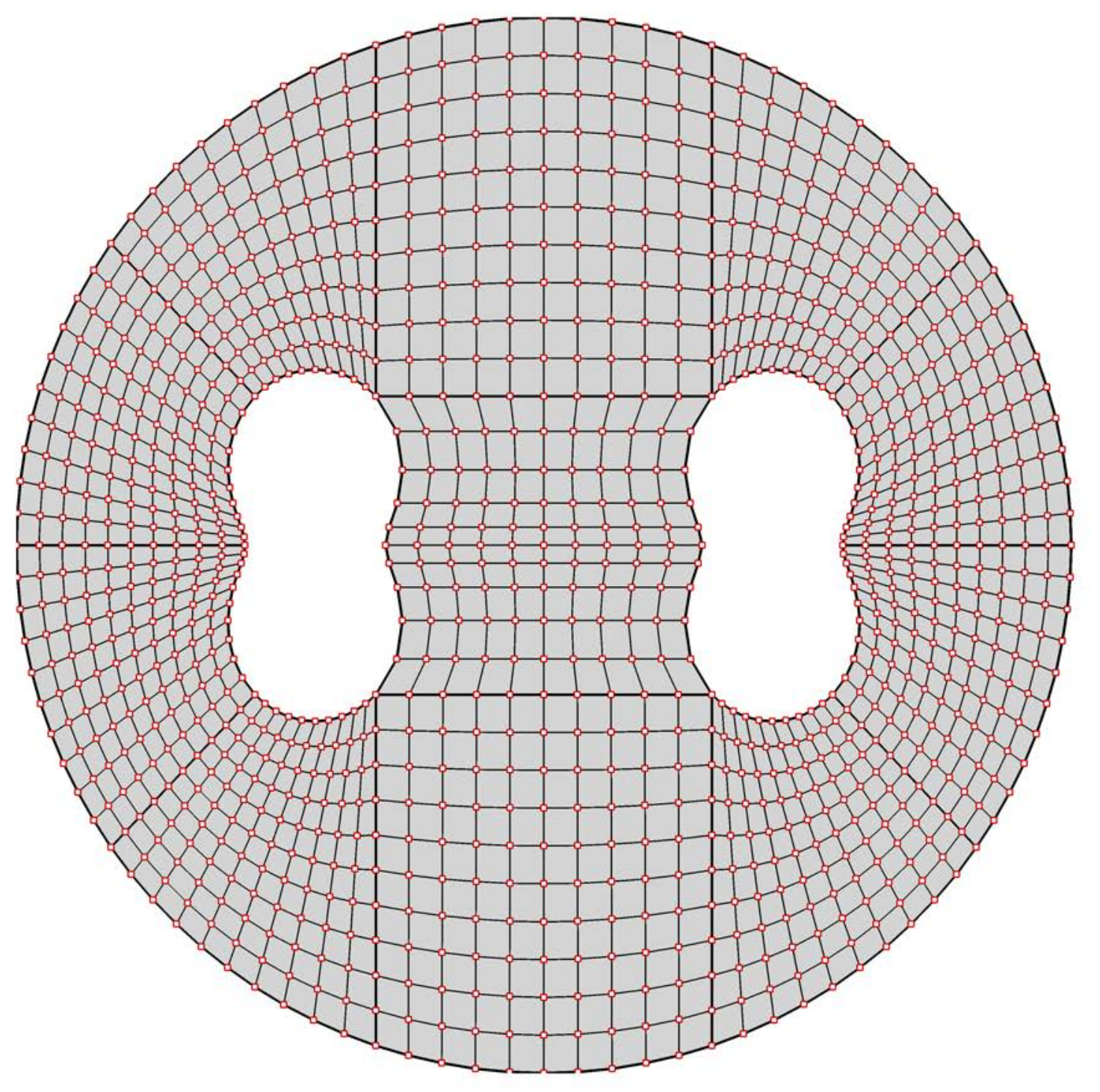}
	\caption {Left: Artificial boundaries and spectral element grid for iterative method, Right: Artificial boundary and spectral element grid for reference solution.}
	\label{multiscatter}
\end{figure}
\begin{figure}[ht!]
	\centering
	\subfigure[$\kappa=10$]{\includegraphics[scale=0.27]{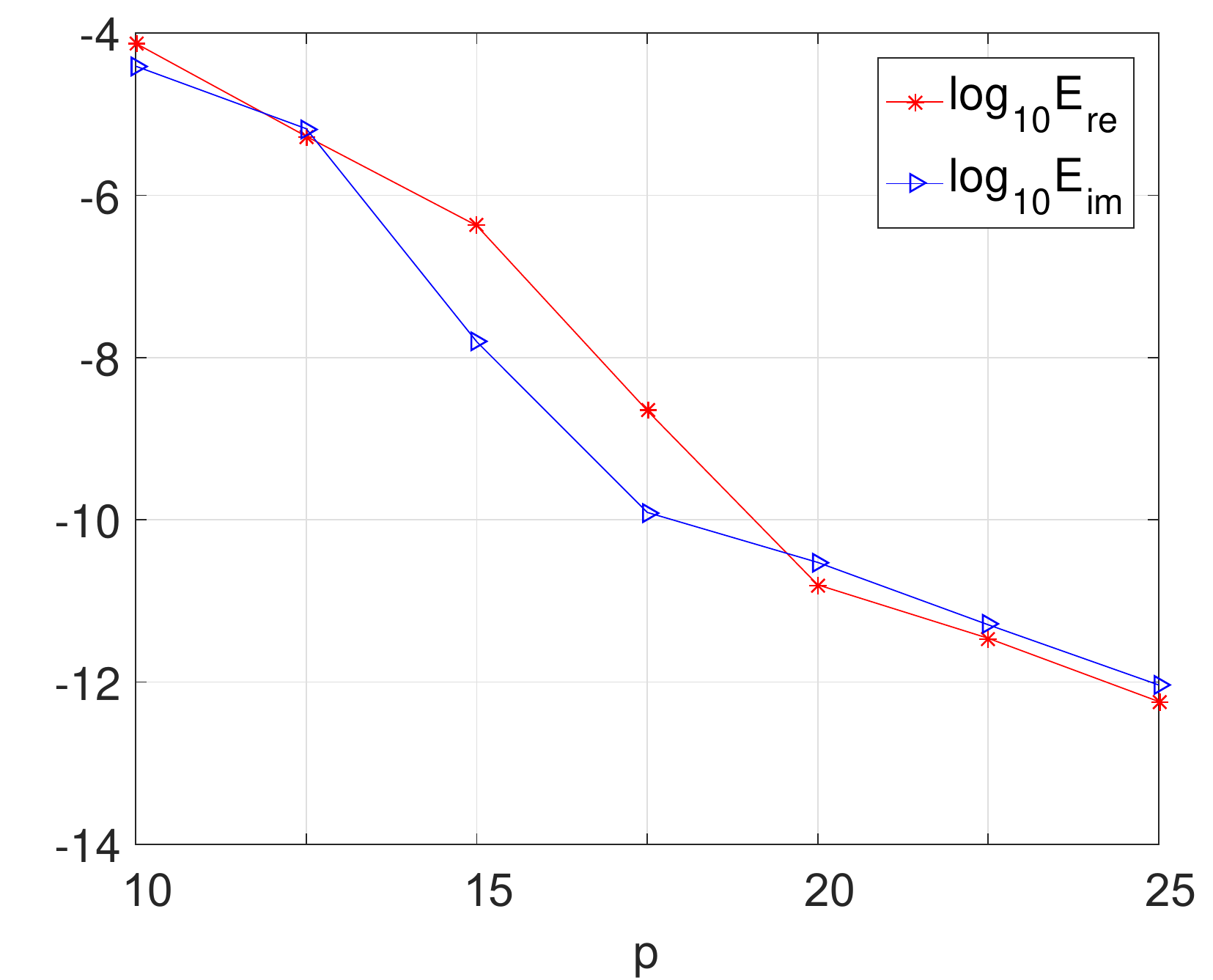}}
	\subfigure[$\kappa=20$]{\includegraphics[scale=0.27]{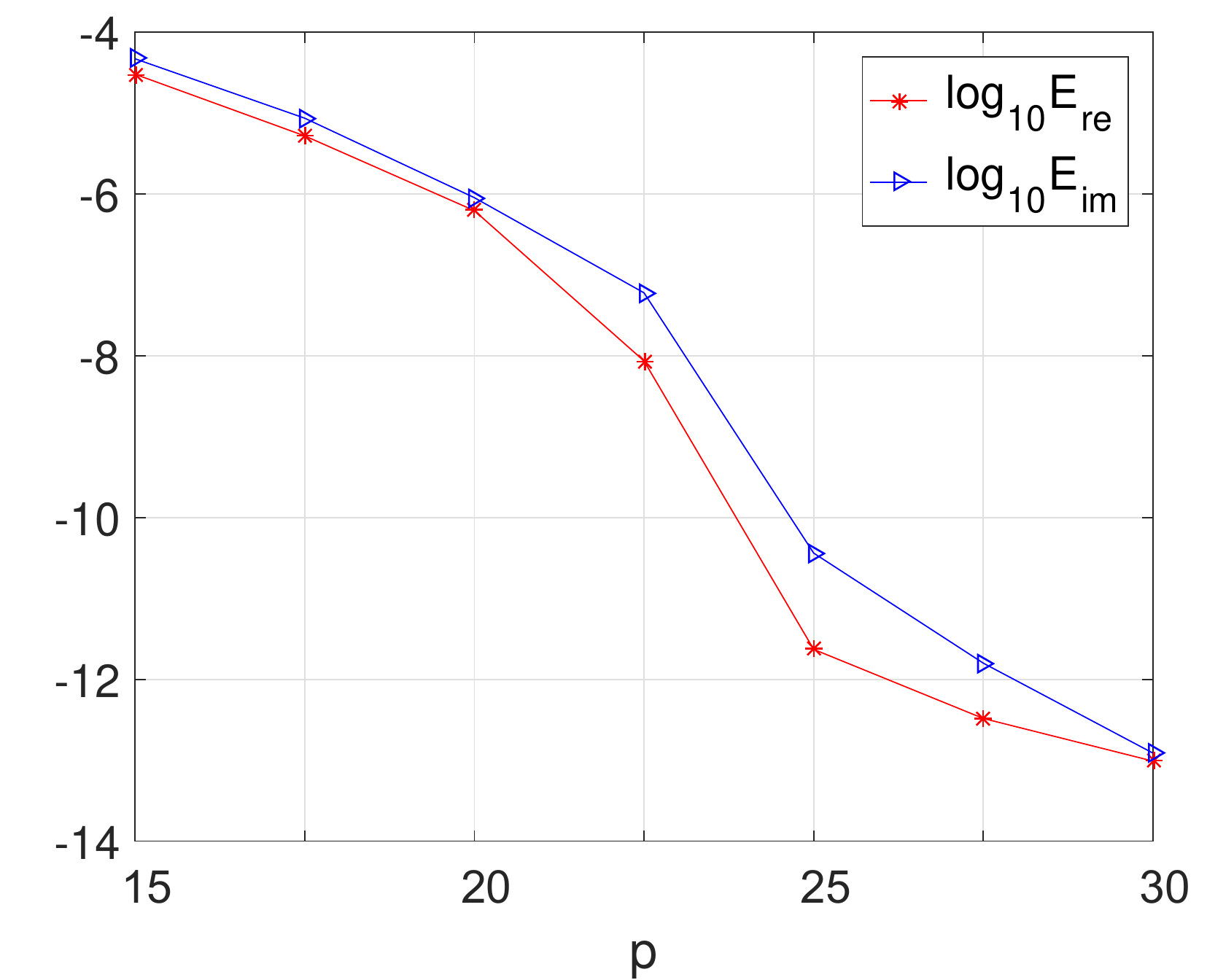}}
	\subfigure[$\kappa=30$]{\includegraphics[scale=0.27]{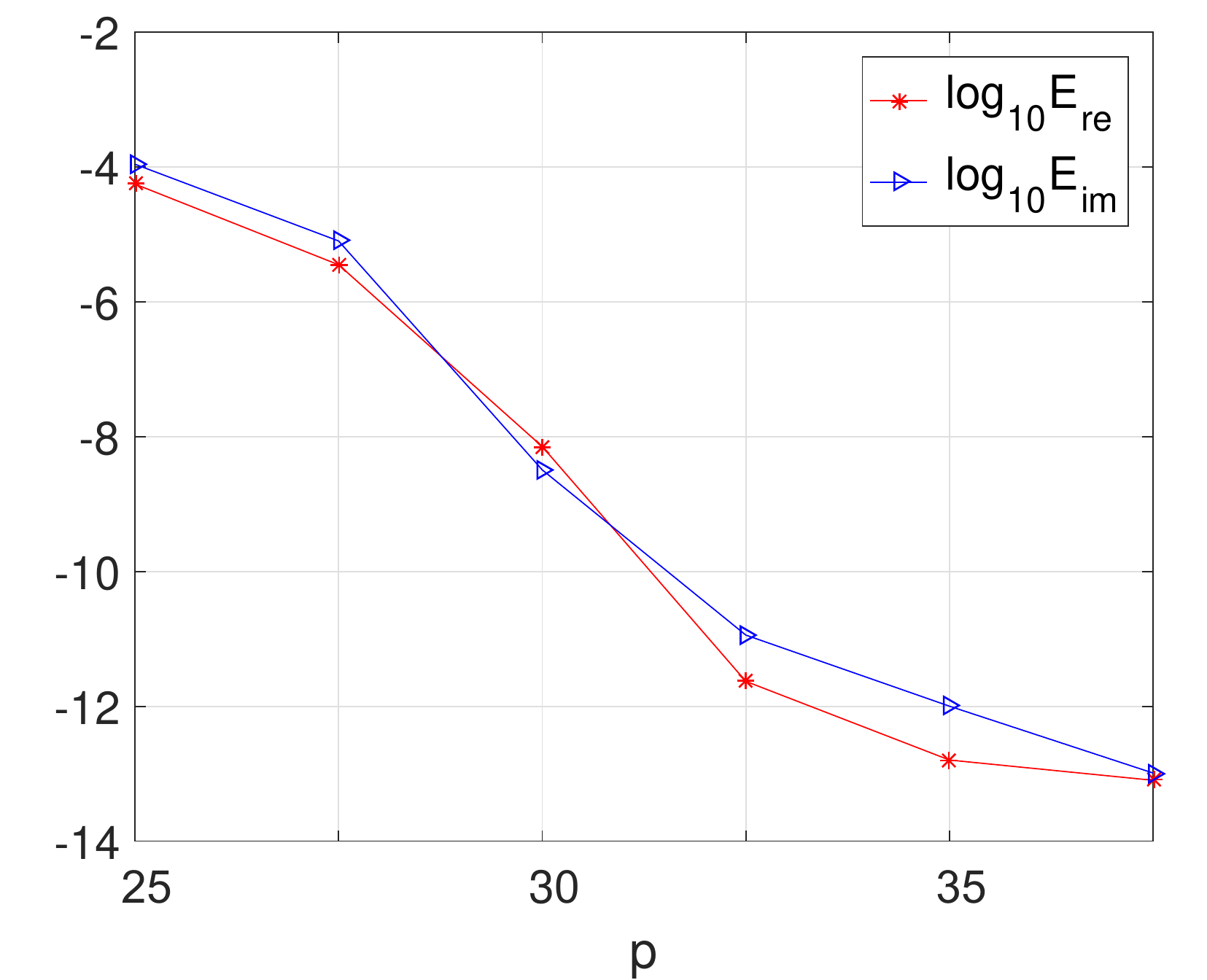}}
	\caption{Convergence rates in $L^2$-norm against polynomial degree $p$.} \label{L2errorplot}
\end{figure}
\begin{figure}[ht!]
	\centering
	\subfigure[$\kappa=10$]{\includegraphics[scale=0.27]{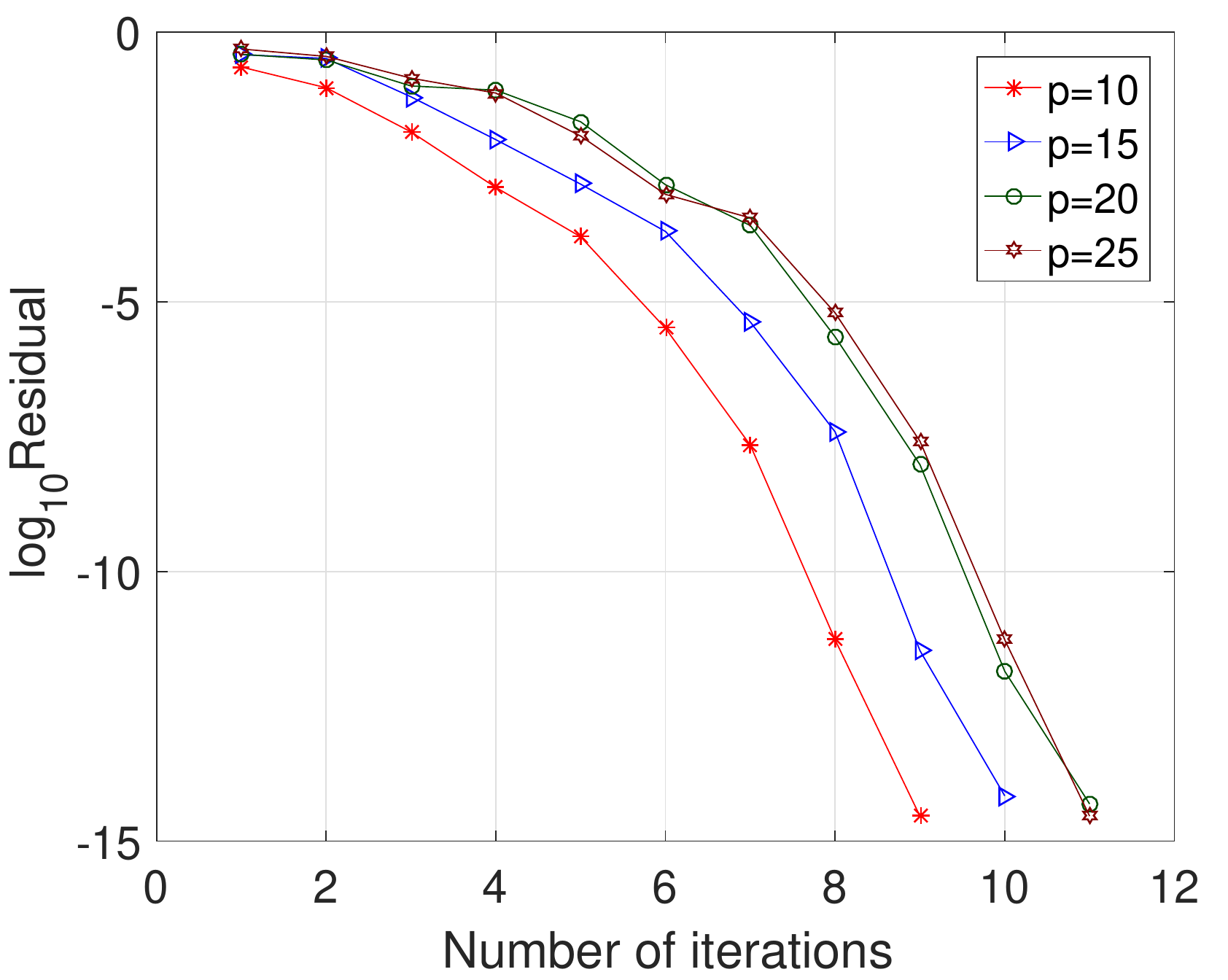}}
	\subfigure[$\kappa=20$]{\includegraphics[scale=0.27]{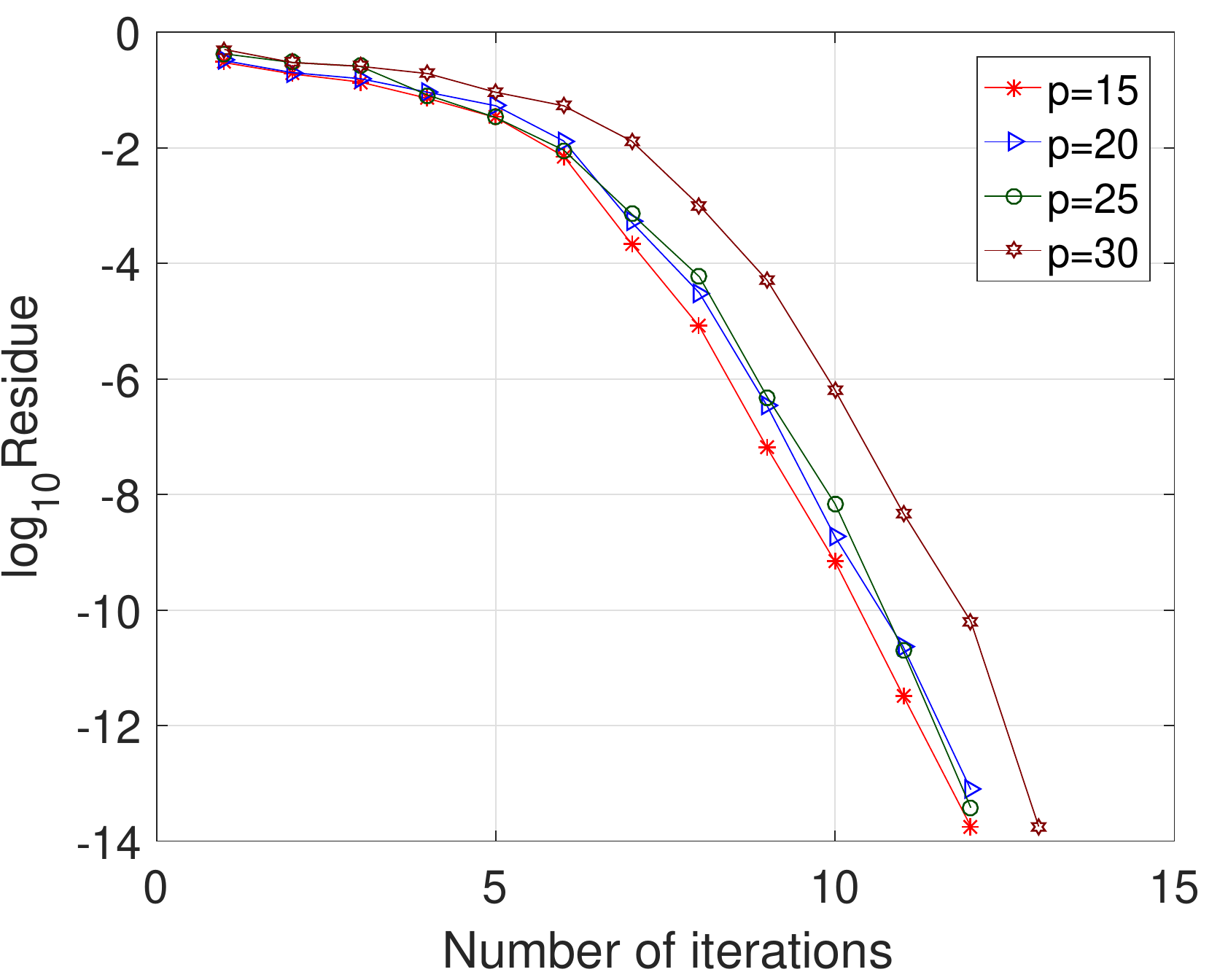}}
	\subfigure[$\kappa=30$]{\includegraphics[scale=0.27]{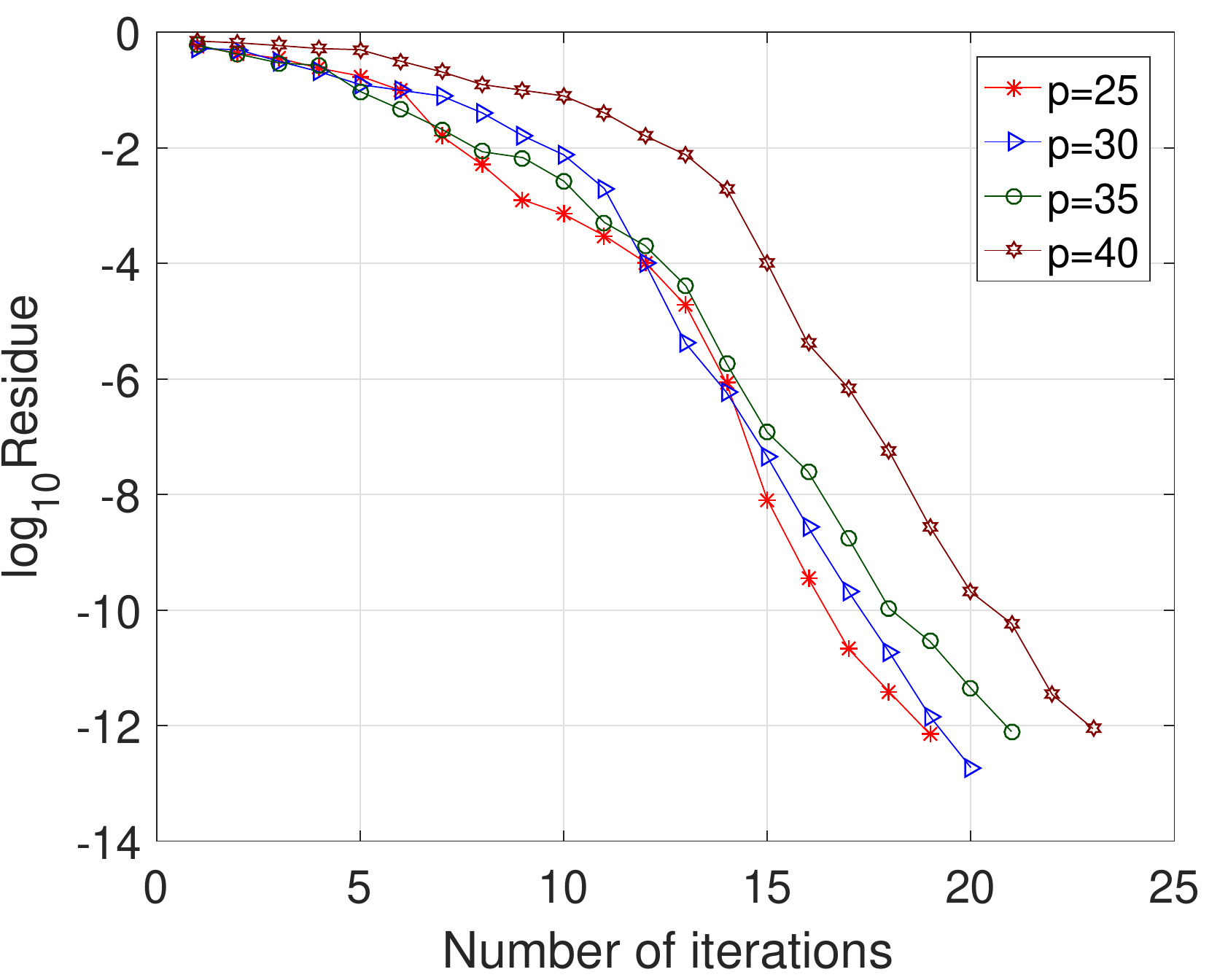}}
	\caption{Residuals against the number of iterations (homogeneous media).} \label{numberofiter}
\end{figure}

The approximate scattering field with polynomial degree $p=20$ for the case $\kappa=20$ is compared with the reference solution in Fig. \ref{scatteringwave2scatterersre} and Fig. \ref{scatteringwave2scatterersim}. Convergence rates in $L^2$-norm for cases with wavenumber $\kappa=10, 20, 30$ are plotted in Fig.  \ref{L2errorplot}. It shows that the iterative method has spectral accuracy with respect to polynomial degree $p$. In addition, we plot in Fig. \ref{numberofiter} the residuals against the number of iterations for different wave number $\kappa$ and polynomial degree $p$.  Clearly, we see that residuals achieve the machine accuracy in almost the same number of  iterations for different polynomial degree $p=10, 15, 20, 30$. That means the condition number of the iterative method is nearly independent of the degree of freedom used in the spectral element discretization.

To compare with the numerical method proposed in \cite{grote2004dirichlet}, we also adopt the SEM to discretize  the truncated problem \eqref{trun2dmscatprob} and obtain the linear system \eqref{blocksystem}. Then, the GMRES and block GMRES iterative method are applied to solve it. The iterations are set to stop at residual less than $1.0\mathrm{e}$-11.  We compare the number of iterations required by different methods in Table \ref{table1}. The numerical results show that our iterative method requires fewer iterations than numerical method proposed in \cite{grote2004dirichlet} combined with block GMRES iterative method for the resulted linear system. This implies that the use of purely outgoing components of the scattering field for the communication between scatterers is more efficient.
\begin{figure}[ht!]
	\centering
	\subfigure[iterative solution]{\includegraphics[scale=0.25]{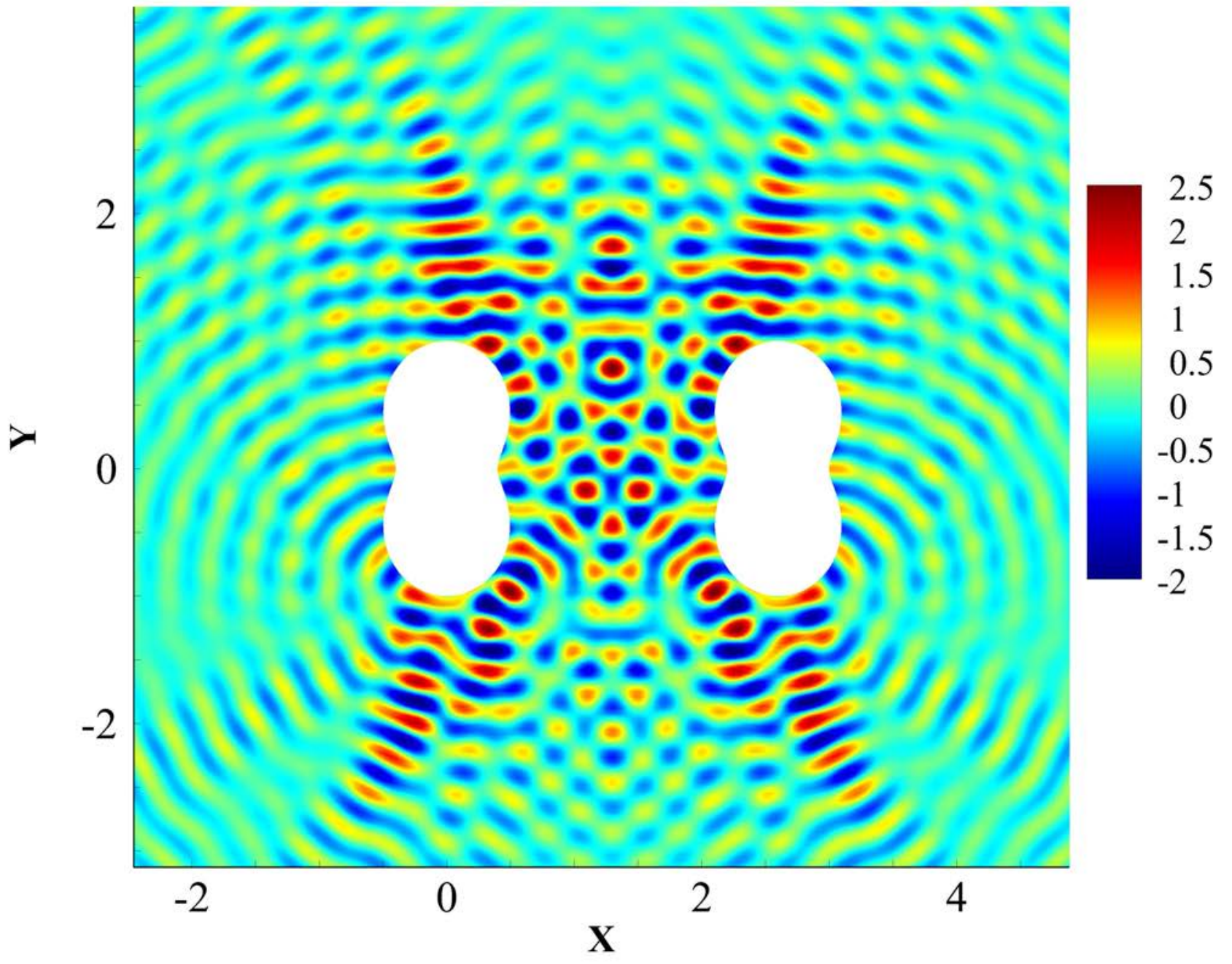}}\quad
	\subfigure[reference solution]{\includegraphics[scale=0.25]{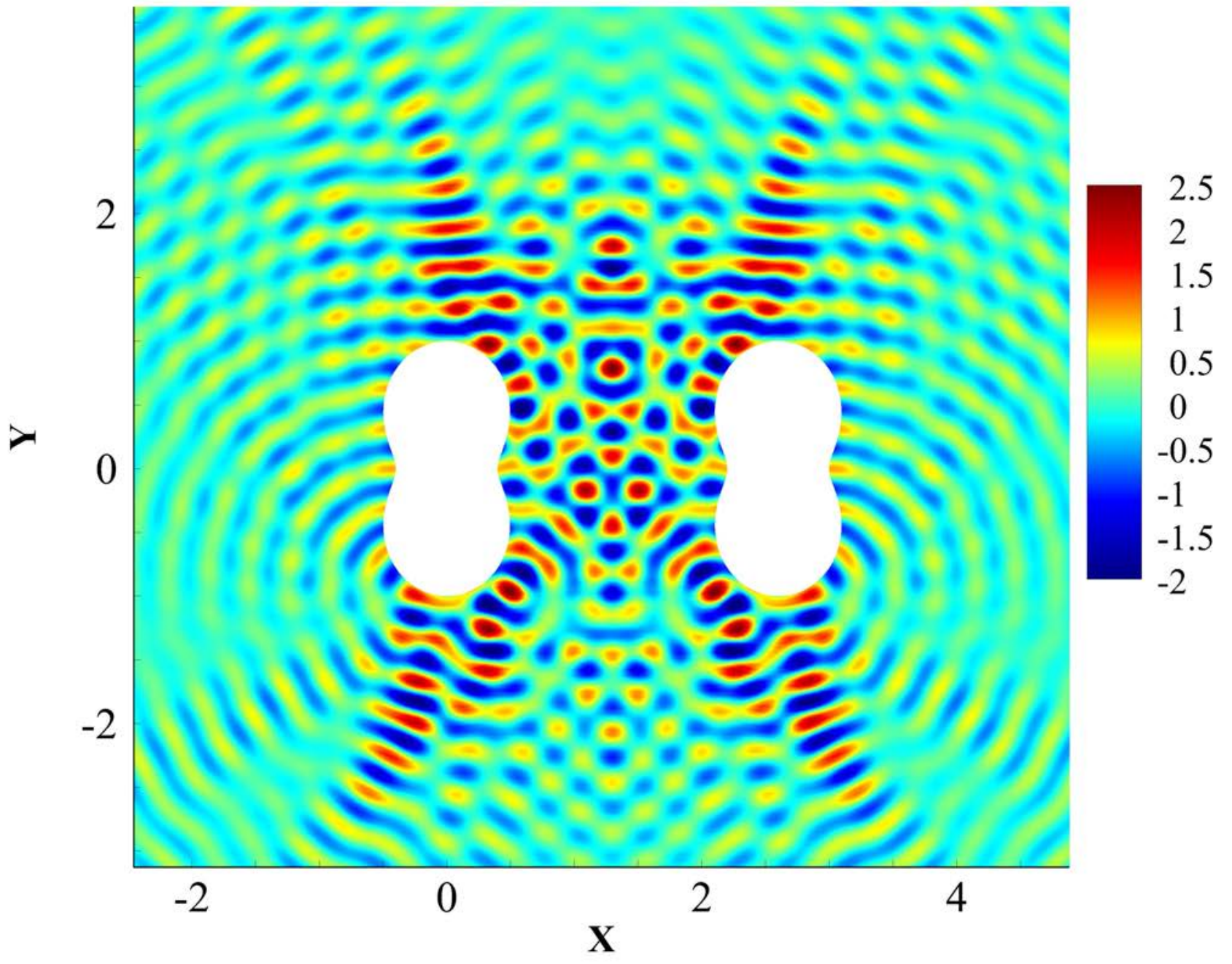}}
	\subfigure[error]{\includegraphics[scale=0.25]{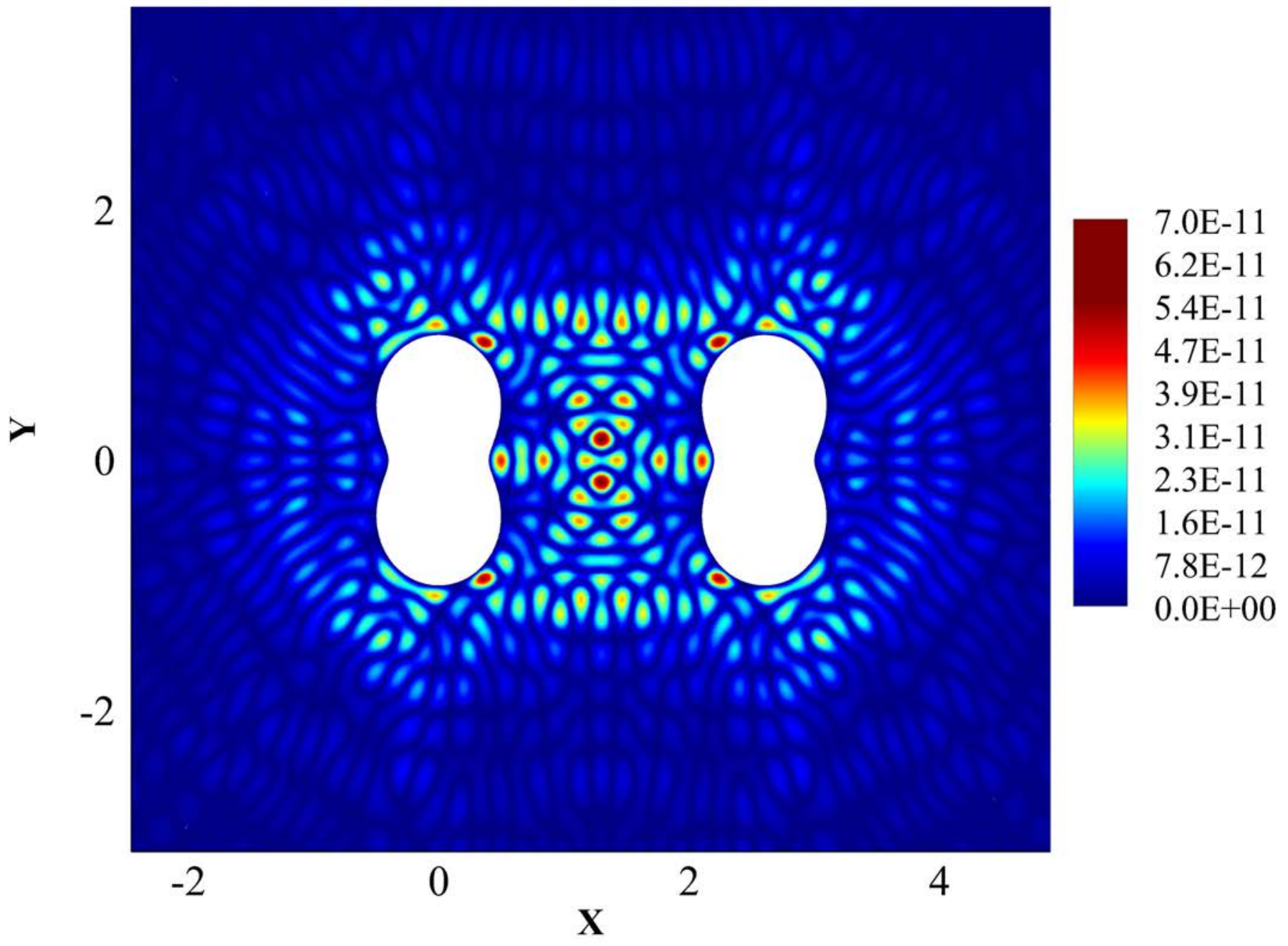}}
	\caption{Real parts of iterative numerical solution ($p=20$), reference solution and error for $\kappa=20$.}\label{scatteringwave2scatterersre}
\end{figure}
\begin{figure}[ht!]
	\centering
	\subfigure[iterative solution]{\includegraphics[scale=0.25]{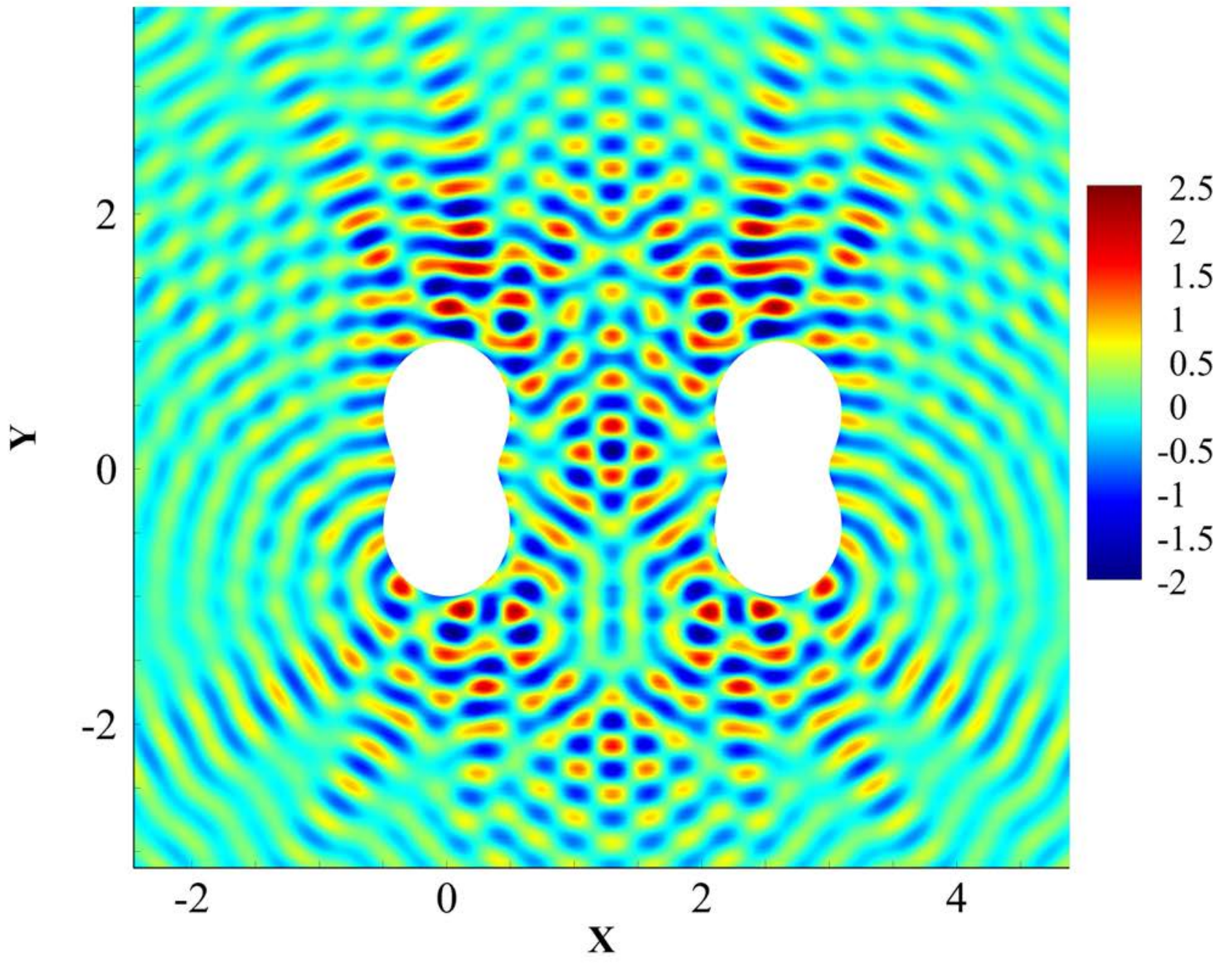}}\quad
	\subfigure[reference solution]{\includegraphics[scale=0.25]{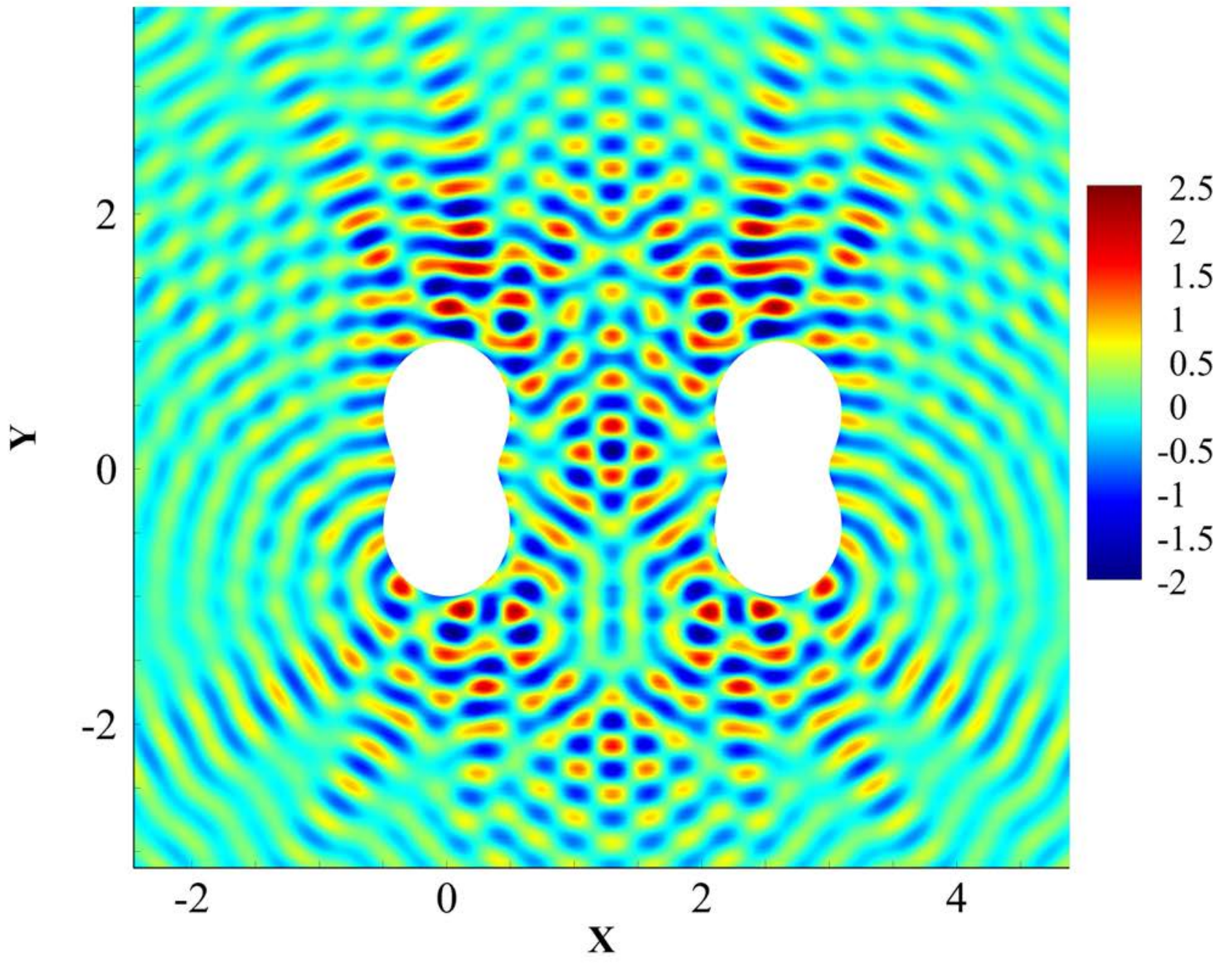}}\quad
	\subfigure[error]{\includegraphics[scale=0.25]{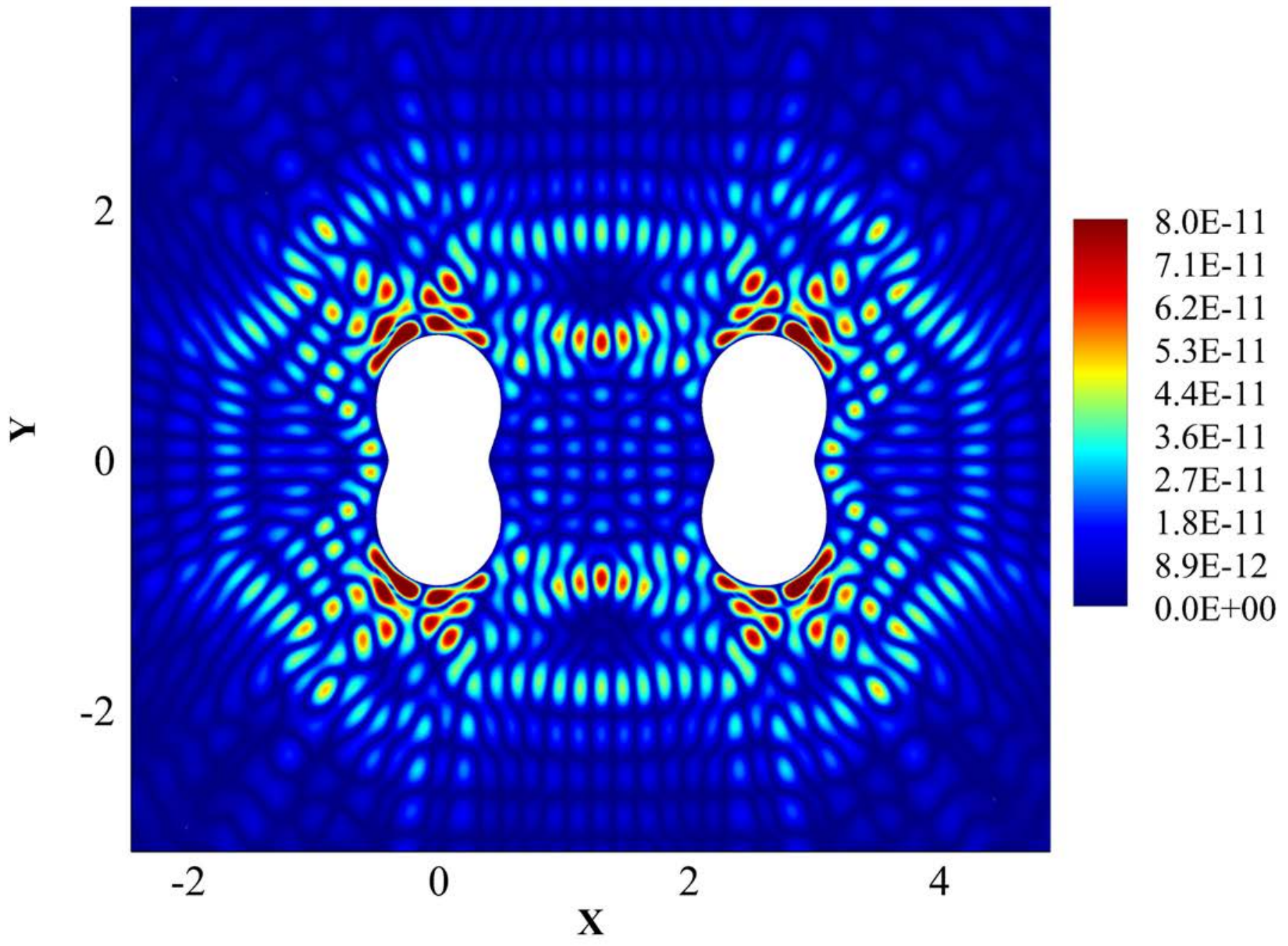}}
	\caption{Imaginary parts of iterative numerical solution ($p=20$), reference solution and error for $\kappa=20$.}\label{scatteringwave2scatterersim}
\end{figure}
{\color{blue}
\begin{table}[ht!]
	\center
	\begin{tabular}{|c|c|c|c|c|}
		\hline
		\multirow{2}{*}{$\kappa $} & \multirow{2}{*}{$p $} & \multicolumn{3}{c|}{Number of iterations} \\\cline{3-5}
		& & GMRES for \eqref{blocksystem} & Block GMRES for \eqref{blocksystem}& Our iterative algorithm\\
		\hline
		\multirow{4}{*}{$10$}  & $10$ & $554$ & $97$ & $9$ \\\cline{2-5}
		& $15$ & $1400$ & $106$ & $10$ \\\cline{2-5}
		& $20$ & $2744$ & $112$  & $11$  \\\cline{2-5}
		& $25$ & $4328$  & $117$  &  $11$ \\
		\hline
		\multirow{4}{*}{$20$}  & $15$ &$1009$  &$127$  & $12$ \\\cline{2-5}
		& $20$ & $1686$ & $147$ & $13$ \\\cline{2-5}
		& $25$ & $3807$ & $153$ & $13$ \\\cline{2-5}
		& $30$ & $5381$ & $208$ & $14$ \\
		\hline
	\end{tabular}
	\caption{The number of iterations  using different numerical methods for multiple scattering problem in homogeneous media.}\label{table1}
\end{table}}

{\bf Example 2:} As already discussed in Remark \ref{remark1}, our iterative method is able to solve the multiple scattering problem with the scatterers being  not well-separated. In this example, we consider two scatterers which are close to each other. The parametric expressions of the scatterers are  given by \eqref{scattererpara} with $k=2, a=0.3, b=0.7, \theta_0 = \pi/4$. The centers of the scatterers are set to be  $\bm{c}_1(0, 0)$ and $\bm{c}_2(1.1, 0.5)$. GMRES iteration is set to stop at residual less than $1.0\mathrm{e}$-11. Highly accurate approximation of the real part of the scattering field for the case $\kappa=20$ is plotted in Fig. \ref{multiscatterer49} (a). 
\smallskip 

{\bf Example 3:} Consider the multiple scattering problem with a large number of scatterers determined by \eqref{scattererpara} with $k = 5, a=0.2, b=0.7, \theta_0=0$. An array of sound soft (Dirichlet boundary condition) scatterers with centers located at the grid points $\{(2.2n, 2.2m)\}_{n,m=0}^6$ are tested. The real part of approximate scattering field is plotted in Fig. \ref{multiscatterer49} (b).
\begin{figure}[ht!]
	\centering
	\subfigure[2 not well separated scatterers ($p=20$, $\kappa=20$)]{\includegraphics[scale=0.39]{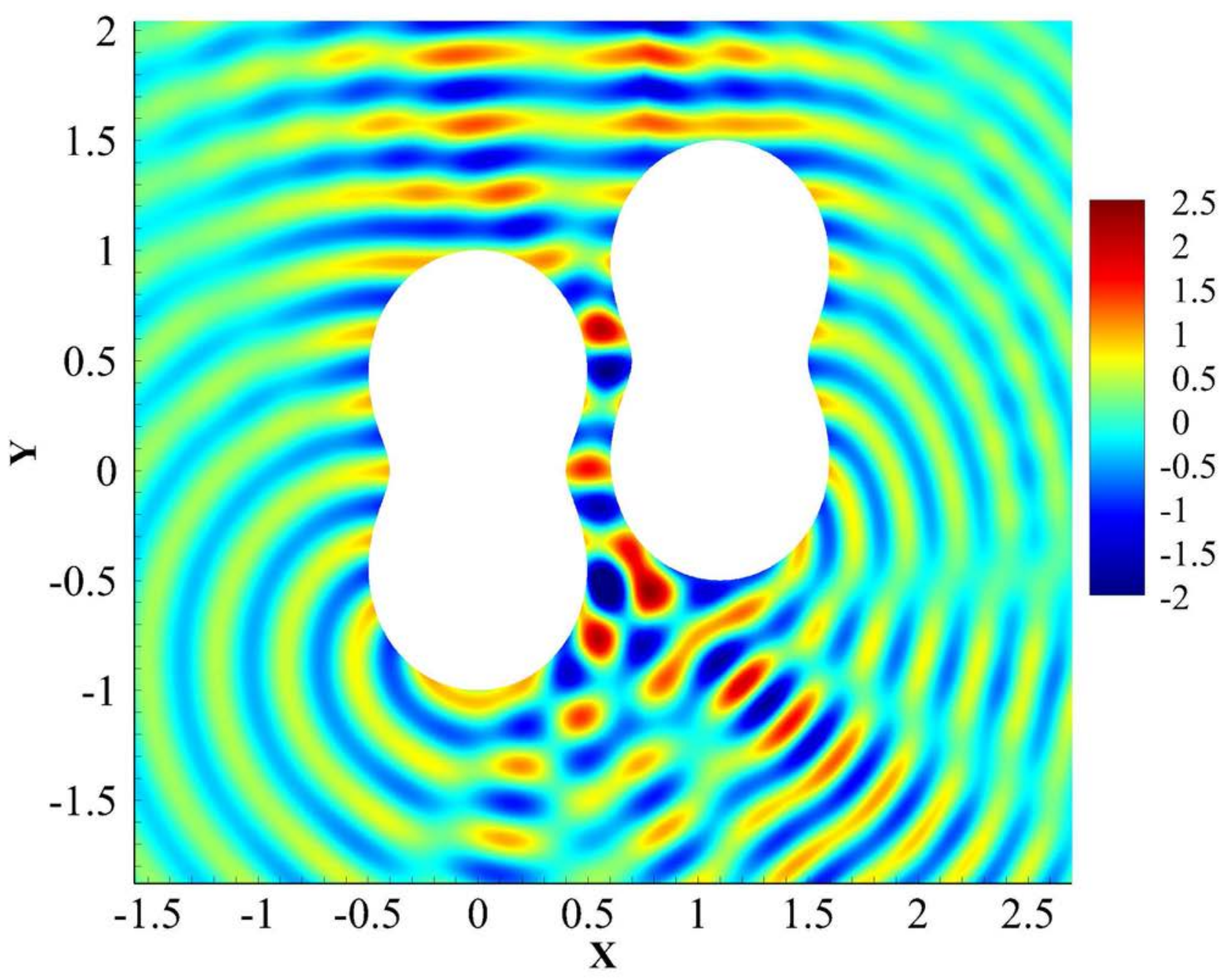}}\qquad
	\subfigure[49 well separated scatterers ($p=15$, $\kappa=10$) ]{\includegraphics[scale=0.34]{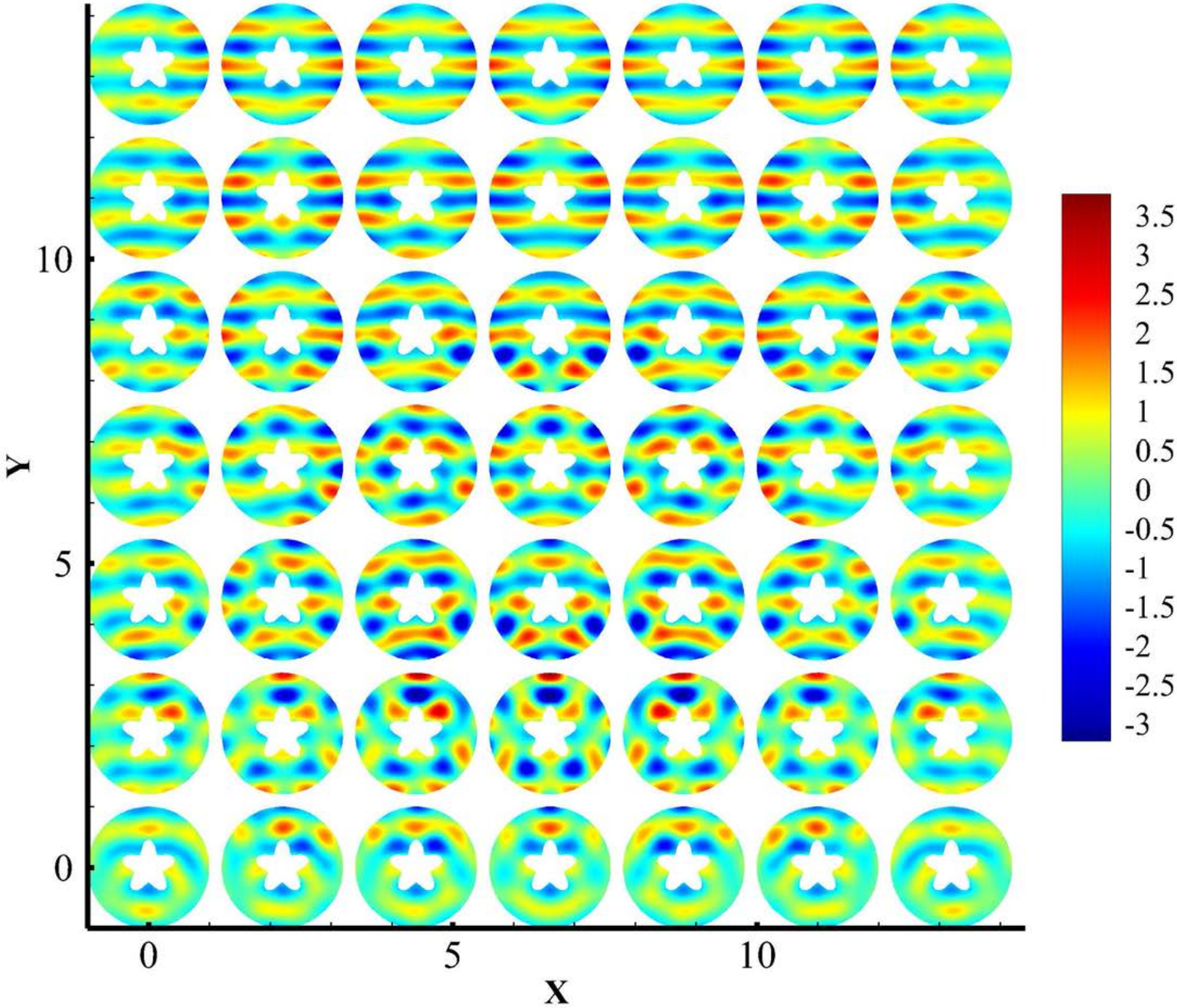}}
	\caption{Real parts of the approximate scattering fields due to sound soft scatterers.}\label{multiscatterer49}
\end{figure}

{\bf Example 4:} The scatterers can have different shapes and be randomly distributed. We plot the approximate scattering field due to 16 randomly distributed sound soft scatterers in Fig. \ref{multiscattererrandom8k20} (a). We also test  the problem with sound hard (Neumann boundary condition) scatterers, see the approximate scattering field plotted in Fig. \ref{multiscattererrandom8k20} (b).
\begin{figure}[ht!]
	\centering
	\subfigure[sound soft scatterers]{\includegraphics[scale=0.36]{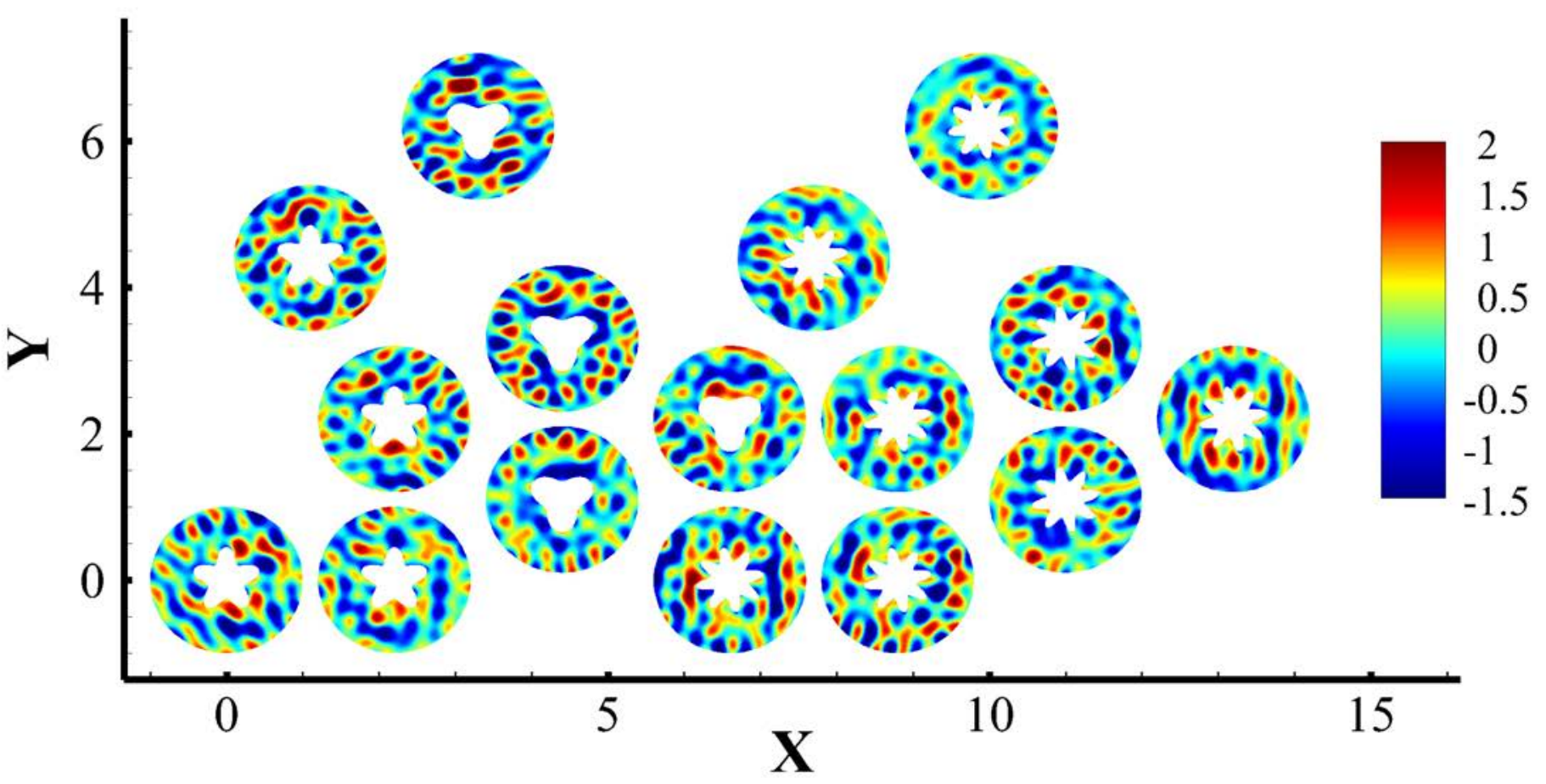}}\qquad
	\subfigure[sound hard scatterers]{\includegraphics[scale=0.36]{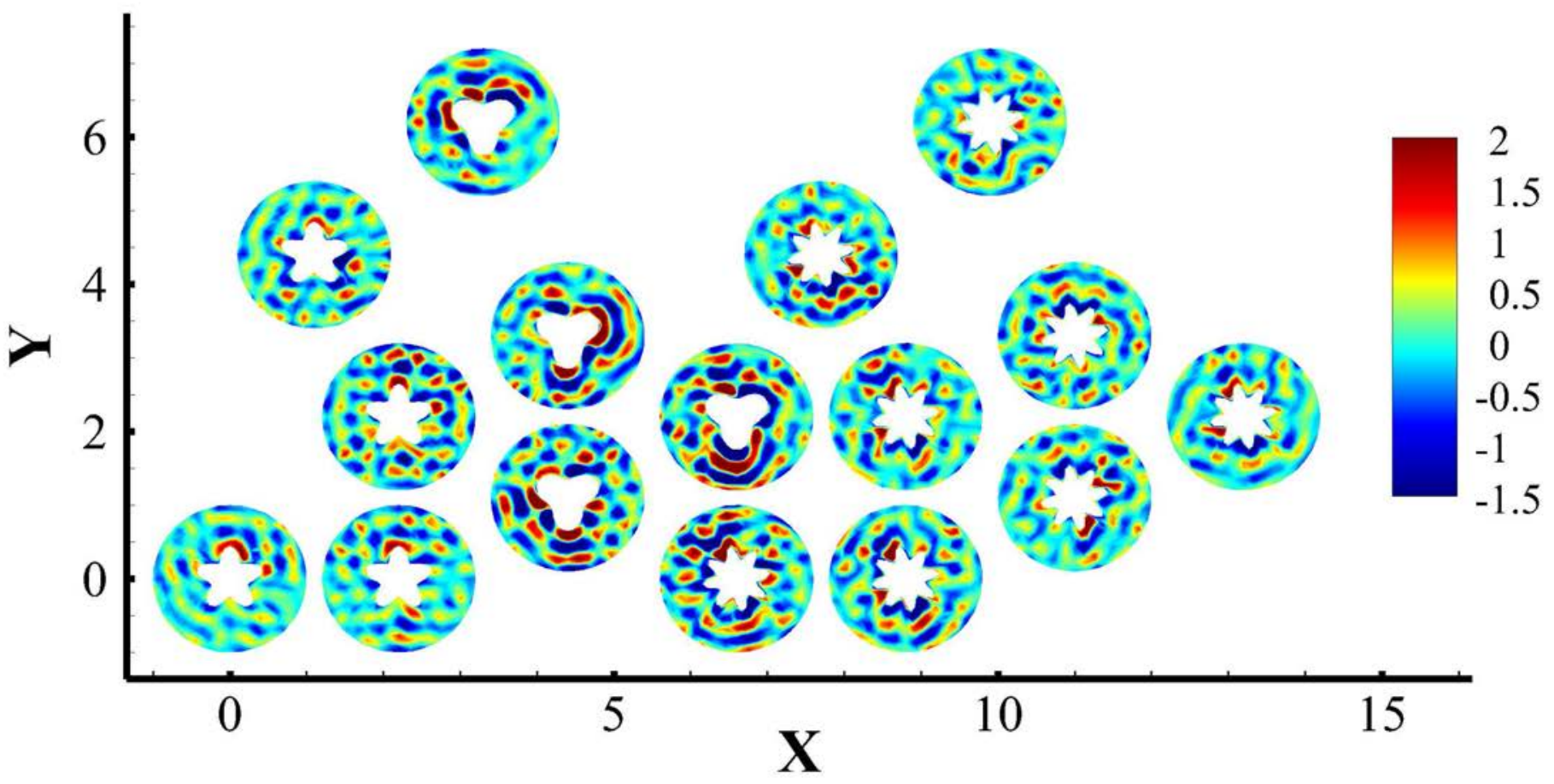}}
	\caption{Real parts of the approximate scattering fields ($p=20$) due to $16$  randomly distributed scatterers ($\kappa=20$).}\label{multiscattererrandom8k20}
\end{figure}

\subsection{Locally inhomogeneous media}
{\bf Example 5:} For accuracy test of {\bf Algorithm 2}, we consider the same two scatterers problem used in Example 1. All other settings are exactly the same as used in Example 1 except the locally inhomogeneous refraction index
\begin{equation}\label{heterrefractionindex1}
n(\bs x)=
\begin{cases}
{\rm exp}(-1/(1-16(|\bs x-\bs c_i|-1)^2))+1, &1.0<|\bs x-\bs c_i|<1.25,\\
1,& \mathrm{otherwise}.
\end{cases}
\end{equation}
It is a function of $|\bs x-\bs c_i|$ in the vicinity of the scatterer centered at $c_i$, see Fig. \ref{nr} (a).   $L^2$-errors of the numerical solutions and corresponding convergence rates are plotted in Fig. \ref{L2errorplotinhomo} and an approximate scattering field with $p=20$ for the case $\kappa=20$ are compared with reference solution in Figs. \ref{scatteringwave2scatterersreINHOMO} (real part). Results presented in Fig. \ref{L2errorplotinhomo} also show that the iterative method has spectral accuracy. From the decaying rates of residuals plotted in Fig. \ref{residueinhomo}, we see that they have similar decaying rates for different polynomial degree $p=10, 15, 20, 30$ in all tests. Further, the convergence rates of our iterative method and block GMRES iterative method together with numerical discretization proposed in \cite{grote2004dirichlet} are compared in Table \ref{table2}. All the iterations are set to stop at residual less than $1.0\mathrm{e}$-11 as before. As in the homogeneous media case, our iterative method requires much fewer iterations to achieve the given accuracy,  which further validates the fact that the way of using purely outgoing components of the scattering field for the communication between scatterers is more efficient.
\begin{figure}[ht!]
	\centering
	\subfigure[]{\includegraphics[scale=0.28]{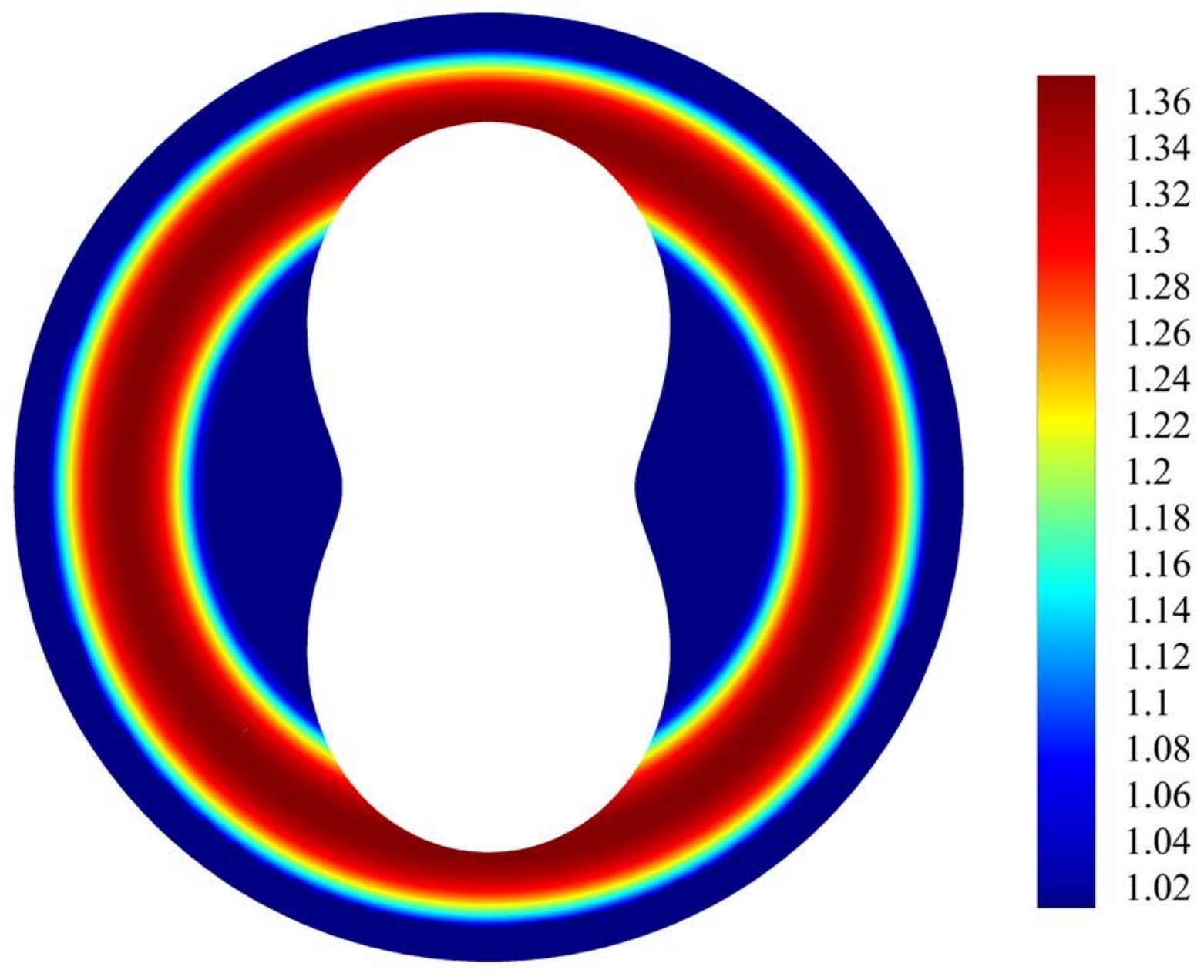}}\quad
	\subfigure[]{\includegraphics[scale=0.28]{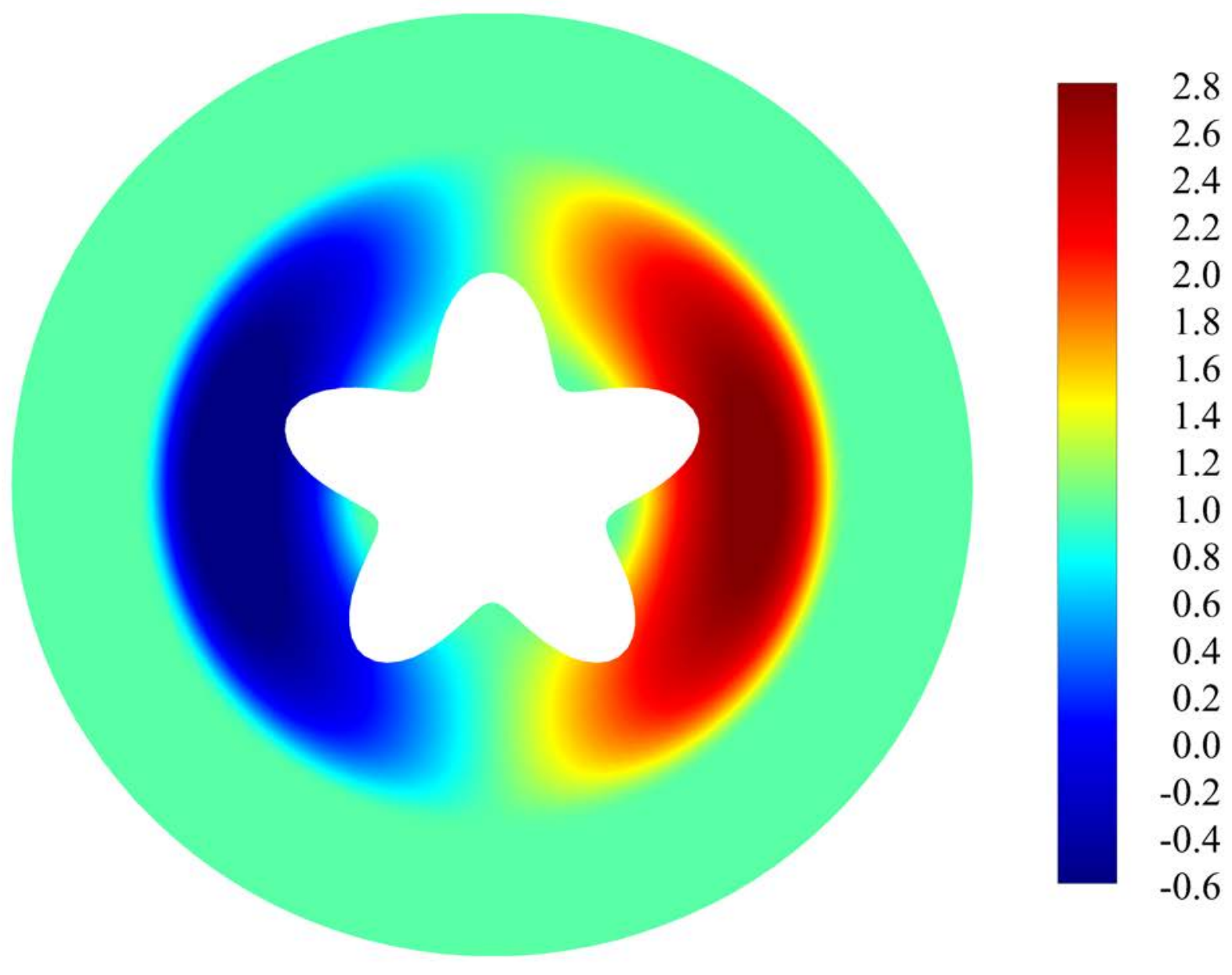}}\quad
	\subfigure[]{\includegraphics[scale=0.28]{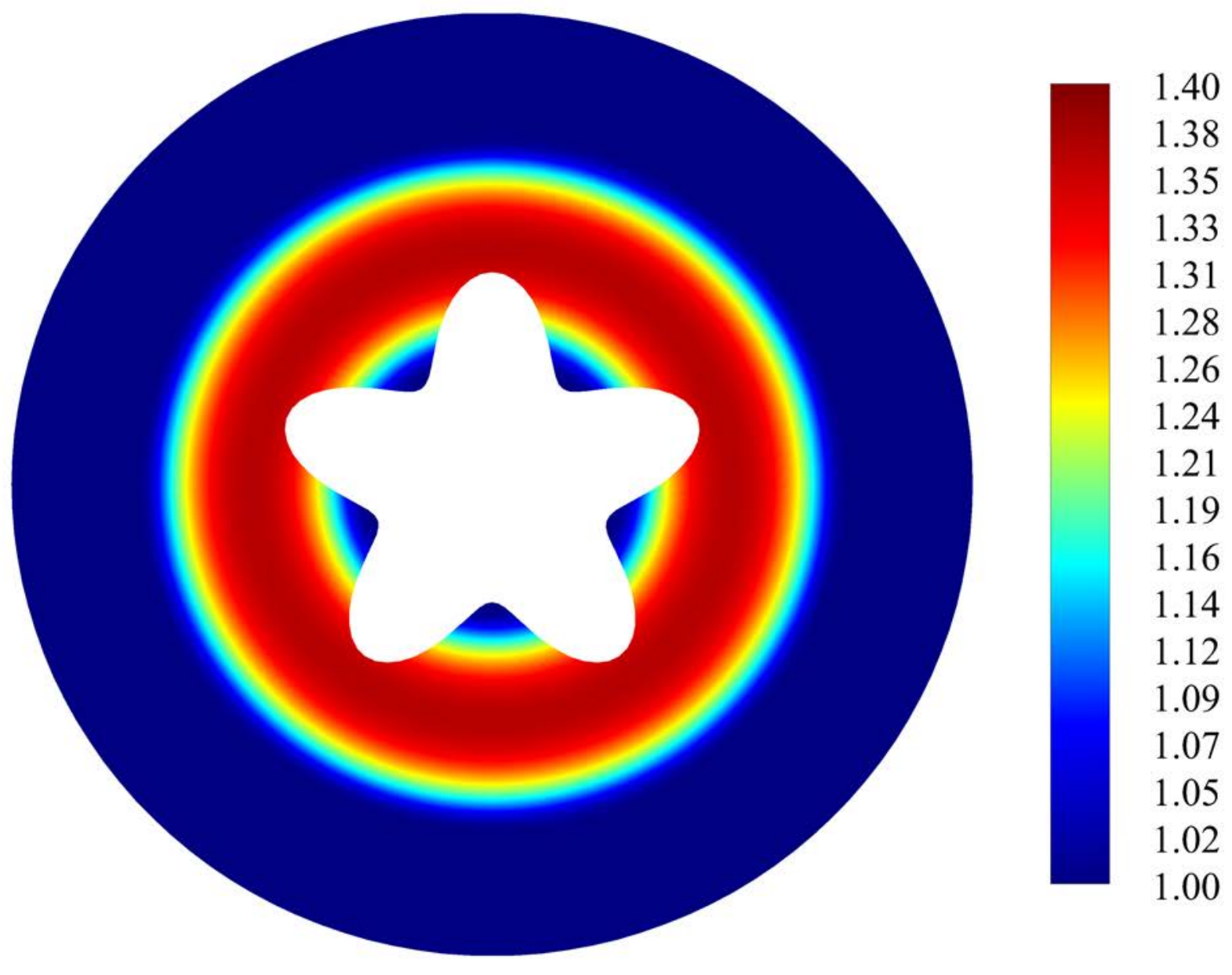}}
	\caption{Refraction indices of locally inhomogeneous media.}\label{nr}
\end{figure}
\begin{figure}[ht!]
	\centering
	\subfigure[$\kappa=10$]{\includegraphics[scale=0.27]{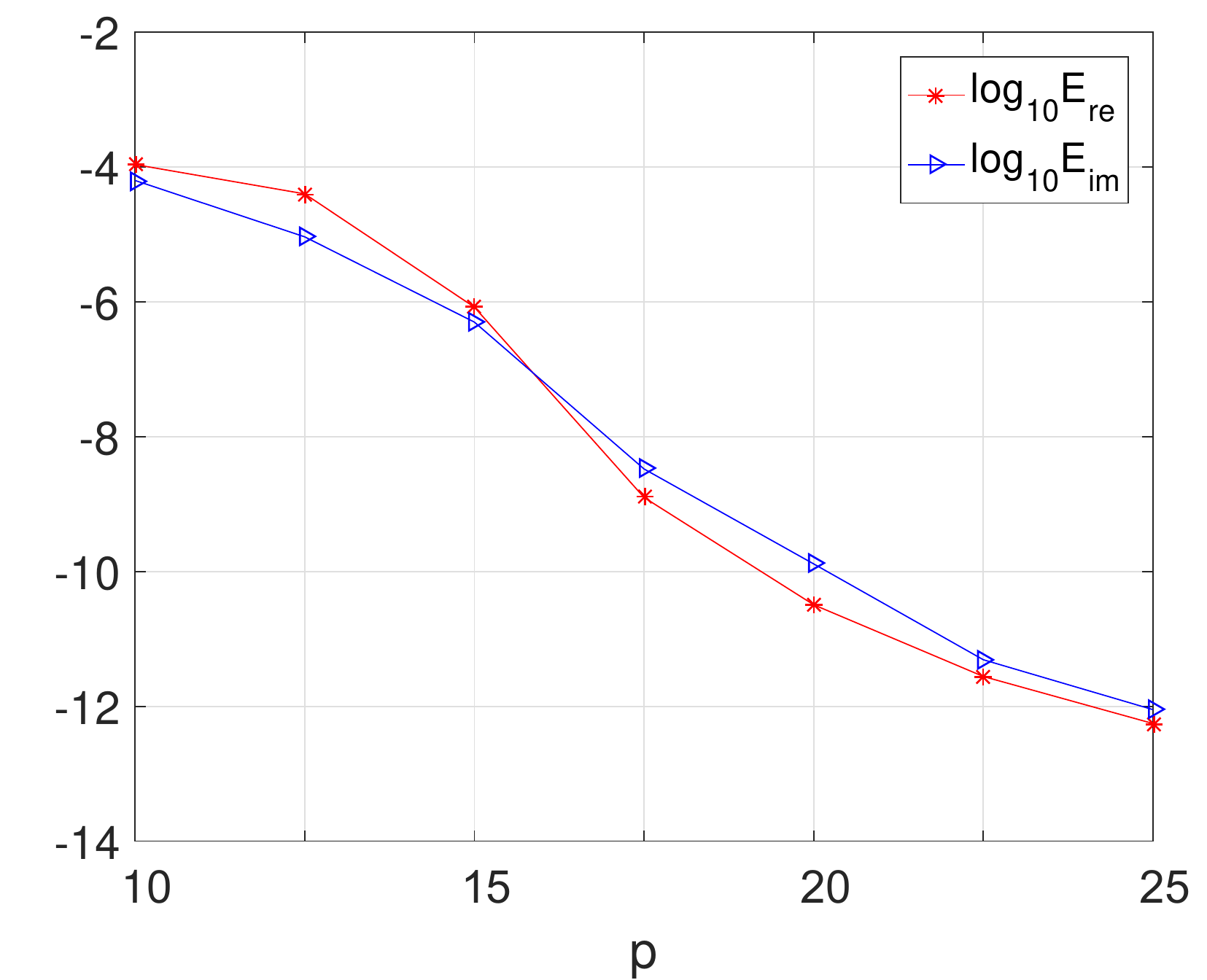}}
	\subfigure[$\kappa=20$]{\includegraphics[scale=0.27]{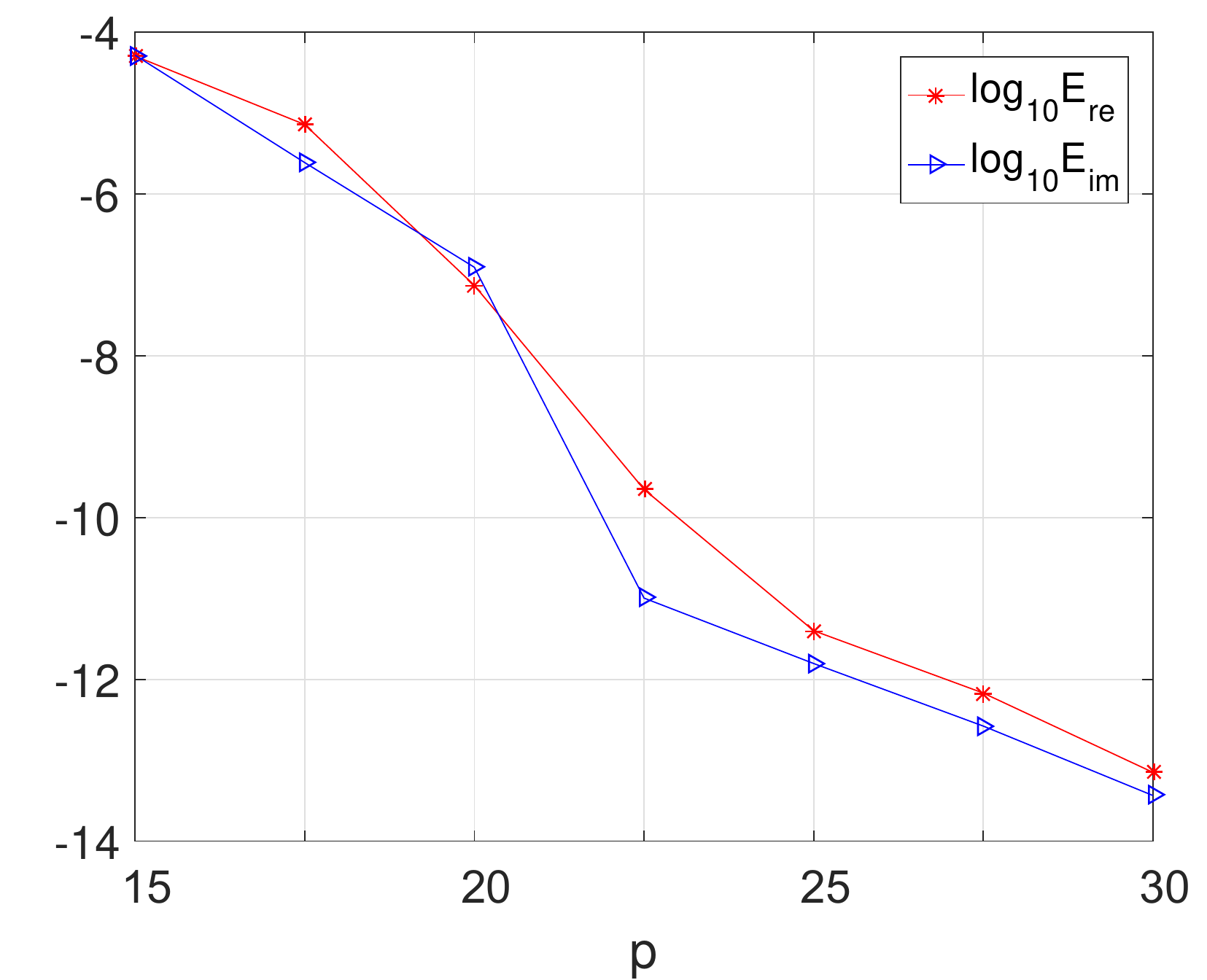}}
	\subfigure[$\kappa=30$]{\includegraphics[scale=0.27]{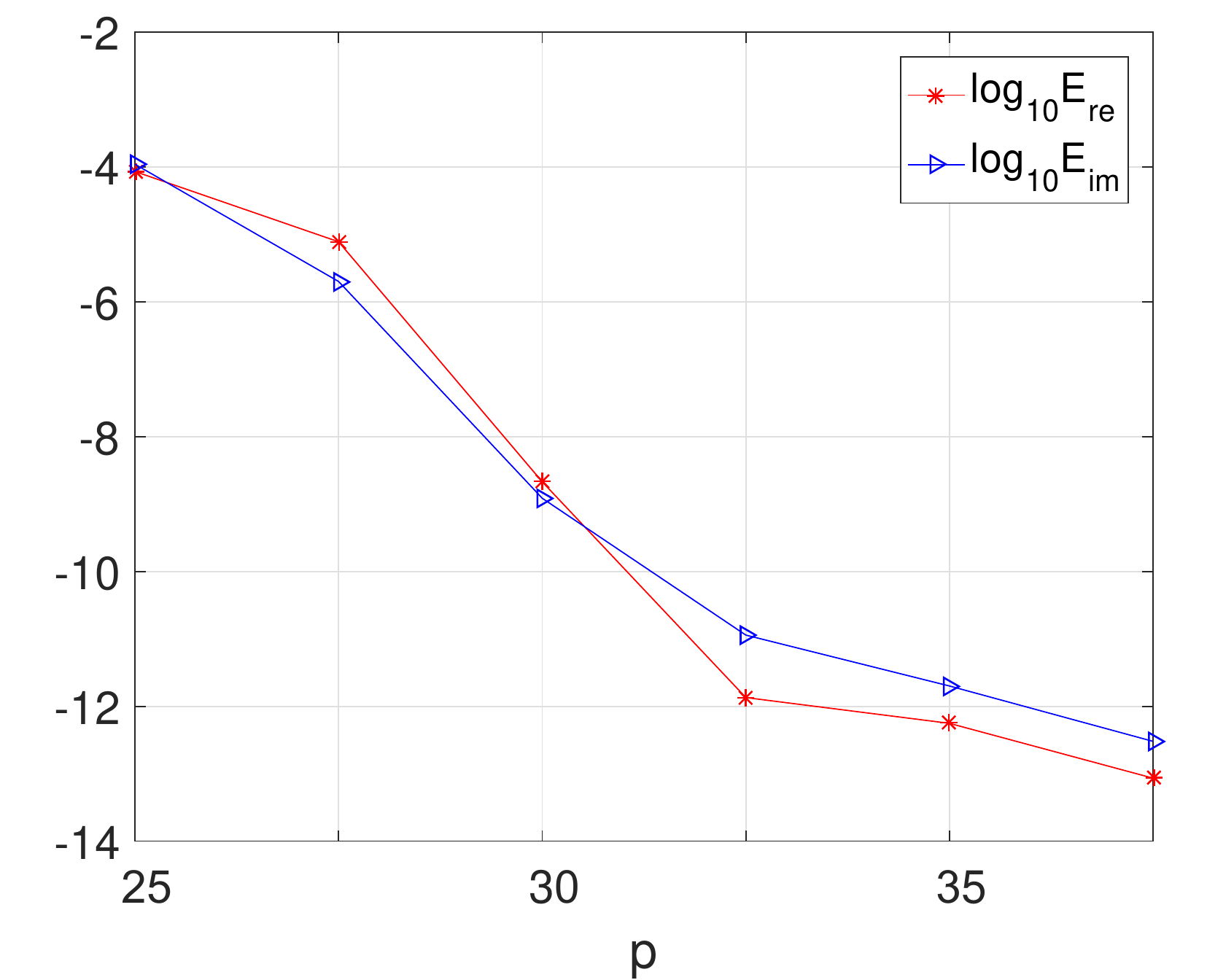}}
	\caption{$L^2$-errors against polynomial degree $p$ (inhomogeneous media).}\label{L2errorplotinhomo}
\end{figure}
\begin{figure}[ht!]
	\centering
	\subfigure[iterative solution]{\includegraphics[scale=0.25]{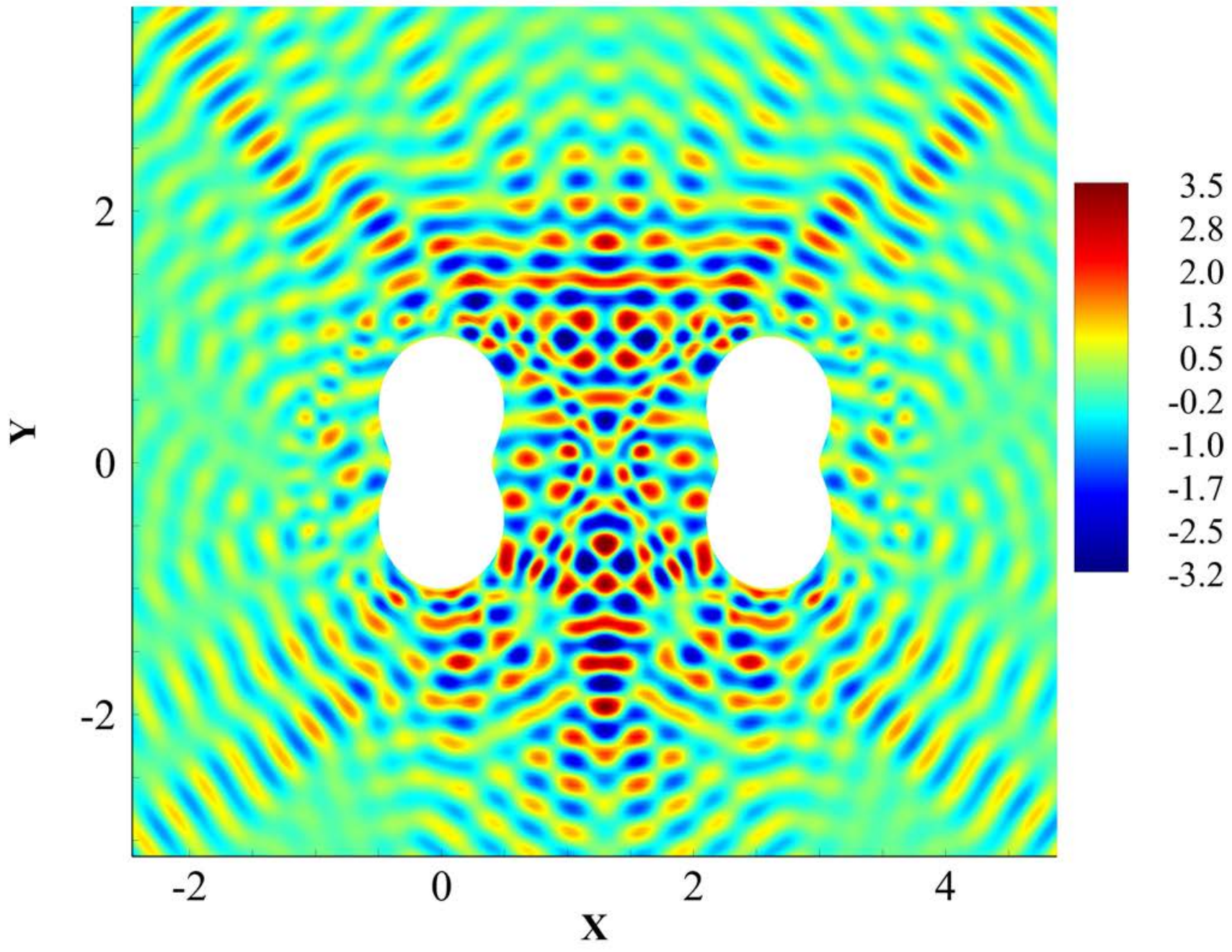}}\quad
	\subfigure[reference solution]{\includegraphics[scale=0.25]{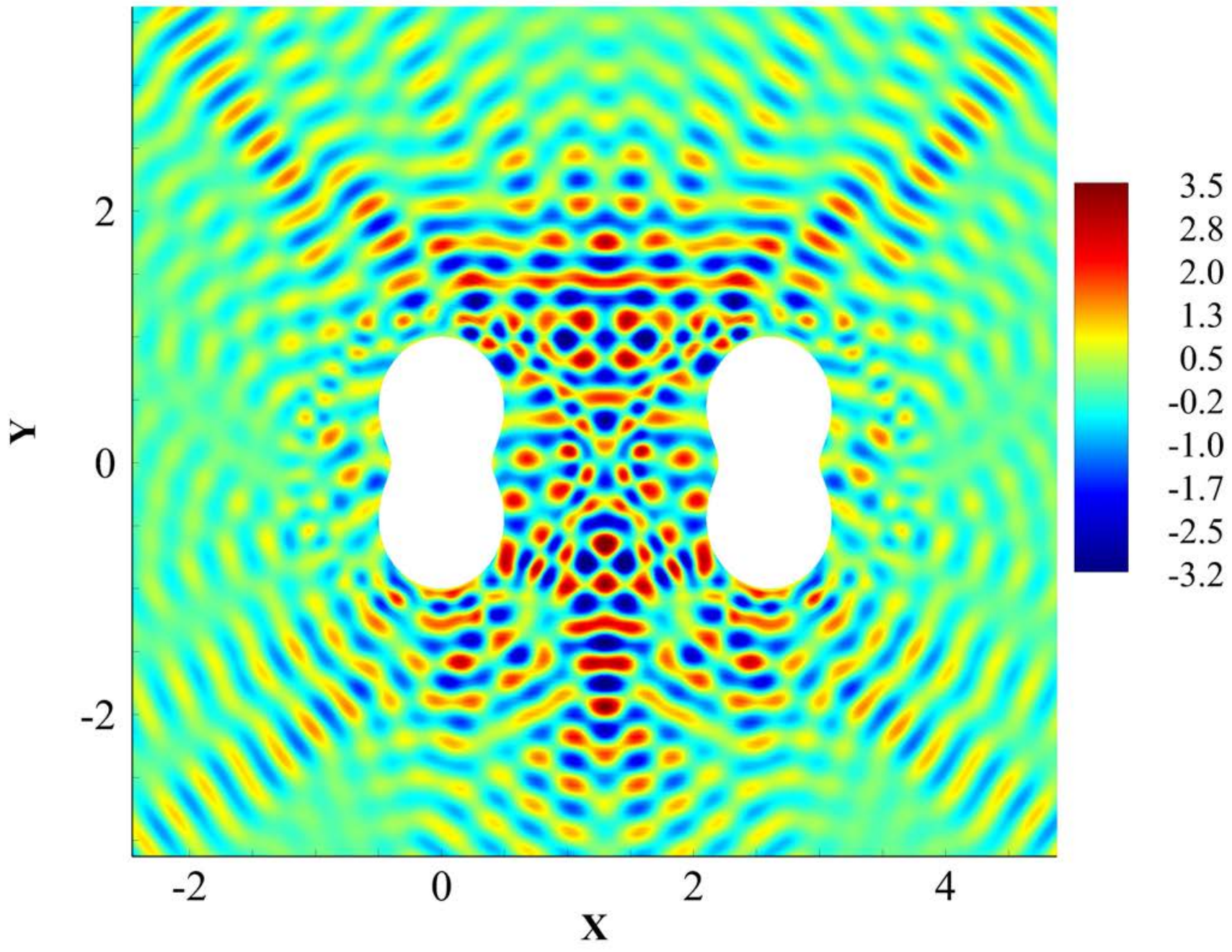}}\quad
	\subfigure[error]{\includegraphics[scale=0.25]{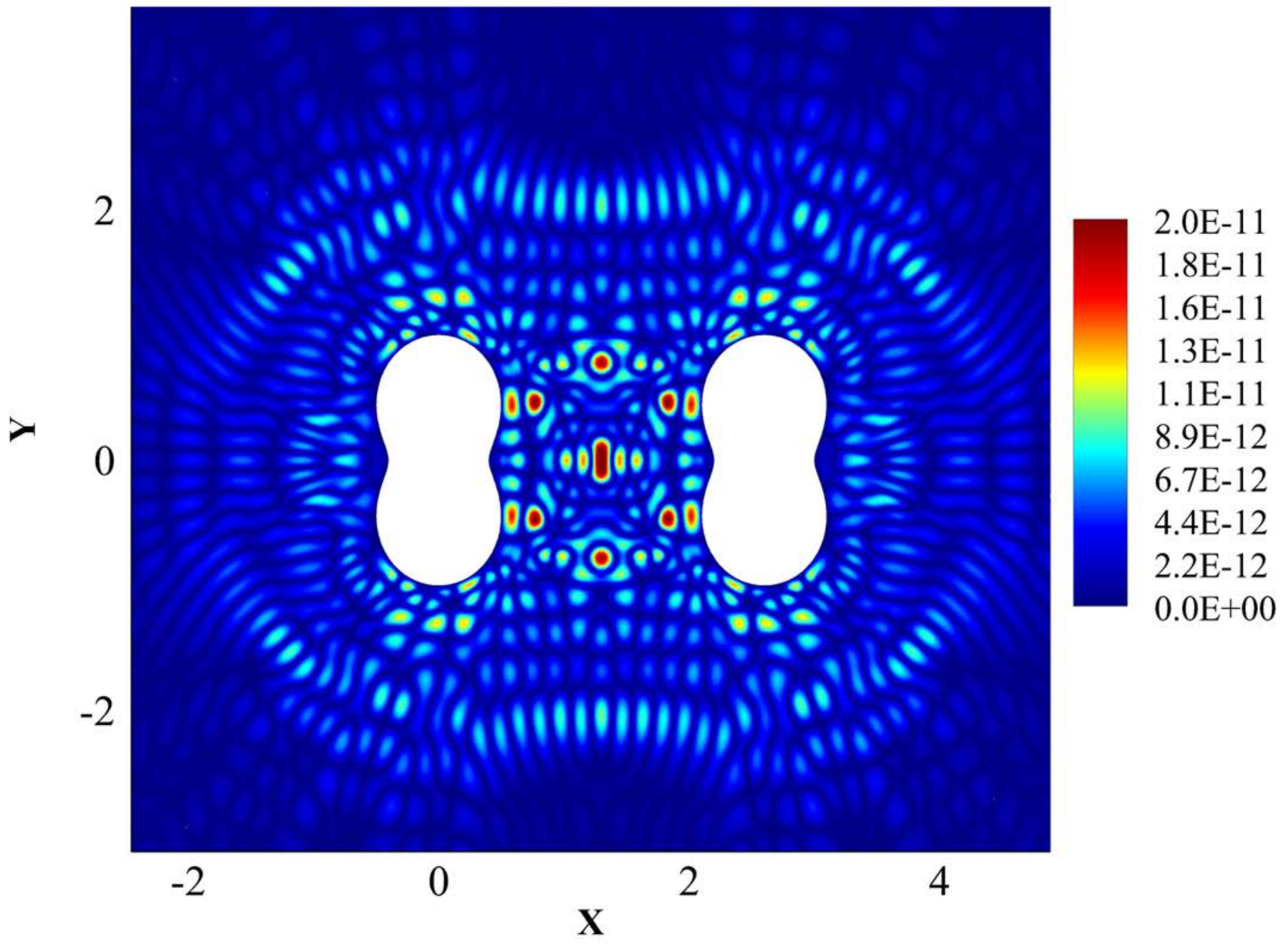}}
	\caption{Real parts of iterative numerical solution ($p=20$), reference solution and error for $\kappa=20$ and inhomogeneous refraction index given by \eqref{heterrefractionindex1}.}\label{scatteringwave2scatterersreINHOMO}
\end{figure}
\begin{figure}[ht!]
	\centering	
	\subfigure[$\kappa=10$]{\includegraphics[scale=0.27]{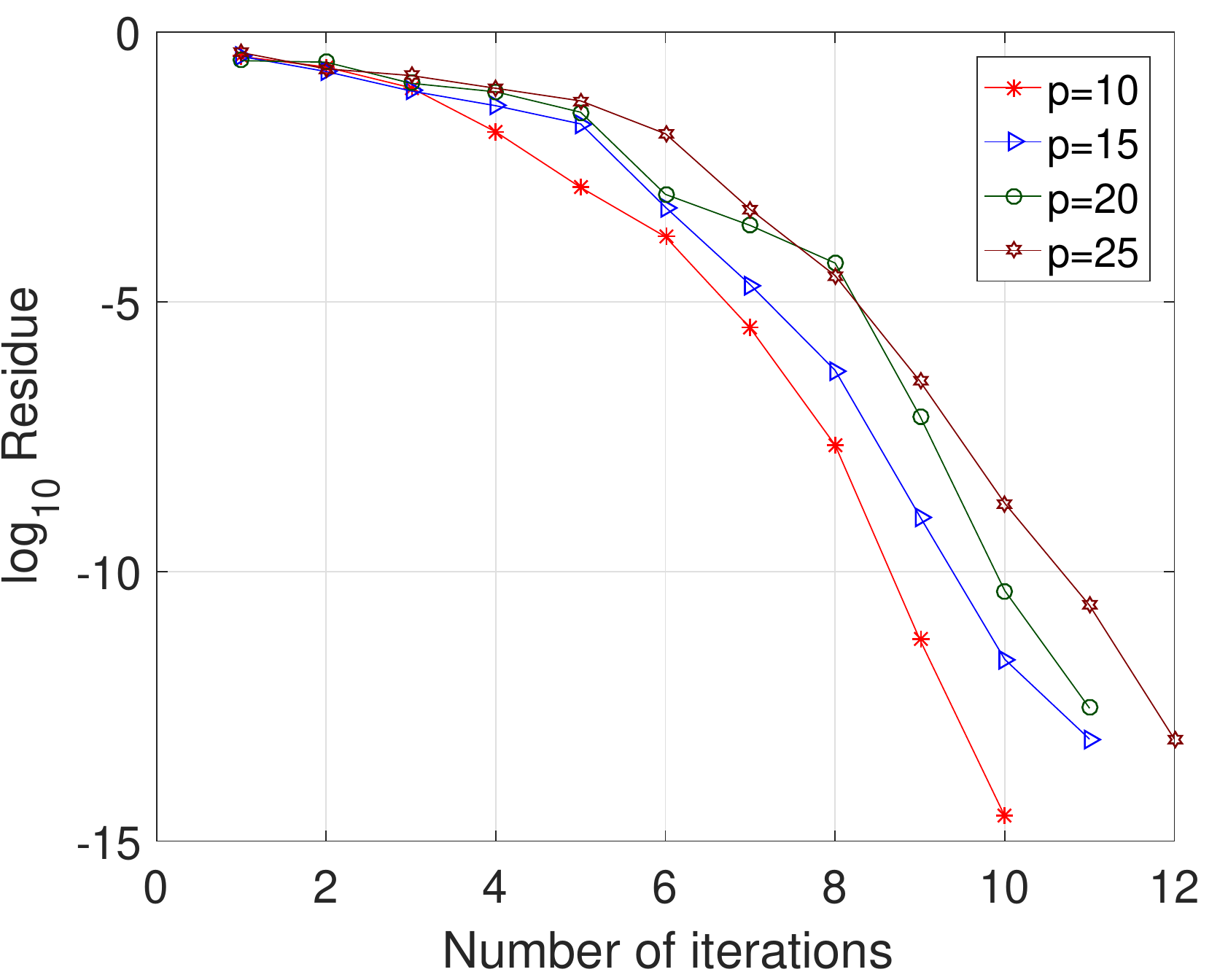}}
	\subfigure[$\kappa=20$]{\includegraphics[scale=0.27]{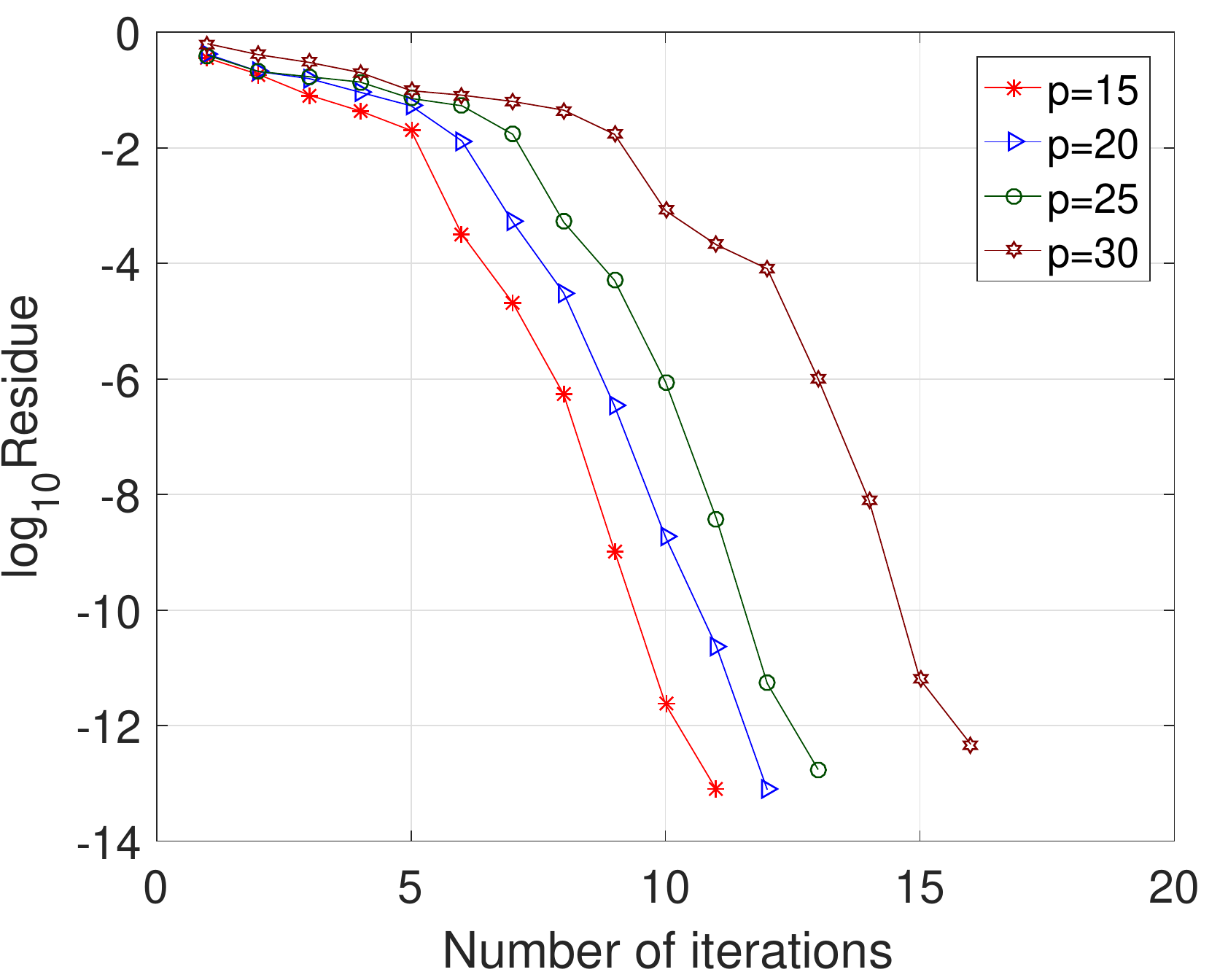}}
	\subfigure[$\kappa=30$]{\includegraphics[scale=0.27]{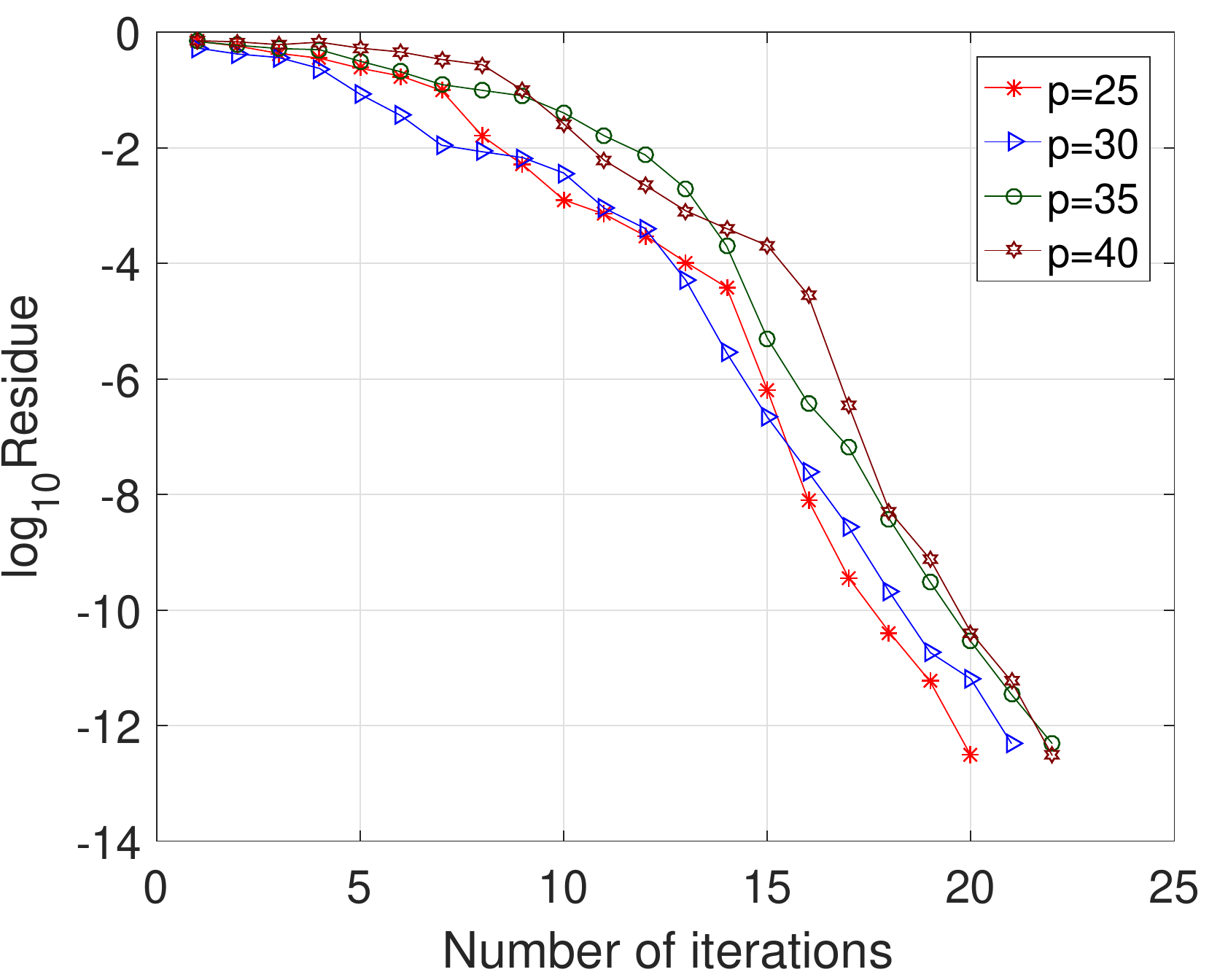}}
	\caption{Residuals against the number of iterations (inhomogeneous media).}\label{residueinhomo}
\end{figure}
\begin{table}[htbp]
	\center
	\begin{tabular}{|c|c|c|c|c|}
		\hline
		\multirow{2}{*}{$\kappa $} & \multirow{2}{*}{$p $} & \multicolumn{3}{c|}{number of iterations} \\\cline{3-5}
		& & GMRES for \eqref{blocksystem} & block GMRES for \eqref{blocksystem}& our iterative algorithm\\
		\hline
		\multirow{4}{*}{$10$}  & $10$ & $303$ & $80$ & $10$ \\\cline{2-5}
		& $15$ & $937$ & $112$ & $11$ \\\cline{2-5}
		& $20$ & $2744$ &$120$  & $11$ \\\cline{2-5}
		& $25$ & $3978$  & $126$ & $12$ \\
		\hline
		\multirow{2}{*}{$20$}  & $15$ & $742$ & $100$ & $11$  \\\cline{2-5}
		& $20$ & $1019$ & $153$ & $12$ \\\cline{2-5}
		& $25$ & $4021$ & $166$ & $14$ \\\cline{2-5}
		& $30$ & $5769$ & $284$ &   $16$\\
		\hline
	\end{tabular}
	\caption{The number of iterations  using different numerical methods for multiple scattering problems in locally inhomogeneous media.}\label{table2}
\end{table}

{\bf Example 6:} Set the refraction index
\begin{equation}\label{heterrefractionindex}
n(\bs x)=
\begin{cases}
x\exp(-1/(1-16(|\bs x-\bs c_i|-0.5)^2))+1, &0.25<|\bs x-\bs c_i|<0.75;\\[2pt]
1,& \mathrm{otherwise}.
\end{cases}
\end{equation}
The contour of $n(\bm{x})$ is plotted in Fig. \ref{nr} (b). Four scatterers determined by \eqref{scattererpara} with $k = 5, a=0.3, b=0.7, \theta_0=0$ and centers $(2.2n, 2.2m), n,m=0, 1$ are considered. The real part of the approximate scattering field is  plotted in Fig. \ref{multiscatterer4INHOMO} (a). Clearly, we can see stronger scattering in the region which has larger refraction index.
\begin{figure}[ht!]
	\centering
	\subfigure[refraction index given by\eqref{heterrefractionindex} ]{\includegraphics[scale=0.31]{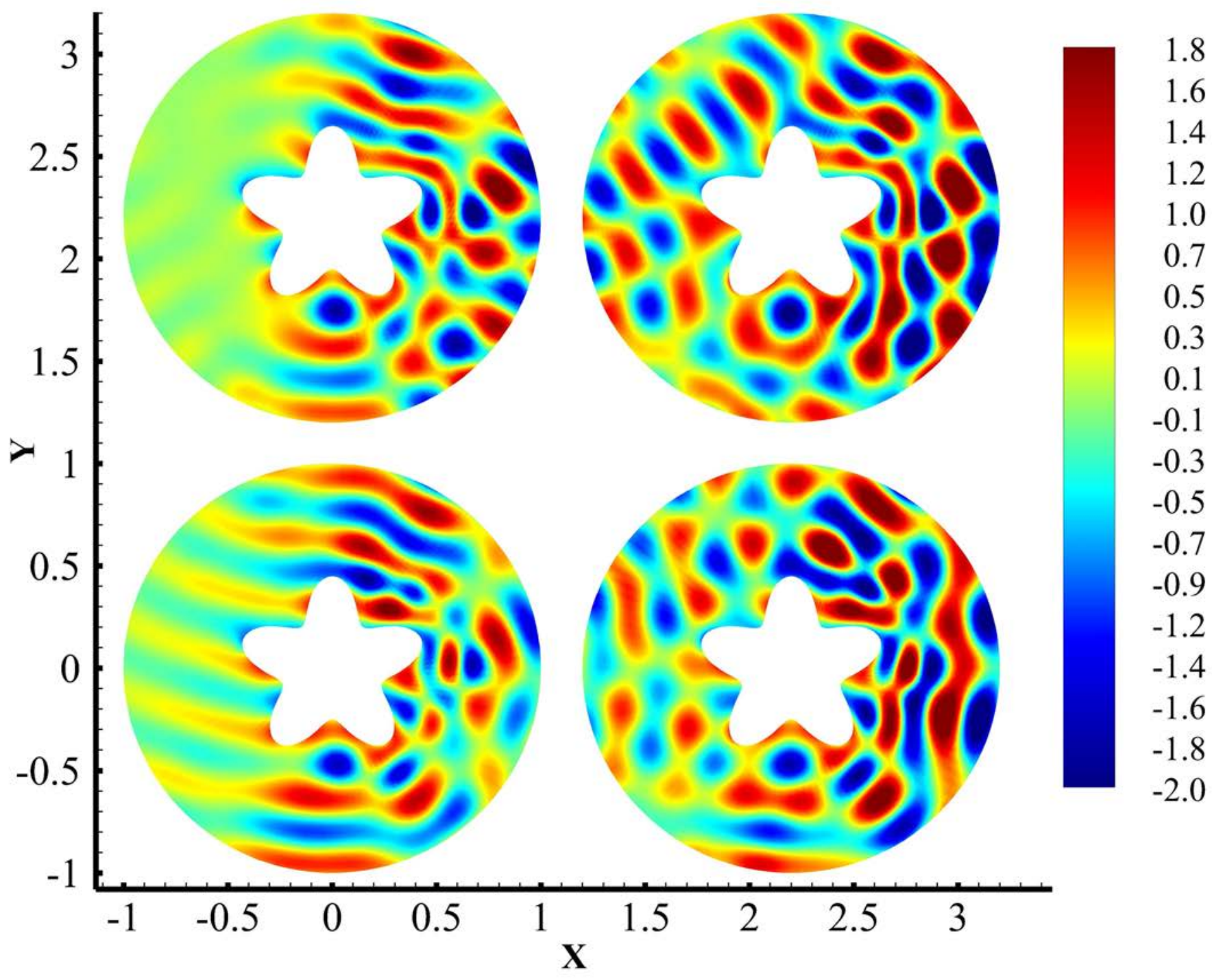}}\qquad
	\subfigure[refraction index given by \eqref{heterrefractionindex2}]{\includegraphics[scale=0.41]{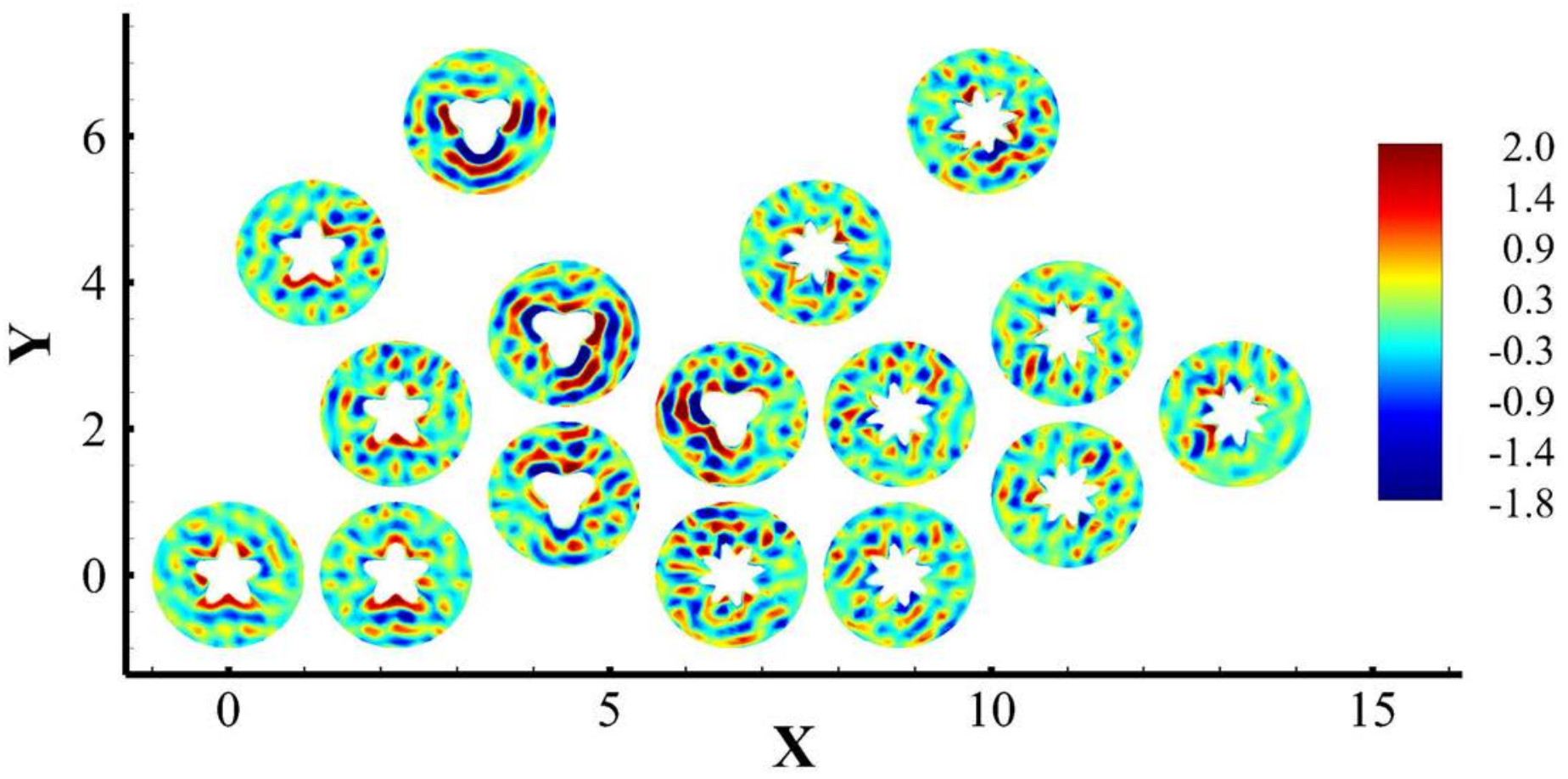}}
	\caption{Real parts of approximate scattering fields ($p=20$) due to $4$  and $16$ scatterers ($\kappa=20$).}\label{multiscatterer4INHOMO}
\end{figure}

\smallskip 
{\bf Example 7:} Consider the scattering problem with $16$ scatterers discussed in Example 3. Set the refraction index
\begin{equation}\label{heterrefractionindex2}
n(\bs x)=
\begin{cases}
{\rm exp}(-1/(1-16(|\bs x-\bs c_i|-0.5)^2))+1, &0.25<|\bs x-\bs c_i|<0.75;\\
1,& \mathrm{otherwise},
\end{cases}
\end{equation}
where the contour is plotted in Fig. \ref{nr} (c). The real part of the approximate scattering field is plotted in Fig. \ref{multiscatterer4INHOMO} (b).

\vskip 20pt
\noindent{\large\bf Conclusion and future work}
\vskip 10pt

In this paper, an efficient iterative method for the multiple scattering problem in locally inhomogeneous media is proposed and analyzed. This method is based on boundary integral equations on artificial boundaries. Thus, the iteration converges within a small number of iterations  which is nearly independent of the degree of freedom of  discretization. At each iteration, only the interior and  exterior problems (solved analytically for circular geometry) 
 with respect to single scatterer need to be solved individually. Therefore it has advantages in solving problems with a large number of scatterers. Moreover, it enjoys a great flexibility due to the capability of using various combinations of iterative algorithms and single scattering problem solvers.

For the future work, we will investigate the extension to penetrable scatterers and 3D multiple scattering problems. A preconditioned version for extremely large number of scatterers will also be considered.

\vskip 10pt
\noindent{\large\bf Acknowledgment}
\vskip 6pt
The research of the first and second author is  supported by NSFC (91430107, 11171104 and 11771138) and the Construct Program of the Key Discipline in Hunan Province. The research of the third author is support by NSFC (grant 11771137), the Construct Program of the Key Discipline in Hunan Province and a Scientific Research Fund of Hunan Provincial Education Department (No. 16B154). The research of the fourth author is  supported by Singapore MOE AcRF Tier 2 Grants (MOE2017-T2-2-144 and MOE2018-T2-1-059).

\bibliographystyle{elsarticle-num}

\end{document}